\documentclass[11pt]{article}
\usepackage[colorlinks=false]{hyperref}
\usepackage{amsfonts, amsmath, amssymb, amsthm, fullpage, enumitem, tikz, bm, mathtools, tabularx, xparse}
\usepackage[noabbrev, capitalize]{cleveref}

\definecolor{litegray}{RGB}{192,192,192}
\definecolor{red}{RGB}{234,67,53}
\definecolor{green}{RGB}{52,168,83}
\definecolor{blue}{RGB}{66,133,244}

\usetikzlibrary{decorations.pathreplacing, calc}
\tikzstyle{vertex}=[circle, draw, fill=litegray, inner sep=0pt, minimum width=4pt]
\tikzstyle{circ}=[circle, draw, fill=litegray, inner sep=1pt]
\tikzstyle{braket}=[decorate,decoration={brace,amplitude=10pt},xshift=0pt,yshift=-10pt, black]

\makeatletter
\NewDocumentCommand{\wlabel}{O{theoremext}m+m}{\begingroup
  \cref@constructprefix{#1}{\cref@result}%
  \protected@edef\@currentlabel{#3}%
  \protected@edef\@currentlabelname{#3}%
  \protected@edef\cref@currentlabel{[#1][][\cref@result]#3}
  \label[#1]{#2}\endgroup}
\makeatother

\title{Forbidden induced subgraphs for graphs and signed graphs\\with eigenvalues bounded from below}
\author{
    Zilin Jiang\thanks{School of Mathematical and Statistical Sciences, and School of Computing and Augmented Intelligence, Arizona State University, Tempe, AZ 85281, USA. Email: {\tt zilinj@asu.edu}. Supported in part by an AMS Simons Travel Grant, and by U.S.\ taxpayers through NSF grant DMS-2127650.}
    \and Alexandr Polyanskii\thanks{Department of Mathematics, Emory University, Atlanta, GA 30322, USA. Email: {\tt apolian@emory.edu}. Supported in part by U.S.\ taxpayers through NSF grant DMS-2349045.}
}
\date{}

\newtheorem{theorem}{Theorem}[section]
\newtheorem{corollary}[theorem]{Corollary}
\newtheorem{lemma}[theorem]{Lemma}
\newtheorem{proposition}[theorem]{Proposition}
\newtheorem{conjecture}[theorem]{Conjecture}
\newtheorem{problem}[theorem]{Problem}

\theoremstyle{definition}
\newtheorem{definition}[theorem]{Definition}

\theoremstyle{remark}
\newtheorem*{remark}{Remark}
\newtheorem*{claim*}{Claim}

\newenvironment{claimproof}[1][Proof]{\begin{proof}[#1]}{\end{proof}}

\DeclareMathOperator{\rank}{rank}
\DeclarePairedDelimiter\abs{\lvert}{\rvert}

\newcommand{\dset}[2]{\left\{{#1}\colon{#2}\right\}}
\newcommand{\sset}[1]{\left\{{#1}\right\}}

\newcommand{\al}{\alpha}
\newcommand{\be}{\beta}
\newcommand{\eps}{\varepsilon}
\newcommand{\la}{\lambda}
\newcommand{\La}{\Lambda}

\newcommand{\kpl}{k_p(\la)}

\newcommand{\D}{\mathcal{D}}
\newcommand{\F}{\mathcal{F}}
\newcommand{\G}{\mathcal{G}}
\newcommand{\HH}{\mathcal{H}}

\newcommand{\N}{\mathbb{N}}
\newcommand{\R}{\mathbb{R}}
\newcommand{\X}{\mathcal{X}}

\newcommand{\Gp}{\G'(\la)}
\newcommand{\Gl}{\G(\la)}
\newcommand{\Gs}{\G^\pm(\la)}
\newcommand{\Gsp}{\G^\mp(\la)}
\newcommand{\Nabd}{N_{\al,\be}(d)}

\newcommand{\mult}{\operatorname{mult}}

\newcommand{\ba}{\bm{a}}
\newcommand{\tba}{\tilde{\ba}}
\newcommand{\bb}{\bm{b}}
\newcommand{\bx}{\bm{x}}
\newcommand{\tx}{\tilde{x}}
\newcommand{\tbx}{\tilde{\bx}}

\begin{document}

\maketitle

\begin{abstract}
	The smallest eigenvalue of a graph is the smallest eigenvalue of its adjacency matrix. We show that the family of graphs with smallest eigenvalue at least $-\lambda$ can be defined by a finite set of forbidden induced subgraphs if and only if $\lambda < \lambda^*$, where $\lambda^* = \rho^{1/2} + \rho^{-1/2} \approx 2.01980$, and $\rho$ is the unique real root of $x^3 = x + 1$. This resolves a question raised by Bussemaker and Neumaier. As a byproduct, we find all the limit points of smallest eigenvalues of graphs, supplementing Hoffman's work on those limit points in $[-2, \infty)$.
    
    We also prove that the same conclusion about forbidden subgraph characterization holds for signed graphs. Our impetus for the study of signed graphs is to determine the maximum cardinality of a spherical two-distance set with two fixed angles (one acute and one obtuse) in high dimensions. Denote by $N_{\alpha, \beta}(d)$ the maximum number of unit vectors in $\mathbb{R}^d$ where all pairwise inner products lie in $\{\alpha, \beta\}$ with $-1 \le \beta < 0 \le \alpha < 1$. Very recently Jiang, Tidor, Yao, Zhang and Zhao determined the limit of $N_{\alpha, \beta}(d)/d$ as $d\to\infty$ when $\alpha + 2\beta < 0$ or $(1-\alpha)/(\alpha-\beta) \in \{1,\sqrt2,\sqrt3\}$, and they proposed a conjecture on the limit in terms of eigenvalue multiplicities of signed graphs. We establish their conjecture whenever $(1-\alpha)/(\alpha - \beta) < \lambda^*$.
\end{abstract}

\section{Introduction} \label{sec:intro}

A fundamental problem in spectral graph theory is the classification and characterization of graphs with bounded eigenvalues. When we talk about eigenvalues of a graph we always refer to its adjacency matrix. In this paper, we study the families of graphs with eigenvalues bounded from below. Let $\Gl$ be the family of graphs with smallest eigenvalue at least $-\la$. For the sake of comparison, we mention the family $\Gp$ of graphs with spectral radius (or, equivalently, largest eigenvalue) at most $\la$.

\begin{remark}
    Since we rarely work with subgraphs that are not induced, all subgraphs are induced throughout this paper. We refer to subgraphs that are not necessarily induced as \emph{general subgraphs}.
\end{remark}

The Cauchy interlacing theorem implies that both $\Gp$ and $\Gl$ are closed under taking subgraphs. It is a natural question to ask whether it is possible to define each of these families by a finite set of forbidden subgraphs.

\begin{definition}
    Given a family $\G$ of graphs that is closed under taking subgraphs, a family $\F$ of graphs is a \emph{forbidden subgraph characterization} of $\G$ if the family $\G$ consists exactly of graphs that do not contain any member of $\F$ as a subgraph, and a graph $F$ is a \emph{minimal forbidden subgraph} for $\G$ if $F$ itself is not in $\G$ but every proper subgraph of $F$ is in $\G$.
\end{definition}

Note that the most economical forbidden subgraph characterization of $\G$ consists precisely of the minimal forbidden subgraphs for $\G$. Thus the existence of a finite forbidden subgraph characterization of $\G$ is equivalent to the finiteness of the minimal forbidden subgraphs for $\G$. In 1992 Bussemaker and Neumaier made the following remark in \cite[p~599]{BN92}:

\begin{quote}
    It would be interesting to know the set of numbers $m$, $-m$ such that $\mathcal{G}_m^\#$ [the set of minimal forbidden subgraphs for $\G'(m)$] or $\mathcal{G}_{-m}^\#$ [the set of minimal forbidden subgraphs for $\G(m)$] are finite; however, these seem to be very difficult problems.
\end{quote}

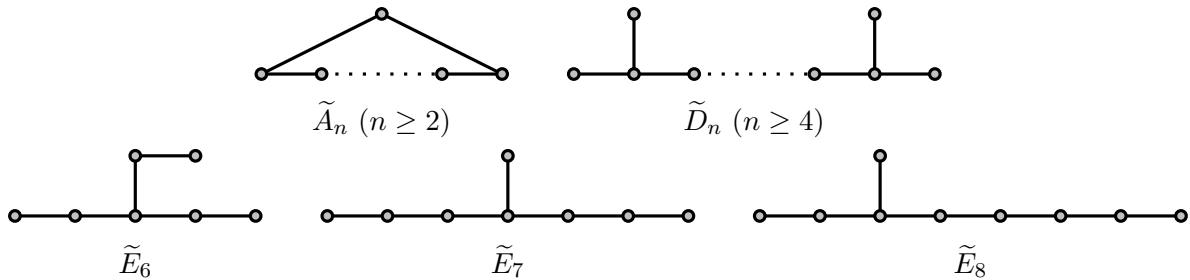
\begin{figure}
    \centering
    \begin{tikzpicture}[very thick, scale=0.4, baseline=(v.base)]
        \coordinate (v) at (0,0);
        \draw[black] (0,0) -- (4,2) node[vertex]{} -- (8,0);
        \draw[black, loosely dotted] (2,0) -- (6,0);
        \draw[black] (0,0) node[vertex]{} -- (2,0) node[vertex]{};
        \draw[black] (6,0) node[vertex]{} -- (8,0) node[vertex]{};
        \node at (4,-1.5) {$\widetilde{A}_n$ ($n \ge 2$)};
    \end{tikzpicture}\qquad
    \begin{tikzpicture}[very thick, scale=0.4, baseline=(v.base)]
        \coordinate (v) at (0,0);
        \draw[black] (2,0) -- (2,2) node[vertex]{};
        \draw[black] (10,0) -- (10,2) node[vertex]{};
        \draw[black, loosely dotted] (4,0) -- (8,0);
        \draw[black] (0,0) node[vertex]{} -- (2,0) node[vertex]{} -- (4,0) node[vertex]{};
        \draw[black] (8,0) node[vertex]{} -- (10,0) node[vertex]{} -- (12,0) node[vertex]{};
        \node at (6,-1.5) {$\widetilde{D}_n$ ($n \ge 4$)};
    \end{tikzpicture}\\
    \begin{tikzpicture}[very thick, scale=0.4, baseline=(v.base)]
        \coordinate (v) at (0,0);
        \draw[black] (4,0) -- (4,2) node[vertex]{} -- (6,2) node[vertex]{};
        \draw[black] (0,0) node[vertex]{} -- (2,0) node[vertex]{} -- (4,0) node[vertex]{} -- (6,0) node[vertex]{} -- (8,0) node[vertex]{};
        \node at (4,-1.5) {$\widetilde{E}_6$};
    \end{tikzpicture}\qquad
    \begin{tikzpicture}[very thick, scale=0.4, baseline=(v.base)]
        \coordinate (v) at (0,0);
        \draw[black] (6,0) -- (6,2) node[vertex]{};
        \draw[black] (0,0) node[vertex]{} -- (2,0) node[vertex]{} -- (4,0) node[vertex]{} -- (6,0) node[vertex]{} -- (8,0) node[vertex]{} -- (10,0) node[vertex]{} -- (12,0) node[vertex]{};
        \node at (6,-1.5) {$\widetilde{E}_7$};
    \end{tikzpicture}\qquad
    \begin{tikzpicture}[very thick, scale=0.4, baseline=(v.base)]
        \coordinate (v) at (0,0);
        \draw[black] (4,0) -- (4,2) node[vertex]{};
        \draw[black] (0,0) node[vertex]{} -- (2,0) node[vertex]{} -- (4,0) node[vertex]{} -- (6,0) node[vertex]{} -- (8,0) node[vertex]{} -- (10,0) node[vertex]{} -- (12,0) node[vertex]{} -- (14,0) node[vertex]{};
        \node at (7,-1.5) {$\widetilde{E}_8$};
    \end{tikzpicture}
    \caption{Maximal connected graphs with spectral radius at most $2$. The number of vertices is one more than the given index. In particular, $\widetilde{D}_4$ is actually a star with four leaves.} \label{fig:dynkin}
\end{figure}

The specific families $\G'(2)$ and $\G(2)$ are well understood. One of the earliest results dates back to 1970 when Smith~\cite{S70} determined all the connected graphs in $\G'(2)$ --- they are general subgraphs of the extended Dynkin diagrams in \cref{fig:dynkin}. The family $\G(2)$ is much richer and more complex --- it contains not only all the graphs in $\G'(2)$, but also all the line graphs.\footnote{A line graph $L(H)$ of a graph $H$ is obtained by creating a vertex per edge in $H$, and connecting two vertices if and only if the corresponding edges in $H$ have a vertex in common. The adjacency matrix of $L(H)$ can be written as $B^TB - 2I$, where $B$ is the vertex-edge incidence matrix of $H$, hence the smallest eigenvalue of $L(H)$ is at least $-2$.\label{footnote:line-graph}} The classification of $\G(2)$ culminated in a beautiful theorem of Cameron, Goethals, Seidel, and Shult~\cite{CGSS76} who related $\G(2)$ to root systems that occur in the classification of semisimple Lie algebras (see \cref{thm:cgss}). We refer the reader to the monograph \cite{CRS04} for a comprehensive account of $\G(2)$.

These classification theorems can be used to establish quantitative answers to Bussemaker and Neumaier's problems. For $\G'(2)$, Cvetkovi\'{c}, Doob and Gutman~\cite[Theorem~2.8]{CDG82} determined that there are $18$ minimal forbidden subgraphs. For $\G(2)$, Rao, Singhi and Vijayan~\cite[Theorem~4.1]{RSV81} observed that the number of vertices in any minimal forbidden subgraph is at most $37$, which was eventually perfected to $10$ by Kumar, Rao and Singhi~\cite{KRS82}. A computer search by Bussemaker and Neumaier~\cite[p~596]{BN92} established that there are, in total, $1812$ minimal forbidden subgraphs for $\G(2)$.

In a recent work, the authors of the current paper resolved the first problem of Bussemaker and Neumaier.\footnote{The original statement of \cite[Theorem~1]{JP20} determines the set of $\la$ for which the subfamily of connected graphs in $\Gp$ (or equivalently, the entire family $\G'(\la)$) can be defined by a finite set of forbidden \emph{general} subgraphs. Using the well-known fact that the spectral radius of $G_1$ is at most that of $G_2$ whenever $G_1$ is a general subgraph of $G_2$, one can show that if $\G$ is a forbidden general subgraph characterization for $\Gp$, then the family $\dset{H}{\exists\, G \in \G \text{ s.t. }G\text{ is a general subgraph of }H\text{ on the same vertex set}}$ is a forbidden (induced) subgraph characterization for $\Gp$, and so the original statement of \cite[Theorem~1]{JP20} is equivalent to \cref{thm:forb1}.}

\begin{theorem}[Theorem~1 of Jiang and Polyanskii~\cite{JP20}] \label{thm:forb1}
    For every integer $m \ge 2$, let $\be_m$ be the largest root of $x^{m+1} = 1 + x + \dots + x^{m-1}$, and let $\al_m := \be_m^{1/2} + \be_m^{-1/2}$. The family $\Gp$ of graphs with spectral radius at most $\la$ has a finite forbidden subgraph characterization if and only if $\la < \la'$ and $\la \not\in\sset{\al_2, \al_3, \dots}$, where
    \[
        \la' := \lim_{m\to\infty} \al_m = \varphi^{1/2} + \varphi^{-1/2} = \sqrt{2 + \sqrt5} \approx 2.05817,
    \]
    and $\varphi$ is the golden ratio $(1+\sqrt5)/2$.
\end{theorem}

The first main theorem of this paper resolves the remaining problem of Bussemaker and Neumaier --- $\Gl$ enjoys a straightforward threshold phenomenon.

\begin{theorem} \label{thm:main1}
    The family $\Gl$ of graphs with smallest eigenvalue at least $-\la$ has a finite forbidden subgraph characterization if and only if $\la < \la^*$, where
    \[
        \la^* := \rho^{1/2} + \rho^{-1/2} \approx 2.01980,
    \]
    and $\rho$ is the unique real root of $x^3 = x + 1$.
\end{theorem}

\begin{remark}
    The constants defined in \cref{thm:forb1,thm:main1} satisfy $2 < \la^* = \al_2 < \al_3 < \dots < \la'$. The constant $\rho$ defined in \cref{thm:main1}, known as the \emph{plastic number}, is the smallest Pisot--Vijayaraghavan number.\footnote{The exact value of $\rho$ is $\sqrt[3]{(9+\sqrt{69})/18} + \sqrt[3]{(9-\sqrt{69})/18}$.}
\end{remark}

As a byproduct, we determine all the limit points of the set of smallest eigenvalues of graphs. Let $\La_1$ consist of $\la \in \R$ such that $-\la$ is the smallest eigenvalue of some graph. Here we invert the smallest eigenvalues of graphs similarly to how we define $\Gl$. Before our work, Hoffman~\cite{H77a} characterized all the limit points of $\La_1$ in $(-\infty, 2]$. His result involves a technical qualification, which was conjectured to be always true, and was later established by Greaves et al.~\cite[Theorem~1]{GKMST15}. Doob~\cite[Theorem~9]{D89} observed that every real number in $\sset{\al_2, \al_3, \al_4, \dots} \cup [\lambda',\infty)$ is a limit point of $\La_1$, and conjectured that $\alpha_2$ (which equals $\la^*$), $\al_3, \al_4 \dots$ are the only limit points of $\La_1$ in $(2, \la')$. We refute this conjecture by finding all the limit points of $\La_1$ in $(2, \la')$.

\begin{corollary} \label{cor:main1}
    For every $\la > 2$, the negative number $-\la$ is a limit point of the set of smallest eigenvalues of graphs if and only if $\la \ge \la^*$.
\end{corollary}

We next turn our attention to \emph{signed graphs}, which are graphs whose edges are each labeled by $+$ or $-$. Throughout the paper we decorate variables for signed graphs with the $\pm$ superscript. When we talk about eigenvalues of a signed graph $G^\pm$ on $n$ vertices, we refer to its \emph{signed adjacency matrix} --- the $n \times n$ matrix whose $(i,j)$-th entry is $1$ if $ij$ is a positive edge, $-1$ if $ij$ is a negative edge, and $0$ otherwise.

It still makes sense to speak of forbidden subgraph characterization of a family of signed graphs with eigenvalues bounded from below. Our second main theorem establishes the same threshold phenomenon for the family of signed graphs with smallest eigenvalue at least $-\la$.

\begin{theorem} \label{thm:main2}
    The family $\Gs$ of signed graphs with smallest eigenvalue at least $-\la$ has a finite forbidden subgraph characterization if and only if $\la < \la^*$.
\end{theorem}

Notice that if $\F^\pm$ is a finite forbidden subgraph characterization of $\Gs$, then the set of all-positive signed graphs in $\F^\pm$ is a finite forbidden subgraph characterization of $\Gl$. Therefore, for $\la < \la^*$, \cref{thm:main2} implies \cref{thm:main1}. However, in \cref{sec:forb}, we still provide the complete proof of \cref{thm:main1} to illustrate the key ideas.

Finally, we turn our attention to the largest eigenvalue of a signed graph. We denote the eigenvalues of a signed graph $G^\pm$ by $\la_1(G^\pm) \le \la_2(G^\pm) \le \dots $ in ascending order, and by $\la^1(G^\pm) \ge \la^2(G^\pm) \ge \dots$ in descending order. Observing that for every signed graph $G^\pm$, $\la_1(G^\pm) \ge -\la$ if and only if $\la^1(-G^\pm) \le \la$, where $-G^\pm$ reverses all the edge signs of $G^\pm$, we obtain an immediate consequence of \cref{thm:main2}.

\begin{corollary} \label{cor:main2}
    The family $\Gsp$ of signed graphs with largest eigenvalue at most $\la$ has a finite forbidden subgraph characterization if and only if $\la < \la^*$. \qed
\end{corollary}

Our motivation to understand the forbidden subgraph characterization of $\Gsp$ comes from the problem of determining the maximum size of a \emph{spherical two-distance set} with two fixed angles (one acute and one obtuse) in high dimensions. For fixed $-1 \le \be < 0 \le \al < 1$, let $\Nabd$ denote the maximum number of unit vectors in $\R^d$ where all pairwise inner products lie in $\sset{\al,\be}$.

In the special case $\al = -\be$, which corresponds to \emph{equiangular lines}, recent work \cite{B16,BDSK18,JP20} culminated in a solution \cite{JTYZZ19} of Jiang, Tidor, Yao, Zhang and Zhao to the problem of determining $N_{\al, -\al}(d)$ for sufficiently large $d$. We refer the reader to \cite[Section~1]{JTYZZ19} for earlier developments on equiangular lines with fixed angles in high dimensions.

For the general case $-1 \le \be < 0 \le \al < 1$, in their subsequent work \cite{JTYZZ23}, Jiang et al.\ proposed a conjecture on the limit of $\Nabd/d$ as $d\to\infty$. To state their conjecture, we need the following spectral graph theoretic quantity.

\begin{definition}
    Given $\la > 0$ and $p \in \N$, define the quantity
    \[
        \kpl = \inf\dset{\frac{\abs{G^\pm}}{\mult(\la, G^\pm)}}{\chi(G^\pm) \le p \text{ and }\la^1(G^\pm) = \la},
    \]
    where $\abs{G^\pm}$ is the number of vertices of $G^\pm$, $\mult(\la, G^\pm)$ is the multiplicity of $\la$ as an eigenvalue of $G^\pm$, and $\chi(G^\pm)$ is the chromatic number of the signed graph $G^\pm$.
\end{definition}

We postpone the definition of $\chi(G^\pm)$, which takes values in $\N^+ \cup \sset{\infty}$, to \cref{sec:app} (see \cref{def:chromatic}). We now state the conjecture on $\Nabd$.

\begin{conjecture}[Conjecture~1.11 of Jiang et al.~\cite{JTYZZ23}] \label{conj:main}
    Fix $-1 \le \be < 0 \le \al < 1$. Set $\la = (1-\al)/(\al - \be)$ and $p = \lfloor -\al/\be \rfloor + 1$. Then
    \[
        \Nabd = \begin{dcases}
            \frac{\kpl d}{\kpl-1} + o(d) & \text{if }\kpl < \infty, \\
            d + o(d) & \text{otherwise}.
        \end{dcases}
    \]
\end{conjecture}

\cref{conj:main} was confirmed in \cite{JTYZZ23} for $p \le 2$ and for $\la \in \sset{1,\sqrt2, \sqrt3}$ separately. Building on the framework developed there, for all $\la < \la^*$, we establish \cref{conj:main} as an application of \cref{cor:main2}, and we reduce the error term $o(d)$ to a constant depending only on $\al$ and $\be$.

The rest of the paper is organized as follows. We prove \cref{thm:main1,cor:main1} in \cref{sec:forb}, and we prove \cref{thm:main2} in \cref{sec:forb-signed}. Part of the proofs are computer-assisted with validated numerics. In \cref{sec:numerics} we explain our computer-aided proof and how anyone can recreate it independently. In \cref{sec:app} we give an application of \cref{thm:main2} to spherical two-distance sets. In \cref{sec:discussion}, we discuss open problems related to the classification of graphs in $\G(\la^*) \setminus \G(2)$, and a possible way to establish more instances of \cref{conj:main}.

\section{Forbidden subgraphs for \texorpdfstring{$\Gl$}{graphs with eigenvalues bounded from below}} \label{sec:forb}

We break the proof of \cref{thm:main1} into three cases $\la < 2$, $\la \in [2, \la^*)$ and $\la \ge \la^*$, where $\la^* \approx 2.01980$ is defined as in \cref{thm:main1}. It is worth pointing out the spectral graph theoretic interpretation of the peculiar constant $\la^*$.

\begin{proposition}[{\cite[Lemma~5(e)]{JP20}}] \label{lem:e2n}
    For every $n \in \N^+$, define the graph $E_{2,n}$ as in \cref{fig:e2n}. As $n \to \infty$, $\la^1(E_{2,n})$ increases to $\la^*$, or equivalently $\la_1(E_{2,n})$ decreases to $-\la^*$.\footnote{The equivalence is due to the fact that the spectrum of a bipartite graph (e.g.\ a tree) is symmetric around $0$.} \qed
\end{proposition}

\begin{figure}
    \centering
    \begin{tikzpicture}[very thick, scale=0.4]
        \draw (3,0) -- +(0,2) node[vertex]{};
        \draw (-1,0) node[vertex]{} -- (1,0) node[vertex]{} -- (3,0) node[vertex]{} -- (5,0) node[vertex]{} -- (7,0) node[vertex]{} -- (9,0) node[vertex]{} -- (11,0) node[vertex]{} -- (13,0) node[vertex]{} -- (15,0) node[vertex]{};
        \draw [braket] (15,0.1) -- (5,0.1) node [black,midway,yshift=-15pt] {\footnotesize $n$};
    \end{tikzpicture}
    \caption{$E_{2,n}$} \label{fig:e2n}
\end{figure}

\subsection{Proof of \texorpdfstring{\cref{thm:main1}}{the first main theorem} for \texorpdfstring{$\la < 2$}{λ < 2}} \label{subsec:less-than-2}

Hoffman demonstrated several sequences of graphs $G_1, G_2, \dots$ such that $G_m$ is a subgraph of $G_{m+1}$ for every $m$ and $\lim_{m \to \infty} \la_1(G_m) \le -2$. To state these results, we introduce the following notions.

\begin{definition} \label{def:extension}
    Given a nonempty vertex subset $A$ of a graph $F$, $\ell \in \N$, and $m \in \N^+$,
    \begin{enumerate}[label=(\alph*)]
        \item the \emph{path extension} $(F, A, \ell)$ is obtained from $F$ by adding a path $v_0 \dots v_\ell$ of length $\ell$, and connecting $v_0$ to every vertex in $A$;\footnote{When $\ell = 0$, the path of length $\ell$ is simply a single vertex.}
        \item the \emph{path-clique extension} $(F, A, \ell, K_m)$ is further obtained from $(F, A, \ell)$ by adding a clique of order $m$, and connecting every vertex in the clique to $v_\ell$;
        \item the \emph{clique extension} $(F, A, K_m)$ is obtained from $F$ by adding a clique of order $m$, and connecting every vertex in the clique to every vertex in $A$.
    \end{enumerate}
\end{definition}

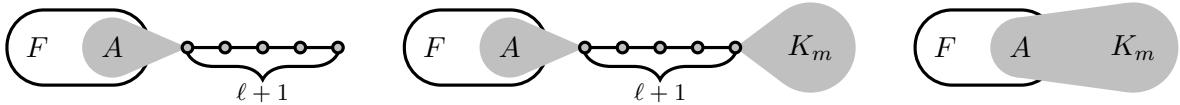
The following figure consists of schematic drawings of the path extension $(F, A, \ell)$, the path-clique extension $(F, A, \ell, K_m)$, and the clique extension $(F, A, K_m)$.

\begin{figure}[h]
    \centering
    \begin{tikzpicture}[very thick, scale=0.5, baseline=(v.base)]
        \coordinate (v) at (0,0);
        \draw[rounded corners=14pt] (-7.8, -1) rectangle (-4, 1) {};
        \fill[litegray] (-3,0) -- (-4.68,0.733) -- (-4.68,-0.733) -- cycle;
        \fill[litegray] (-5, 0) circle (0.8);
        \draw (-5,0) node{$A$};
        \draw (-7,0) node{$F$};
        \draw (-3,0) node[vertex]{} -- (-2,0) node[vertex]{} -- (-1,0) node[vertex]{} -- (0,0) node[vertex]{} -- (1,0) node[vertex]{};
        \draw [braket] (1,0.2) -- (-3,0.2) node [black,midway,yshift=-15pt] {\footnotesize $\ell + 1$};
    \end{tikzpicture}\qquad%
    \begin{tikzpicture}[very thick, scale=0.5, baseline=(v.base)]
        \coordinate (v) at (0,0);
        \draw[rounded corners=14pt] (-7.8, -1) rectangle (-4, 1) {};
        \fill[litegray] (-3,0) -- (-4.68,0.733) -- (-4.68,-0.733) -- cycle;
        \fill[litegray] (1,0) -- (2.28,0.96) -- (2.28,-0.96) -- cycle;
        \fill[litegray] (-5, 0) circle (0.8);
        \fill[litegray] (3, 0) circle (1.2);
        \draw (-5,0) node{$A$};
        \draw (-7,0) node{$F$};
        \draw (3,0) node{$K_m$};
        \draw (-3,0) node[vertex]{} -- (-2,0) node[vertex]{} -- (-1,0) node[vertex]{} -- (0,0) node[vertex]{} -- (1,0) node[vertex]{};
        \draw [braket] (1,0.2) -- (-3,0.2) node [black,midway,yshift=-15pt] {\footnotesize $\ell + 1$};
    \end{tikzpicture}\qquad%
    \begin{tikzpicture}[very thick, scale=0.5, baseline=(v.base)]
        \coordinate (v) at (0,0);
        \draw[rounded corners=14pt] (-7.8, -1) rectangle (-4, 1) {};
        \fill[litegray] (-5.107,0.79) -- (-2.16,1.19) -- (-2.16,-1.19) -- (-5.107,-0.79) -- cycle;
        \fill[litegray] (-5, 0) circle (0.8);
        \fill[litegray] (-2, 0) circle (1.2);
        \draw (-5,0) node{$A$};
        \draw (-7,0) node{$F$};
        \draw (-2,0) node{$K_m$};
    \end{tikzpicture}
\end{figure}
  
We compile some of Hoffman's computation~\cite{H77a} and two classical results in the following lemma.

\begin{lemma} \label{lem:hoffman}
    Denote by $C_n$ the cycle of length $n$, and $V_2(C_n)$ a set of two adjacent vertices of $C_n$. The path-clique extensions and clique-extensions of $C_n$ satisfy:
    \begin{itemize}
        \item[(c1)] $\lim_{m\to\infty} \la_1(C_n, V_2(C_n), \ell, K_m) \le -2$ for fixed $n \ge 3$ and $\ell \in \N$; \wlabel{lem:c1}{c1}
        \item[(c2)] $\lim_{m\to\infty} \la_1(C_n, V_2(C_n), K_m) = -2$ for fixed $n \ge 3$. \wlabel{lem:c2}{c2}
    \end{itemize}
    Denote by $K_n$ the complete graphs with $n$ vertices. The path-clique extensions of $K_n$ satisfy:
    \begin{itemize}
        \item[(k)] $\lim_{m\to\infty} \la_1(K_m, V(K_m), \ell, K_m) = -2$ for fixed $\ell \in \N$. \wlabel{lem:k}{k}
    \end{itemize}
    Denote by $\overline{K}_2$ the null graph with $2$ vertices. The path-clique extensions and the clique extensions of $\overline{K}_2$ satisfy:
    \begin{itemize}
        \item[(n1)] $\lim_{m\to\infty} \la_1(\overline{K}_2, V(\overline{K}_2), \ell, K_m) = -2$ for fixed $\ell \in \N$; \wlabel{lem:n1}{n1}
        \item[(n2)] $\lim_{m\to\infty} \la_1(\overline{K}_2, V(\overline{K}_2), K_m) = -2$. \wlabel{lem:n2}{n2}
    \end{itemize}
    Denote by $P_n$ the path of length $n$ (with $n+1$ vertices), and $S_n$ the star with $n$ leaves. Their smallest eigenvalues satisfy:
    \begin{itemize}
        \item[(p)] $\lim_{n\to\infty} \la_1(P_n) = -2$; \wlabel{lem:p}{p}
        \item[(s)] $\la_1(S_n) = -\sqrt{n}$. \wlabel{lem:s}{s} \qed
    \end{itemize}
\end{lemma}

\begin{remark}
    In \cref{lem:hoffman}, (\ref{lem:c1}) to (\ref{lem:n2}) are taken directly from \cite[Lemmas~2.4 to 2.9]{H77a}, and (\ref{lem:p}) and (\ref{lem:s}) follow from the classical results $\la^1(P_n) = 2\cos(\pi/(n+2))$ and $\la^1(S_n) = \sqrt{n}$.
\end{remark}

\cref{lem:hoffman} allows us to build a finite set of forbidden subgraphs for $\Gl$. To state the result, we introduce the following definition.

\begin{definition}[Extension family] \label{def:extension-family}
    Given a graph $F$ and $\ell, m \in \N^+$, the \emph{extension family} $\X(F, \ell, m)$ of $F$ consists of the path-extension $(F, A, \ell)$, the path-clique extension $(F, A, \ell_0, K_m)$, and the clique extension $(F, A, K_m)$, where $A$ ranges over the nonempty vertex subsets of $F$, and $\ell_0$ ranges over $\sset{0, \dots, \ell-1}$.
\end{definition}

\begin{lemma} \label{lem:hoffman-extension}
    For every $\la < 2$, there exist $\ell, m \in \N^+$ such that the extension families $\X(C, \ell, m)$ and $\X(D, \ell, m)$ are both disjoint from $\Gl$, where $C$ is the claw graph and $D$ is the diamond graph (see \cref{fig:claw-and-diamond}).
\end{lemma}

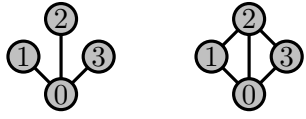
\begin{figure}[t]
    \centering
    \begin{tikzpicture}[very thick, scale=0.8]
        \draw (0,0) -- (-1,1);
        \draw (0,0) -- (0,2);
        \draw (0,0) -- (1,1);
        \node[circ] at (0,0) {0};
        \node[circ] at (-1,1) {1};
        \node[circ] at (0,2) {2};
        \node[circ] at (1,1) {3};
        \draw (5,0) -- +(-1,1) -- +(0,2);
        \draw (5,0) -- +(1,1) -- +(0,2);
        \draw (5,0) -- +(0,2);
        \node[circ] at (5,0) {0};
        \node[circ] at (4,1) {1};
        \node[circ] at (5,2) {2};
        \node[circ] at (6,1) {3};
    \end{tikzpicture}
    \caption{The claw graph $C$ and the diamond graph $D$.} \label{fig:claw-and-diamond}
\end{figure}

\begin{proof}
    From \cref{lem:hoffman}(\ref{lem:p}), we obtain $\ell \in \N^+$ such that $P_\ell \not\in \Gl$. Clearly, for every nonempty $A$, the path extensions $(C, A, \ell)$ and $(D, A, \ell)$, each of which contains $P_\ell$ as a subgraph, are not in $\Gl$. With hindsight, using \cref{lem:hoffman}(\ref{lem:c1}, \ref{lem:c2}, \ref{lem:n1}, \ref{lem:n2}), we choose $m \in \N^+$ such that none of the following graphs
    \begin{equation} \label{eqn:hoffman-extension}
        (C_3, V_2(C_3), \ell_0, K_m), (C_3, V_2(C_3), K_m), (\overline{K}_2, V(\overline{K}_2), \ell_0, K_m), (\overline{K}_2, V(\overline{K}_2), K_m),
    \end{equation}
    where $\ell_0 \in \sset{0, \dots, \ell + 1}$, is in $\Gl$.

    It suffices to prove that every clique extension and every path-clique extension in the extension families $\X(C, \ell, m)$ and $\X(D, \ell, m)$ contains one of the graphs in \eqref{eqn:hoffman-extension} as a subgraph. Label the vertices of $C$ and $D$ as in \cref{fig:claw-and-diamond}, and pick an arbitrary $\ell_0 \in \sset{0, \dots, \ell - 1}$. The clique extension $(C, A, K_m)$ and the path-clique extension $(C, A, \ell_0, K_m)$ respectively contain
    \begin{align*}
        (\overline{K}_2, V(\overline{K}_2), 0, K_m) \text{ and } (\overline{K}_2, V(\overline{K}_2), \ell_0 + 1, K_m) & \text{ when } 0 \in A \text{ and }\abs{A \cap \sset{1,2,3}} \le 1;\\
        (\overline{K}_2, V(\overline{K}_2), 1, K_m) \text{ and } (\overline{K}_2, V(\overline{K}_2), \ell_0 + 2, K_m) & \text{ when } 0 \not\in A \text{ and }\abs{A \cap \sset{1,2,3}} = 1;\\
        (\overline{K}_2, V(\overline{K}_2), K_m) \text{ and } (\overline{K}_2, V(\overline{K}_2), \ell_0, K_m) & \text{ when } \abs{A \cap \sset{1,2,3}} \ge 2,
    \end{align*}
    and the clique extension $(D, A, K_m)$ and the path-clique extension $(D, A, \ell_0, K_m)$ respectively contain
    \begin{align*}
        (\overline{K}_2, V(\overline{K}_2), 0, K_m) \text{ and } (\overline{K}_2, V(\overline{K}_2), \ell_0 + 1, K_m) & \text{ when } A \subseteq \sset{0,2} \text{ and }A \neq \varnothing; \\
        (C_3, V_2(C_3), 0, K_m) \text{ and } (C_3, V_2(C_3), \ell_0 + 1, K_m) & \text{ when } A \in \sset{\sset{1}, \sset{3}}; \\
        (\overline{K}_2, V(\overline{K}_2), K_m) \text{ and } (\overline{K}_2, V(\overline{K}_2), \ell_0, K_m) & \text{ when } A \supseteq \sset{1,3}; \\
        (C_3, V_2(C_3), K_m) \text{ and } (C_3, V_2(C_3), \ell_0, K_m) & \text{ when } \abs{A \cap \sset{0,2}} \ge 1 \text{ and } \abs{A \cap \sset{1,3}} = 1.
    \end{align*}
    Therefore the extension families $\X(C, \ell, m)$ and $\X(D, \ell, m)$ are both disjoint from $\Gl$.
\end{proof}

The next result shows that forbidding a star and an extension family of $F$ effectively forbids $F$ itself in every sufficiently large connected graph.

\begin{lemma} \label{lem:forbid-extensions}
    For every graph $F$ and $k, \ell, m \in \N^+$, there exists $N \in \N$ such that for every connected graph $G$ with more than $N$ vertices, if no member in $\sset{S_k} \cup \X(F, \ell, m)$ is a subgraph of $G$, then neither is $F$.
\end{lemma}

\begin{proof}
    With hindsight, we choose $N = vd^{\ell}$, where $d$ is the Ramsey number $R(k, 2^{v + 1}m + v)$, and $v$ is the number of vertices of $F$. Suppose that $G$ is a connected graph with more than $N$ vertices that contains no member in $\sset{S_k} \cup \X(F, \ell, m)$ as a subgraph. Assume for the sake of contradiction that the subgraph of $G$ induced on a vertex subset, say $V$, is isomorphic to $F$.
    
    Since $G$ is a connected graph with more than $vd^{\ell}$ vertices, we claim that either there exists a vertex $u$ at distance at least $\ell$ from $V$, or the maximum degree of $G$ is at least $d$. Indeed, assume for the sake of contradiction that every vertex of $G$ is within distance $\ell-1$ from $V$, and the maximum degree of $G$ is less than $d$. For $i \in \sset{0, 1, \dots, \ell-1}$, let $V_i$ be the set of vertices of at distance $i$ from $V$. Clearly $V_0 = V$, and so $\abs{V_0} = v$. Using the assumption on the maximum degree of $G$, one can inductively show that $\abs{V_i} \le vd^i$. Therefore the number of vertices in $G$ expressed as $\sum_{i = 0}^{\ell - 1}\abs{V_i}$ is at most $\sum_{i=0}^{\ell - 1}vd^i = v(d^\ell - 1)/(d-1) \le vd^\ell$, which yields a contradiction.
    
    We break the rest of the proof into two cases.

    \bigskip
    \noindent\textit{Case 1:} There exists a vertex $u$ at distance at least $\ell$ from $V$. Thus the subgraph $G[V]$ and a shortest path from $V$ to $u$ contains the path extension $(F, A, \ell)$ as a subgraph for some nonempty $A \subseteq V$, which is a contradiction.

    \bigskip
    \noindent\textit{Case 2:} The maximum degree of $G$ is at least $d$. Let $u$ be a vertex of $G$ with the maximum degree, and let $N(u)$ be the set of neighbors of $u$ in $G$. Since $S_k$ is not a subgraph of $G$, the subgraph $G[N(u)]$ cannot have an independent set of size $k$. Because $d = R(k, 2^{v + 1}m + v)$, the subgraph $G[N(u)]$ contains a clique of order $2^{v + 1}m + v$, and so there exists a clique $G[U]$ of size $2^{v + 1}m$ that is vertex-disjoint from $F$.
    
    We further partition the vertices in $U$ as follows. For every subset $A$ of $V$, let $U(A)$ be the set of vertices in $U$ that are adjacent to all vertices in $A$ and not adjacent to any vertex in $V \setminus A$. By the pigeonhole principle, there exists a subset $A$ of $V$ such that $\abs{U(A)} \ge 2m$. If $A$ is nonempty, then $G[V \cup U(A)]$ contains the clique extension $(F, A, K_m)$ as a subgraph, which is a contradiction.
    
    Hereafter we may assume that $\abs{U(\varnothing)} \ge 2m$, hence the distance between $V$ and $U(\varnothing)$ is at least $2$. Let $v_0 \dots v_{\ell_0}$ be a shortest path between $V$ and $U(\varnothing)$, where $v_0$ and $v_{\ell_0}$ are respectively at distance $1$ from $V$ and $U(\varnothing)$. Let $A_1$ be the nonempty set of vertices in $V$ that are adjacent to $v_0$, and denote by $v_{\ell_0 + 1}$ an arbitrary vertex in $U(\varnothing)$ that is adjacent to $v_{\ell_0}$. We may assume that $\ell_0 < \ell - 1$ because otherwise $G[V \cup \sset{v_0, \dots, v_{\ell_0}, v_{\ell_0 + 1}}]$ would contain the path extension $(F, A_1, \ell)$ as a subgraph, which is a contradiction. Let $A_2$ be the nonempty set of vertices in $U(\varnothing)$ that are adjacent to $v_{\ell_0}$. If $\abs{A_2} \ge m$, then $G[V \cup \sset{v_0, \dots, v_{\ell_0}} \cup A_2]$ contains the path-clique extension $(F, A_1, \ell_0, K_m)$ as a subgraph. Otherwise $\abs{U(\varnothing) \setminus A_2} > m$, and so $G[V \cup \sset{v_0, \dots, v_{\ell_0}, v_{\ell_0 + 1}} \cup (U(\varnothing) \setminus A_2)]$ contains the path-clique extension $(F, A_1, \ell_0 + 1, K_m)$.
\end{proof}

Combining \cref{lem:hoffman-extension,lem:forbid-extensions}, we obtain a finite set of forbidden subgraphs for $\Gl$ that forbids the claw graph and the diamond graph in every sufficiently large connected graph. The following result, which is an immediate consequence of \cite[Theorem~4]{RW65}, gives a sufficient condition for line graphs.

\begin{theorem}[Theorem~4 of van Rooij and Wilf~\cite{RW65}] \label{thm:claw-diamond}
    Every graph that contains neither the claw graph nor the diamond graph as a subgraph is a line graph. \qed
\end{theorem}

Furthermore, we can bootstrap to obtain a finite set of forbidden subgraphs for $\Gl$ that forces every sufficiently large connected graph to be the line graph of a tree whose complexity is uniformly bounded.

\begin{lemma} \label{lem:line-graph-of-tn}
    For every $\la < 2$, there exist $N \in \N$ and a finite family $\F_1$ that is disjoint from $\Gl$ such that for every connected graph $G$ with more than $N$ vertices, if $G$ contains no member in $\F_1$ as a subgraph, then there exists a rooted tree $H \in \mathcal{T}_N$ such that $G$ is the line graph of $H$, where $\mathcal{T}_N$ is the family of rooted trees such that every connected component obtained from removing the root has at most $N$ vertices.
\end{lemma}

\begin{proof}
    We obtain $\ell, m \in \N^+$ from \cref{lem:hoffman,lem:hoffman-extension} such that the following family
    \begin{align*}
        \F_1 := & \sset{S_4, P_\ell} \cup \X(C, \ell, m) \cup \X(D, \ell, m) \\
        & \cup \dset{(C_n, V_2(C_n), \ell_0, K_m)}{n \in \sset{3, \dots, \ell + 1}, \ell_0 \in \sset{0, \dots, \ell - 1}} \\
        & \cup \dset{(C_n, V_2(C_n), K_m)}{n \in \sset{3, \dots, \ell + 1}} \\
        & \cup \dset{(K_m, V(K_m), \ell_0, K_m)}{\ell_0 \in \sset{0, \dots, \ell-1}}
    \end{align*}
    is disjoint from $\Gl$.

    \label{proof:main1-claim-start}
    We obtain $N_0$ from \cref{lem:forbid-extensions} such that for every connected graph $G$ with more than $N_0$ vertices, if no member in $\sset{S_4} \cup \X(C, \ell, m) \cup \X(D, \ell, m)$ is a subgraph of $G$, then neither is $C$ nor $D$. With hindsight, we choose $N = \max(N_0, (m+1)^{\ell+2})$. Suppose that $G$ is a connected graph with more than $N$ vertices, and suppose that no member in $\F_1$ is a subgraph of $G$. By our choice of $N_0$, we know that neither $C$ nor $D$ is a subgraph of $G$.
    
    \cref{thm:claw-diamond} implies that $G$ is the line graph of another connected graph, say $H$. Since $P_\ell$ is not a subgraph of $G$, the path $P_{\ell + 1}$ cannot be a general subgraph of $H$. In particular, the diameter\footnote{Recall that the \emph{diameter} of a graph is the maximum distance between a pair of vertices in the graph.} of $H$ is at most $\ell$. Let $d$ be the maximum degree of $H$. If $d < m + 2$, then $H$ has at most $(m+1)^{\ell+1}$ vertices, and at most $(m+1)^{\ell+2}$ edges, and so $G$ has at most $(m+1)^{\ell+2}$ vertices, which is a contradiction.
    
    We may assume that $d \ge m + 2$. Let $u$ be a vertex of $H$ with degree $d$. We claim that $H$ is a tree. Suppose on the contrary that $H$ contains a cycle as a general subgraph. Since $H$ does not contain $P_{\ell + 1}$ as a general subgraph, the length of any cycle in $H$ cannot exceed $\ell + 1$. Suppose in addition that $u$ is on a cycle in $H$. Let $C_n$ be a shortest cycle containing $u$, where $n \in \sset{3, \dots, \ell + 1}$. By the choice of $C_n$, the vertex $u$ has at least $d - 2 \ge m$ neighbors outside $C_n$. Thus $H$ contains $C_n$ with $S_m$ attached at $u$ as a general subgraph, and so $G$ contains $(C_n, V_2(C_n), K_m)$ as a subgraph, which is a contradiction. Therefore $u$ is not on any cycle in $H$. Let $C_n$ be any cycle in $H$, where $n \in \sset{3, \dots, \ell + 1}$. Let $P_{\ell_0}$ be a shortest path between $u$ and $C_n$, where $\ell_0 \in \sset{1, \dots, \ell}$. Since $u$ is not on any cycle, $u$ has at least $d - 1 > m$ neighbors outside $C_n$ and $P_{\ell_0}$. Thus $H$ contains $C_n \cup P_{\ell_0}$ with $S_m$ attached at $u$ as a general subgraph, and so $G$ contains $(C_n, V_2(C_n), \ell_0 - 1, K_m)$ as a subgraph, which is a contradiction.

    We now view $H$ as a tree rooted at $u$. We claim that $u$ is the only vertex with degree larger than $m$. Suppose on the contrary that there is another vertex $u'$ with degree at least $m + 1$ in $H$. Let $P_{\ell_0}$ be a shortest path between $u$ and $u'$, where $\ell_0 \in \sset{1, \dots, \ell}$. Thus $H$ contains $P_{\ell_0}$ with two vertex-disjoint stars $S_m$ respectively attached to $u$ and $u'$ as a (general) subgraph, and so $G$ contains $(K_m, V(K_m), \ell_0 - 1, K_m)$ as a subgraph, which is a contradiction. Therefore all the vertices but $u$ in $H$ has degree at most $m$. After the root $u$ is removed from $H$, each connected component has diameter at most $\ell$ and degree at most $m$, and so each connected component has at most $m^{\ell+1} \le N$ vertices.
\end{proof}

The last ingredient for the proof of \cref{thm:main1} for $\la < 2$ is the following lemma attributed to Dickson, who used it to prove a result about perfect numbers in number theory.

\begin{lemma}[Lemma~A of Dickson~\cite{D13}] \label{lem:dickson}
    For every $n \in \N^+$, the partially ordered set $(\N^n, \le)$, in which $(a_1, \dots a_n) \le (b_1, \dots, b_n)$ if and only if $a_i \le b_i$ for every $i$, does not contain infinite antichains.
\end{lemma}

Dickson gave two proofs of \cref{lem:dickson}, one of which uses induction, while the other uses Hilbert's basis theorem. For the reader's convenience, we provide a short combinatorial proof based on the infinite Ramsey's theorem.

\begin{proof}
    Suppose that $s_1, s_2, \dots$ is an infinite sequence of distinct tuples in $\N^n$. For every $i < j$, color the edge $s_i s_j$ by $0$ if $s_i < s_j$, otherwise by any $k \in \sset{1, \dots, n}$ such that $s_{i,k} > s_{j,k}$. The infinite Ramsey's theorem provides an infinite subset $I \subseteq \N^+$ such that the edges $s_i s_j$, where $i, j \in I$ and $i\neq j$, all receive the same color $c$. Because $(\N, \le)$ does not contain infinite descending chains, it must be the case that $c = 0$. This implies that $\dset{s_i}{i \in I}$ is an infinite ascending chain, and in particular $\sset{s_1, s_2, \dots}$ is not an antichain.
\end{proof}

We are ready to establish the first main theorem for $\la < 2$.

\begin{proof}[Proof of \cref{thm:main1} for $\la < 2$]
    Let $N \in \N$ and $\F_1$ be given by \cref{lem:line-graph-of-tn}, and set
    \begin{align*}
        \F_0 & := \dset{G \not\in \Gl}{G \text{ has at most }N \text{ vertices}}, \\
        \widetilde{\F}_2 & := \dset{G \not\in \Gl}{\exists H \in \mathcal{T}_N \text{ s.t.\ }G = L(H)},
    \end{align*}
    where $\mathcal{T}_N$ is the family of rooted trees such that every connected component obtained from removing the root has at most $N$ vertices. Setting $\F_2$ to be the family of graphs that are minimal in $\widetilde{\F}_2$ under taking subgraphs, one can check that $\F_0 \cup \F_1 \cup \F_2$ is a forbidden subgraph characterization of $\Gl$.

    It suffices to prove that $\F_2$ is finite. Let $T_1, \dots, T_n$ be an enumeration of rooted trees on at most $N$ vertices. We encode $G \in \F_2$ by $t_G \in \N^n$ as follows. Let $H$ be the rooted tree in $\mathcal{T}_N$ such that $G = L(H)$. After removing the root $u$ from $H$, we view each connected component as a tree rooted at the vertex that is a child of $u$ in $H$. Set $t_G := (t_1, \dots, t_n)$, where $t_i$ is the number of occurrences of $T_i$ as a connected component in the graph obtained by removing $u$ from $H$. Because no member of $\F_2$ is a subgraph of any other, one can deduce that $\dset{t_G}{G \in \F_2}$ is an antichain in $(\N^n, \le)$, and so $\F_2$ is finite by \cref{lem:dickson}.
\end{proof}

\subsection{Proof of \texorpdfstring{\cref{thm:main1}}{the first main theorem} for \texorpdfstring{$\la \in [2, \la^*)$}{λ in [2, λ*)}} \label{sec:main1-2}

We prove a stronger result from which the first main theorem for $\la \in [2, \la^*)$ and its corollary for $\la \in (2,\la^*)$ follow.

\begin{theorem} \label{thm:main1-2}
    For every $\la \in [2, \la^*)$, the number of connected graphs in $\Gl \setminus \G(2)$ is finite.
\end{theorem}

\begin{proof}[Proof of \cref{thm:main1} for $\la \in [2, \la^*)$ assuming \cref{thm:main1-2}] \label{proof:main1-0}
    It is known that $\G(2)$ has a finite forbidden subgraph characterization --- in fact, there are $1812$ minimal forbidden subgraphs for $\G(2)$; see \cite{BN92}. In view of \cref{thm:main1-2}, it suffices to show that if $G$ is a minimal forbidden subgraph for $\Gl$, then $G$ is a minimal forbidden subgraph for $\G(2)$, or $G$, after removing some vertex, is a connected graph in $\Gl \setminus \G(2)$.
    
    Indeed, let $G$ be a minimal forbidden subgraph for $\Gl$. Because $\la_1(G) < -\la \le -2$, the graph $G$ contains a minimal forbidden subgraph for $\G(2)$ as a subgraph, say $F$. Clearly both $G$ and $F$ are connected. We may assume that $F$ is a proper subgraph of $G$ because otherwise we are done already. Consider the graph $H$ obtained after removing $v$ from $G$, where $v$ is a vertex of $G$ that is furthest from $F$. Using the connectivity of $F$, one can then show that $H$ is connected. Finally, $H \in \Gl$ because of the minimality of $G$, and $H \not\in \G(2)$ because $F$ is a subgraph of $H$.
\end{proof}

\begin{proof}[Proof of \cref{cor:main1} for $\la \in (2,\la^*)$]
    It follows immediately from \cref{thm:main1-2} that $-\la$ is not a limit point of the set of smallest eigenvalues of graphs for $\la \in (2,\la^*)$
\end{proof}

The proof of \cref{thm:main1-2} centers around the notion of \emph{generalized line graphs}. Although we do not need their definition, we state it nevertheless for concreteness.

\begin{definition}[Cocktail party graphs and generalized line graphs] \label{def:gen-line-graph}
    The \emph{cocktail party graph} $\overline{a K_2}$ is obtained from the complete graph on $2a$ vertices by deleting a perfect matching. Given a graph $G$ with vertices $v_1, \dots, v_n$, and $a_1, \dots a_n \in \N$, the \emph{generalized line graph} $L(G; a_1, \dots, a_n)$ is obtained from the line graph $L(G)$ of $G$ by adjoining $n$ vertex-disjoint cocktail party graphs $\overline{a_1K_2}, \dots, \overline{a_nK_2}$ where every vertex of the $i$th cocktail party graph $\overline{a_iK_2}$ is adjacent to every vertex of $L(G)$ that contains $v_i$. See \cref{fig:generalized-line-graph} for an example.
\end{definition}

\begin{figure}[t]
    \centering
    \begin{tikzpicture}[very thick, scale=0.6, baseline=(v.base)]
        \coordinate (v) at (0,0);
        \draw (-2,0) -- (0,-1) node[vertex]{} -- (2,0);
        \draw (-2,0) node[vertex]{} -- (2,0) node[vertex]{} -- (0,1) node[vertex]{} -- cycle;
        \draw (-2,0) node[left]{$v_1$};
        \draw (0,1) node[above]{$v_2$};
        \draw (0,-1.1) node[below]{$v_3$};
        \draw (2,0) node[right]{$v_4$};
    \end{tikzpicture}\qquad
    \begin{tikzpicture}[very thick, scale=0.6, baseline=(v.base)]
        \coordinate (v) at (0,0);
        \draw[fill=litegray, draw=none] (0,0) -- (-5.5,0.2) -- (-5.5,-0.2) -- cycle;
        \draw[fill=litegray, draw=none] (-3,1.5) -- (-5.5,0.25) -- (-5.5,-0.3) -- cycle;
        \draw[fill=litegray, draw=none] (-3,-1.5) -- (-5.5,0.3) -- (-5.5,-0.25) -- cycle;
        \draw[fill=litegray, draw=none] (-3,1.5) -- (0,3.3) -- (0,2.8) -- cycle;
        \draw[fill=litegray, draw=none] (3,1.5) -- (0,3.3) -- (0,2.8) -- cycle;
        \draw[fill=litegray, draw=none] (0,0) -- (6,0.22) -- (6,-0.22) -- cycle;
        \draw[fill=litegray, draw=none] (3,1.5) -- (6,0.27) -- (6,-0.32) -- cycle;
        \draw[fill=litegray, draw=none] (3,-1.5) -- (6,0.32) -- (6,-0.27) -- cycle;
        \draw (-3,1.5) node[vertex]{} -- (3,1.5) -- (3,-1.5) node[vertex]{} -- (-3,-1.5) -- (-3,1.5) -- (3,-1.5);
        \draw (-3,-1.5) node[vertex]{} -- (0,0) node[vertex]{} -- (3,1.5)node[vertex]{};
        \node[below] at (0,-0.3) {$v_1v_4$};
        \node[above] at (-3.5,1.5) {$v_1v_2$};
        \node[above] at (3.5,1.5) {$v_2v_4$};
        \node[below] at (-3.5,-1.6) {$v_1v_3$};
        \node[below] at (3.5,-1.6) {$v_3v_4$};
        \draw[fill=litegray, draw=none] (-5.5, 0) circle (1.1);
        \draw (-5, -0.5) node[vertex]{} -- (-5,0.5) node[vertex]{}-- (-6,0.5) node[vertex]{} -- (-6,-0.5) node[vertex]{} -- cycle;
        \draw[fill=litegray, draw=none] (0, 3) circle (0.9);
        \draw (-0.5,3) node[vertex]{};
        \draw (0.5,3) node[vertex]{};
        \draw[fill=litegray, draw=none] (6, 0) circle (1.4);
        \draw (7,0) -- ++(150:{sqrt(3)}) -- ++(0,{-sqrt(3)}) -- cycle;
        \draw (5,0) -- ++(30:{sqrt(3)}) -- ++(0,{-sqrt(3)}) -- cycle;
        \draw (7,0) node[vertex]{} -- ++(120:1) node[vertex]{} -- ++(-1,0) node[vertex]{} -- ++(240:1) node[vertex]{} -- ++(-60:1) node[vertex]{} -- ++(1,0) node[vertex]{} -- cycle;
    \end{tikzpicture}
    \caption{A graph $G$ and a schematic drawing of its generalized line graph $L(G; 2, 1, 0, 3)$.} \label{fig:generalized-line-graph}
\end{figure}
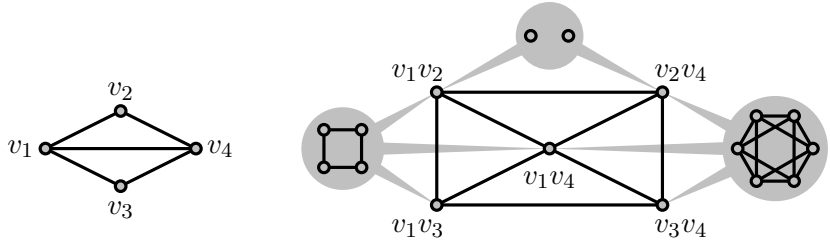

Just like line graphs, all the generalized line graphs are in $\G(2)$, and they have a finite forbidden subgraph characterization.

\begin{theorem}[Theorem~2.1 of Hoffman~\cite{H77b}] \label{thm:generalized-line-smallest-ev}
    The smallest eigenvalue of a generalized line graph is at least $-2$. \qed
\end{theorem}

\begin{theorem}[Cvetkovi\'{c}, Doob, and Simi\'{c}~\cite{CDS80,CDS81}, and Rao, Singhi, and Vijayan~\cite{RSV81}] \label{thm:forb-char-generalized-line}
    The minimal forbidden subgraphs for the family $\D_\infty$ of generalized line graphs are listed in \cref{fig:minimal-forb-1}. \qed
\end{theorem}

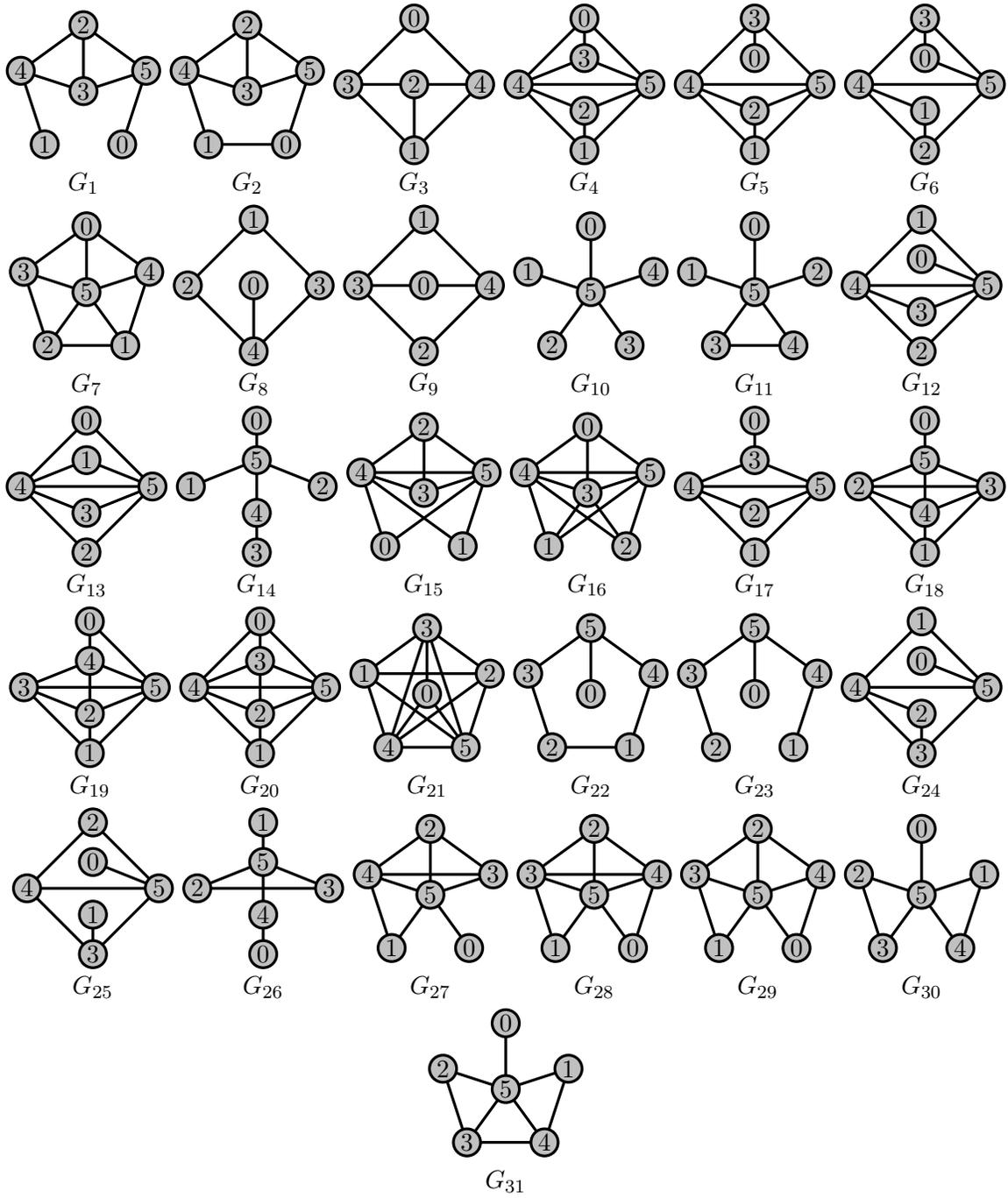
\begin{figure}
    \centering
    \begin{tikzpicture}[very thick, baseline=(v0.base)]
    \coordinate (v0) at (0,0);
    \coordinate (v1) at (0,.90450849718);
    \coordinate (v2) at (-.95105651629,0);
    \coordinate (v3) at (0,-.90450849718);
    \coordinate (v4) at (.95105651629,0);
    \draw (v0) -- (v3);
    \draw (v1) -- (v2);
    \draw (v1) -- (v4);
    \draw (v2) -- (v3);
    \draw (v3) -- (v4);
    \node[circ] at (v0) {0};
    \node[circ] at (v1) {1};
    \node[circ] at (v2) {2};
    \node[circ] at (v3) {3};
    \node[circ] at (v4) {4};
    \node at (0,-1.4) {$G_{1}$};
    \end{tikzpicture}
    \begin{tikzpicture}[very thick, baseline=(v0.base)]
    \coordinate (v0) at (0,0);
    \coordinate (v1) at (0,.90450849718);
    \coordinate (v2) at (-.95105651629,0);
    \coordinate (v3) at (0,-.90450849718);
    \coordinate (v4) at (.95105651629,0);
    \draw (v0) -- (v2);
    \draw (v0) -- (v4);
    \draw (v1) -- (v2);
    \draw (v1) -- (v4);
    \draw (v2) -- (v3);
    \draw (v3) -- (v4);
    \node[circ] at (v0) {0};
    \node[circ] at (v1) {1};
    \node[circ] at (v2) {2};
    \node[circ] at (v3) {3};
    \node[circ] at (v4) {4};
    \node at (0,-1.4) {$G_{2}$};
    \end{tikzpicture}
    \begin{tikzpicture}[very thick, baseline=(v0.base)]
    \coordinate (v0) at (0,0);
    \coordinate (v1) at (0,.90450849718);
    \coordinate (v2) at (-.95105651629,0);
    \coordinate (v3) at (0,-.90450849718);
    \coordinate (v4) at (.95105651629,0);
    \draw (v0) -- (v2);
    \draw (v0) -- (v3);
    \draw (v0) -- (v4);
    \draw (v1) -- (v2);
    \draw (v1) -- (v4);
    \draw (v2) -- (v3);
    \draw (v3) -- (v4);
    \node[circ] at (v0) {0};
    \node[circ] at (v1) {1};
    \node[circ] at (v2) {2};
    \node[circ] at (v3) {3};
    \node[circ] at (v4) {4};
    \node at (0,-1.4) {$G_{3}$};
    \end{tikzpicture}
    \begin{tikzpicture}[very thick, baseline=(v.base)]
    \coordinate (v) at (0,.09549150281);
    \coordinate (v0) at (0,0);
    \path (v0) + (90:1) coordinate (v1);
    \path (v0) + (162:1) coordinate (v2);
    \path (v0) + (234:1) coordinate (v3);
    \path (v0) + (306:1) coordinate (v4);
    \path (v0) + (378:1) coordinate (v5);
    \draw (v0) -- (v1);
    \draw (v0) -- (v2);
    \draw (v0) -- (v5);
    \draw (v2) -- (v3);
    \draw (v4) -- (v5);
    \node[circ] at (v0) {0};
    \node[circ] at (v1) {1};
    \node[circ] at (v2) {2};
    \node[circ] at (v3) {3};
    \node[circ] at (v4) {4};
    \node[circ] at (v5) {5};
    \node at (0,-1.30450849719) {$G_{4}$};
    \end{tikzpicture}
    \begin{tikzpicture}[very thick, baseline=(v.base)]
    \coordinate (v) at (0,.09549150281);
    \coordinate (v0) at (0,0);
    \path (v0) + (90:1) coordinate (v1);
    \path (v0) + (162:1) coordinate (v2);
    \path (v0) + (234:1) coordinate (v3);
    \path (v0) + (306:1) coordinate (v4);
    \path (v0) + (378:1) coordinate (v5);
    \draw (v0) -- (v1);
    \draw (v0) -- (v2);
    \draw (v0) -- (v3);
    \draw (v0) -- (v5);
    \draw (v3) -- (v4);
    \node[circ] at (v0) {0};
    \node[circ] at (v1) {1};
    \node[circ] at (v2) {2};
    \node[circ] at (v3) {3};
    \node[circ] at (v4) {4};
    \node[circ] at (v5) {5};
    \node at (0,-1.30450849719) {$G_{5}$};
    \end{tikzpicture}
    \begin{tikzpicture}[very thick, baseline=(v.base)]
    \coordinate (v) at (0,.09549150281);
    \coordinate (v0) at (0,0);
    \path (v0) + (90:1) coordinate (v1);
    \path (v0) + (162:1) coordinate (v2);
    \path (v0) + (234:1) coordinate (v3);
    \path (v0) + (306:1) coordinate (v4);
    \path (v0) + (378:1) coordinate (v5);
    \draw (v0) -- (v1);
    \draw (v0) -- (v2);
    \draw (v0) -- (v3);
    \draw (v0) -- (v4);
    \draw (v0) -- (v5);
    \node[circ] at (v0) {0};
    \node[circ] at (v1) {1};
    \node[circ] at (v2) {2};
    \node[circ] at (v3) {3};
    \node[circ] at (v4) {4};
    \node[circ] at (v5) {5};
    \node at (0,-1.30450849719) {$G_{6}$};
    \end{tikzpicture}
    \begin{tikzpicture}[very thick, baseline=(v.base)]
    \coordinate (v) at (0,.09549150281);
    \coordinate (v0) at (0,0);
    \path (v0) + (90:1) coordinate (v1);
    \path (v0) + (162:1) coordinate (v2);
    \path (v0) + (234:1) coordinate (v3);
    \path (v0) + (306:1) coordinate (v4);
    \path (v0) + (378:1) coordinate (v5);
    \draw (v0) -- (v1);
    \draw (v1) -- (v2);
    \draw (v1) -- (v5);
    \draw (v2) -- (v3);
    \draw (v3) -- (v4);
    \draw (v4) -- (v5);
    \node[circ] at (v0) {0};
    \node[circ] at (v1) {1};
    \node[circ] at (v2) {2};
    \node[circ] at (v3) {3};
    \node[circ] at (v4) {4};
    \node[circ] at (v5) {5};
    \node at (0,-1.30450849719) {$G_{7}$};
    \end{tikzpicture}
    \begin{tikzpicture}[very thick, baseline=(v.base)]
    \coordinate (v) at (0,.09549150281);
    \coordinate (v0) at (0,0);
    \path (v0) + (90:1) coordinate (v1);
    \path (v0) + (162:1) coordinate (v2);
    \path (v0) + (234:1) coordinate (v3);
    \path (v0) + (306:1) coordinate (v4);
    \path (v0) + (378:1) coordinate (v5);
    \draw (v0) -- (v2);
    \draw (v0) -- (v3);
    \draw (v0) -- (v4);
    \draw (v0) -- (v5);
    \draw (v1) -- (v2);
    \draw (v3) -- (v4);
    \node[circ] at (v0) {0};
    \node[circ] at (v1) {1};
    \node[circ] at (v2) {2};
    \node[circ] at (v3) {3};
    \node[circ] at (v4) {4};
    \node[circ] at (v5) {5};
    \node at (0,-1.30450849719) {$G_{8}$};
    \end{tikzpicture}
    \begin{tikzpicture}[very thick, baseline=(v.base)]
    \coordinate (v) at (0,.09549150281);
    \coordinate (v0) at (0,0);
    \path (v0) + (90:1) coordinate (v1);
    \path (v0) + (162:1) coordinate (v2);
    \path (v0) + (234:1) coordinate (v3);
    \path (v0) + (306:1) coordinate (v4);
    \path (v0) + (378:1) coordinate (v5);
    \draw (v0) -- (v1);
    \draw (v0) -- (v2);
    \draw (v0) -- (v3);
    \draw (v0) -- (v4);
    \draw (v0) -- (v5);
    \draw (v3) -- (v4);
    \node[circ] at (v0) {0};
    \node[circ] at (v1) {1};
    \node[circ] at (v2) {2};
    \node[circ] at (v3) {3};
    \node[circ] at (v4) {4};
    \node[circ] at (v5) {5};
    \node at (0,-1.30450849719) {$G_{9}$};
    \end{tikzpicture}
    \begin{tikzpicture}[very thick, baseline=(v.base)]
    \coordinate (v) at (0,.09549150281);
    \coordinate (v0) at (0,0);
    \path (v0) + (90:1) coordinate (v1);
    \path (v0) + (162:1) coordinate (v2);
    \path (v0) + (234:1) coordinate (v3);
    \path (v0) + (306:1) coordinate (v4);
    \path (v0) + (378:1) coordinate (v5);
    \draw (v0) -- (v1);
    \draw (v0) -- (v2);
    \draw (v0) -- (v5);
    \draw (v1) -- (v2);
    \draw (v1) -- (v5);
    \draw (v2) -- (v3);
    \draw (v4) -- (v5);
    \node[circ] at (v0) {0};
    \node[circ] at (v1) {1};
    \node[circ] at (v2) {2};
    \node[circ] at (v3) {3};
    \node[circ] at (v4) {4};
    \node[circ] at (v5) {5};
    \node at (0,-1.30450849719) {$G_{10}$};
    \end{tikzpicture}
    \begin{tikzpicture}[very thick, baseline=(v.base)]
    \coordinate (v) at (0,.09549150281);
    \coordinate (v0) at (0,0);
    \path (v0) + (90:1) coordinate (v1);
    \path (v0) + (162:1) coordinate (v2);
    \path (v0) + (234:1) coordinate (v3);
    \path (v0) + (306:1) coordinate (v4);
    \path (v0) + (378:1) coordinate (v5);
    \draw (v0) -- (v1);
    \draw (v0) -- (v2);
    \draw (v0) -- (v3);
    \draw (v0) -- (v5);
    \draw (v1) -- (v2);
    \draw (v1) -- (v5);
    \draw (v4) -- (v5);
    \node[circ] at (v0) {0};
    \node[circ] at (v1) {1};
    \node[circ] at (v2) {2};
    \node[circ] at (v3) {3};
    \node[circ] at (v4) {4};
    \node[circ] at (v5) {5};
    \node at (0,-1.30450849719) {$G_{11}$};
    \end{tikzpicture}
    \begin{tikzpicture}[very thick, baseline=(v.base)]
    \coordinate (v) at (0,.09549150281);
    \coordinate (v0) at (0,0);
    \path (v0) + (90:1) coordinate (v1);
    \path (v0) + (162:1) coordinate (v2);
    \path (v0) + (234:1) coordinate (v3);
    \path (v0) + (306:1) coordinate (v4);
    \path (v0) + (378:1) coordinate (v5);
    \draw (v0) -- (v1);
    \draw (v0) -- (v2);
    \draw (v0) -- (v3);
    \draw (v0) -- (v4);
    \draw (v0) -- (v5);
    \draw (v2) -- (v3);
    \draw (v4) -- (v5);
    \node[circ] at (v0) {0};
    \node[circ] at (v1) {1};
    \node[circ] at (v2) {2};
    \node[circ] at (v3) {3};
    \node[circ] at (v4) {4};
    \node[circ] at (v5) {5};
    \node at (0,-1.30450849719) {$G_{12}$};
    \end{tikzpicture}
    \begin{tikzpicture}[very thick, baseline=(v.base)]
    \coordinate (v) at (0,.09549150281);
    \coordinate (v0) at (0,0);
    \path (v0) + (90:1) coordinate (v1);
    \path (v0) + (162:1) coordinate (v2);
    \path (v0) + (234:1) coordinate (v3);
    \path (v0) + (306:1) coordinate (v4);
    \path (v0) + (378:1) coordinate (v5);
    \draw (v0) -- (v1);
    \draw (v0) -- (v2);
    \draw (v0) -- (v5);
    \draw (v1) -- (v2);
    \draw (v1) -- (v5);
    \draw (v2) -- (v3);
    \draw (v3) -- (v4);
    \draw (v4) -- (v5);
    \node[circ] at (v0) {0};
    \node[circ] at (v1) {1};
    \node[circ] at (v2) {2};
    \node[circ] at (v3) {3};
    \node[circ] at (v4) {4};
    \node[circ] at (v5) {5};
    \node at (0,-1.30450849719) {$G_{13}$};
    \end{tikzpicture}
    \begin{tikzpicture}[very thick, baseline=(v.base)]
    \coordinate (v) at (0,.09549150281);
    \coordinate (v0) at (0,0);
    \path (v0) + (90:1) coordinate (v1);
    \path (v0) + (162:1) coordinate (v2);
    \path (v0) + (234:1) coordinate (v3);
    \path (v0) + (306:1) coordinate (v4);
    \path (v0) + (378:1) coordinate (v5);
    \draw (v0) -- (v3);
    \draw (v1) -- (v2);
    \draw (v1) -- (v3);
    \draw (v1) -- (v4);
    \draw (v1) -- (v5);
    \draw (v2) -- (v3);
    \draw (v3) -- (v4);
    \draw (v4) -- (v5);
    \node[circ] at (v0) {0};
    \node[circ] at (v1) {1};
    \node[circ] at (v2) {2};
    \node[circ] at (v3) {3};
    \node[circ] at (v4) {4};
    \node[circ] at (v5) {5};
    \node at (0,-1.30450849719) {$G_{14}$};
    \end{tikzpicture}
    \begin{tikzpicture}[very thick, baseline=(v.base)]
    \coordinate (v) at (0,.09549150281);
    \coordinate (v0) at (0,0);
    \path (v0) + (90:1) coordinate (v1);
    \path (v0) + (162:1) coordinate (v2);
    \path (v0) + (234:1) coordinate (v3);
    \path (v0) + (306:1) coordinate (v4);
    \path (v0) + (378:1) coordinate (v5);
    \draw (v0) -- (v2);
    \draw (v0) -- (v5);
    \draw (v1) -- (v2);
    \draw (v1) -- (v5);
    \draw (v2) -- (v3);
    \draw (v2) -- (v5);
    \draw (v3) -- (v4);
    \draw (v3) -- (v5);
    \node[circ] at (v0) {0};
    \node[circ] at (v1) {1};
    \node[circ] at (v2) {2};
    \node[circ] at (v3) {3};
    \node[circ] at (v4) {4};
    \node[circ] at (v5) {5};
    \node at (0,-1.30450849719) {$G_{15}$};
    \end{tikzpicture}
    \begin{tikzpicture}[very thick, baseline=(v.base)]
    \coordinate (v) at (0,.09549150281);
    \coordinate (v0) at (0,0);
    \path (v0) + (90:1) coordinate (v1);
    \path (v0) + (162:1) coordinate (v2);
    \path (v0) + (234:1) coordinate (v3);
    \path (v0) + (306:1) coordinate (v4);
    \path (v0) + (378:1) coordinate (v5);
    \draw (v0) -- (v1);
    \draw (v0) -- (v2);
    \draw (v0) -- (v3);
    \draw (v0) -- (v4);
    \draw (v0) -- (v5);
    \draw (v2) -- (v3);
    \draw (v3) -- (v4);
    \draw (v4) -- (v5);
    \node[circ] at (v0) {0};
    \node[circ] at (v1) {1};
    \node[circ] at (v2) {2};
    \node[circ] at (v3) {3};
    \node[circ] at (v4) {4};
    \node[circ] at (v5) {5};
    \node at (0,-1.30450849719) {$G_{16}$};
    \end{tikzpicture}
    \begin{tikzpicture}[very thick, baseline=(v.base)]
    \coordinate (v) at (0,.09549150281);
    \coordinate (v0) at (0,0);
    \path (v0) + (90:1) coordinate (v1);
    \path (v0) + (162:1) coordinate (v2);
    \path (v0) + (234:1) coordinate (v3);
    \path (v0) + (306:1) coordinate (v4);
    \path (v0) + (378:1) coordinate (v5);
    \draw (v0) -- (v2);
    \draw (v0) -- (v5);
    \draw (v1) -- (v2);
    \draw (v1) -- (v5);
    \draw (v2) -- (v3);
    \draw (v2) -- (v5);
    \draw (v3) -- (v5);
    \draw (v4) -- (v5);
    \node[circ] at (v0) {0};
    \node[circ] at (v1) {1};
    \node[circ] at (v2) {2};
    \node[circ] at (v3) {3};
    \node[circ] at (v4) {4};
    \node[circ] at (v5) {5};
    \node at (0,-1.30450849719) {$G_{17}$};
    \end{tikzpicture}
    \begin{tikzpicture}[very thick, baseline=(v.base)]
    \coordinate (v) at (0,.09549150281);
    \coordinate (v0) at (0,0);
    \path (v0) + (90:1) coordinate (v1);
    \path (v0) + (162:1) coordinate (v2);
    \path (v0) + (234:1) coordinate (v3);
    \path (v0) + (306:1) coordinate (v4);
    \path (v0) + (378:1) coordinate (v5);
    \draw (v0) -- (v2);
    \draw (v0) -- (v3);
    \draw (v0) -- (v4);
    \draw (v0) -- (v5);
    \draw (v1) -- (v2);
    \draw (v1) -- (v5);
    \draw (v2) -- (v3);
    \draw (v2) -- (v5);
    \draw (v3) -- (v4);
    \node[circ] at (v0) {0};
    \node[circ] at (v1) {1};
    \node[circ] at (v2) {2};
    \node[circ] at (v3) {3};
    \node[circ] at (v4) {4};
    \node[circ] at (v5) {5};
    \node at (0,-1.30450849719) {$G_{18}$};
    \end{tikzpicture}
    \begin{tikzpicture}[very thick, baseline=(v.base)]
    \coordinate (v) at (0,.09549150281);
    \coordinate (v0) at (0,0);
    \path (v0) + (90:1) coordinate (v1);
    \path (v0) + (162:1) coordinate (v2);
    \path (v0) + (234:1) coordinate (v3);
    \path (v0) + (306:1) coordinate (v4);
    \path (v0) + (378:1) coordinate (v5);
    \draw (v0) -- (v1);
    \draw (v0) -- (v2);
    \draw (v0) -- (v3);
    \draw (v0) -- (v5);
    \draw (v1) -- (v2);
    \draw (v1) -- (v5);
    \draw (v2) -- (v3);
    \draw (v2) -- (v5);
    \draw (v3) -- (v4);
    \node[circ] at (v0) {0};
    \node[circ] at (v1) {1};
    \node[circ] at (v2) {2};
    \node[circ] at (v3) {3};
    \node[circ] at (v4) {4};
    \node[circ] at (v5) {5};
    \node at (0,-1.30450849719) {$G_{19}$};
    \end{tikzpicture}
    \begin{tikzpicture}[very thick, baseline=(v.base)]
    \coordinate (v) at (0,.09549150281);
    \coordinate (v0) at (0,0);
    \path (v0) + (90:1) coordinate (v1);
    \path (v0) + (162:1) coordinate (v2);
    \path (v0) + (234:1) coordinate (v3);
    \path (v0) + (306:1) coordinate (v4);
    \path (v0) + (378:1) coordinate (v5);
    \draw (v0) -- (v1);
    \draw (v0) -- (v2);
    \draw (v0) -- (v3);
    \draw (v0) -- (v4);
    \draw (v0) -- (v5);
    \draw (v1) -- (v2);
    \draw (v1) -- (v5);
    \draw (v2) -- (v3);
    \draw (v4) -- (v5);
    \node[circ] at (v0) {0};
    \node[circ] at (v1) {1};
    \node[circ] at (v2) {2};
    \node[circ] at (v3) {3};
    \node[circ] at (v4) {4};
    \node[circ] at (v5) {5};
    \node at (0,-1.30450849719) {$G_{20}$};
    \end{tikzpicture}
    \begin{tikzpicture}[very thick, baseline=(v.base)]
    \coordinate (v) at (0,.09549150281);
    \coordinate (v0) at (0,0);
    \path (v0) + (90:1) coordinate (v1);
    \path (v0) + (162:1) coordinate (v2);
    \path (v0) + (234:1) coordinate (v3);
    \path (v0) + (306:1) coordinate (v4);
    \path (v0) + (378:1) coordinate (v5);
    \draw (v0) -- (v1);
    \draw (v0) -- (v2);
    \draw (v0) -- (v3);
    \draw (v0) -- (v4);
    \draw (v0) -- (v5);
    \draw (v1) -- (v2);
    \draw (v1) -- (v5);
    \draw (v2) -- (v3);
    \draw (v2) -- (v5);
    \node[circ] at (v0) {0};
    \node[circ] at (v1) {1};
    \node[circ] at (v2) {2};
    \node[circ] at (v3) {3};
    \node[circ] at (v4) {4};
    \node[circ] at (v5) {5};
    \node at (0,-1.30450849719) {$G_{21}$};
    \end{tikzpicture}
    \begin{tikzpicture}[very thick, baseline=(v.base)]
    \coordinate (v) at (0,.09549150281);
    \coordinate (v0) at (0,0);
    \path (v0) + (90:1) coordinate (v1);
    \path (v0) + (162:1) coordinate (v2);
    \path (v0) + (234:1) coordinate (v3);
    \path (v0) + (306:1) coordinate (v4);
    \path (v0) + (378:1) coordinate (v5);
    \draw (v0) -- (v2);
    \draw (v0) -- (v5);
    \draw (v1) -- (v2);
    \draw (v1) -- (v5);
    \draw (v2) -- (v3);
    \draw (v2) -- (v4);
    \draw (v2) -- (v5);
    \draw (v3) -- (v5);
    \draw (v4) -- (v5);
    \node[circ] at (v0) {0};
    \node[circ] at (v1) {1};
    \node[circ] at (v2) {2};
    \node[circ] at (v3) {3};
    \node[circ] at (v4) {4};
    \node[circ] at (v5) {5};
    \node at (0,-1.30450849719) {$G_{22}$};
    \end{tikzpicture}
    \begin{tikzpicture}[very thick, baseline=(v.base)]
    \coordinate (v) at (0,.09549150281);
    \coordinate (v0) at (0,0);
    \path (v0) + (90:1) coordinate (v1);
    \path (v0) + (162:1) coordinate (v2);
    \path (v0) + (234:1) coordinate (v3);
    \path (v0) + (306:1) coordinate (v4);
    \path (v0) + (378:1) coordinate (v5);
    \draw (v0) -- (v2);
    \draw (v1) -- (v2);
    \draw (v1) -- (v3);
    \draw (v1) -- (v4);
    \draw (v1) -- (v5);
    \draw (v2) -- (v3);
    \draw (v2) -- (v4);
    \draw (v3) -- (v4);
    \draw (v3) -- (v5);
    \draw (v4) -- (v5);
    \node[circ] at (v0) {0};
    \node[circ] at (v1) {1};
    \node[circ] at (v2) {2};
    \node[circ] at (v3) {3};
    \node[circ] at (v4) {4};
    \node[circ] at (v5) {5};
    \node at (0,-1.30450849719) {$G_{23}$};
    \end{tikzpicture}
    \begin{tikzpicture}[very thick, baseline=(v.base)]
    \coordinate (v) at (0,.09549150281);
    \coordinate (v0) at (0,0);
    \path (v0) + (90:1) coordinate (v1);
    \path (v0) + (162:1) coordinate (v2);
    \path (v0) + (234:1) coordinate (v3);
    \path (v0) + (306:1) coordinate (v4);
    \path (v0) + (378:1) coordinate (v5);
    \draw (v0) -- (v1);
    \draw (v0) -- (v2);
    \draw (v0) -- (v3);
    \draw (v0) -- (v4);
    \draw (v0) -- (v5);
    \draw (v1) -- (v2);
    \draw (v1) -- (v5);
    \draw (v2) -- (v3);
    \draw (v3) -- (v4);
    \draw (v4) -- (v5);
    \node[circ] at (v0) {0};
    \node[circ] at (v1) {1};
    \node[circ] at (v2) {2};
    \node[circ] at (v3) {3};
    \node[circ] at (v4) {4};
    \node[circ] at (v5) {5};
    \node at (0,-1.30450849719) {$G_{24}$};
    \end{tikzpicture}
    \begin{tikzpicture}[very thick, baseline=(v.base)]
    \coordinate (v) at (0,.09549150281);
    \coordinate (v0) at (0,0);
    \path (v0) + (90:1) coordinate (v1);
    \path (v0) + (162:1) coordinate (v2);
    \path (v0) + (234:1) coordinate (v3);
    \path (v0) + (306:1) coordinate (v4);
    \path (v0) + (378:1) coordinate (v5);
    \draw (v0) -- (v1);
    \draw (v0) -- (v3);
    \draw (v0) -- (v4);
    \draw (v1) -- (v2);
    \draw (v1) -- (v3);
    \draw (v1) -- (v4);
    \draw (v1) -- (v5);
    \draw (v2) -- (v3);
    \draw (v3) -- (v4);
    \draw (v4) -- (v5);
    \node[circ] at (v0) {0};
    \node[circ] at (v1) {1};
    \node[circ] at (v2) {2};
    \node[circ] at (v3) {3};
    \node[circ] at (v4) {4};
    \node[circ] at (v5) {5};
    \node at (0,-1.30450849719) {$G_{25}$};
    \end{tikzpicture}
    \begin{tikzpicture}[very thick, baseline=(v.base)]
    \coordinate (v) at (0,.09549150281);
    \coordinate (v0) at (0,0);
    \path (v0) + (90:1) coordinate (v1);
    \path (v0) + (162:1) coordinate (v2);
    \path (v0) + (234:1) coordinate (v3);
    \path (v0) + (306:1) coordinate (v4);
    \path (v0) + (378:1) coordinate (v5);
    \draw (v0) -- (v1);
    \draw (v0) -- (v2);
    \draw (v0) -- (v3);
    \draw (v0) -- (v4);
    \draw (v0) -- (v5);
    \draw (v1) -- (v2);
    \draw (v1) -- (v3);
    \draw (v1) -- (v4);
    \draw (v1) -- (v5);
    \draw (v3) -- (v4);
    \node[circ] at (v0) {0};
    \node[circ] at (v1) {1};
    \node[circ] at (v2) {2};
    \node[circ] at (v3) {3};
    \node[circ] at (v4) {4};
    \node[circ] at (v5) {5};
    \node at (0,-1.30450849719) {$G_{26}$};
    \end{tikzpicture}
    \begin{tikzpicture}[very thick, baseline=(v.base)]
    \coordinate (v) at (0,.09549150281);
    \coordinate (v0) at (0,0);
    \path (v0) + (90:1) coordinate (v1);
    \path (v0) + (162:1) coordinate (v2);
    \path (v0) + (234:1) coordinate (v3);
    \path (v0) + (306:1) coordinate (v4);
    \path (v0) + (378:1) coordinate (v5);
    \draw (v0) -- (v1);
    \draw (v0) -- (v2);
    \draw (v0) -- (v3);
    \draw (v0) -- (v4);
    \draw (v1) -- (v2);
    \draw (v1) -- (v3);
    \draw (v1) -- (v4);
    \draw (v1) -- (v5);
    \draw (v2) -- (v3);
    \draw (v3) -- (v4);
    \draw (v4) -- (v5);
    \node[circ] at (v0) {0};
    \node[circ] at (v1) {1};
    \node[circ] at (v2) {2};
    \node[circ] at (v3) {3};
    \node[circ] at (v4) {4};
    \node[circ] at (v5) {5};
    \node at (0,-1.30450849719) {$G_{27}$};
    \end{tikzpicture}
    \begin{tikzpicture}[very thick, baseline=(v.base)]
    \coordinate (v) at (0,.09549150281);
    \coordinate (v0) at (0,0);
    \path (v0) + (90:1) coordinate (v1);
    \path (v0) + (162:1) coordinate (v2);
    \path (v0) + (234:1) coordinate (v3);
    \path (v0) + (306:1) coordinate (v4);
    \path (v0) + (378:1) coordinate (v5);
    \draw (v0) -- (v1);
    \draw (v0) -- (v2);
    \draw (v0) -- (v3);
    \draw (v0) -- (v4);
    \draw (v0) -- (v5);
    \draw (v1) -- (v2);
    \draw (v1) -- (v3);
    \draw (v1) -- (v4);
    \draw (v1) -- (v5);
    \draw (v2) -- (v3);
    \draw (v4) -- (v5);
    \node[circ] at (v0) {0};
    \node[circ] at (v1) {1};
    \node[circ] at (v2) {2};
    \node[circ] at (v3) {3};
    \node[circ] at (v4) {4};
    \node[circ] at (v5) {5};
    \node at (0,-1.30450849719) {$G_{28}$};
    \end{tikzpicture}
    \begin{tikzpicture}[very thick, baseline=(v.base)]
    \coordinate (v) at (0,.09549150281);
    \coordinate (v0) at (0,0);
    \path (v0) + (90:1) coordinate (v1);
    \path (v0) + (162:1) coordinate (v2);
    \path (v0) + (234:1) coordinate (v3);
    \path (v0) + (306:1) coordinate (v4);
    \path (v0) + (378:1) coordinate (v5);
    \draw (v0) -- (v1);
    \draw (v0) -- (v2);
    \draw (v0) -- (v3);
    \draw (v0) -- (v4);
    \draw (v0) -- (v5);
    \draw (v1) -- (v2);
    \draw (v1) -- (v3);
    \draw (v1) -- (v4);
    \draw (v1) -- (v5);
    \draw (v2) -- (v3);
    \draw (v3) -- (v4);
    \draw (v4) -- (v5);
    \node[circ] at (v0) {0};
    \node[circ] at (v1) {1};
    \node[circ] at (v2) {2};
    \node[circ] at (v3) {3};
    \node[circ] at (v4) {4};
    \node[circ] at (v5) {5};
    \node at (0,-1.30450849719) {$G_{29}$};
    \end{tikzpicture}
    \begin{tikzpicture}[very thick, baseline=(v.base)]
    \coordinate (v) at (0,.09549150281);
    \coordinate (v0) at (0,0);
    \path (v0) + (90:1) coordinate (v1);
    \path (v0) + (162:1) coordinate (v2);
    \path (v0) + (234:1) coordinate (v3);
    \path (v0) + (306:1) coordinate (v4);
    \path (v0) + (378:1) coordinate (v5);
    \draw (v0) -- (v1);
    \draw (v0) -- (v2);
    \draw (v0) -- (v3);
    \draw (v0) -- (v4);
    \draw (v0) -- (v5);
    \draw (v1) -- (v2);
    \draw (v1) -- (v5);
    \draw (v2) -- (v3);
    \draw (v2) -- (v4);
    \draw (v2) -- (v5);
    \draw (v3) -- (v5);
    \draw (v4) -- (v5);
    \node[circ] at (v0) {0};
    \node[circ] at (v1) {1};
    \node[circ] at (v2) {2};
    \node[circ] at (v3) {3};
    \node[circ] at (v4) {4};
    \node[circ] at (v5) {5};
    \node at (0,-1.30450849719) {$G_{30}$};
    \end{tikzpicture}
    \begin{tikzpicture}[very thick, baseline=(v.base)]
    \coordinate (v) at (0,.09549150281);
    \coordinate (v0) at (0,0);
    \path (v0) + (90:1) coordinate (v1);
    \path (v0) + (162:1) coordinate (v2);
    \path (v0) + (234:1) coordinate (v3);
    \path (v0) + (306:1) coordinate (v4);
    \path (v0) + (378:1) coordinate (v5);
    \draw (v0) -- (v1);
    \draw (v0) -- (v2);
    \draw (v0) -- (v3);
    \draw (v0) -- (v4);
    \draw (v0) -- (v5);
    \draw (v1) -- (v2);
    \draw (v1) -- (v5);
    \draw (v2) -- (v3);
    \draw (v2) -- (v4);
    \draw (v2) -- (v5);
    \draw (v3) -- (v4);
    \draw (v3) -- (v5);
    \draw (v4) -- (v5);
    \node[circ] at (v0) {0};
    \node[circ] at (v1) {1};
    \node[circ] at (v2) {2};
    \node[circ] at (v3) {3};
    \node[circ] at (v4) {4};
    \node[circ] at (v5) {5};
    \node at (0,-1.30450849719) {$G_{31}$};
    \end{tikzpicture}
    \caption{Minimal forbidden subgraphs for $\D_\infty$.} \label{fig:minimal-forb-1}
\end{figure}

The key observation for the proof of \cref{thm:main1-2} is that for every minimal forbidden subgraph $F$ for $\D_\infty$, there exists an extension family of $F$ disjoint from $\Gl$. We first deal with path extensions and clique extensions.

\begin{lemma} \label{lem:computation-1}
    For every minimal forbidden subgraph $F$ in \cref{fig:minimal-forb-1} for the family $\D_\infty$ of generalized line graphs, and every nonempty vertex subset $A$ of $F$, the path extensions and the clique extensions of $F$ satisfy
    \[
        \lim_{\ell\to\infty} \la_1(F, A, \ell) \le -\la^*
        \quad \text{and} \quad
        \lim_{m\to\infty} \la_1(F, A, K_m) < -\la^*.
    \]
    Moreover, the equality holds in the first inequality if and only if $F = G_4$ and $A \in \sset{\sset{3}, \sset{4}}$.
\end{lemma}

We prove \cref{lem:computation-1} under computer assistance in \cref{sec:numerics}. The next result takes care of path-clique extensions.

\begin{lemma} \label{lem:path-clique-extension}
    Suppose that $A$ is a nonempty vertex subset of a graph $F$ and $\la \ge 2$. If the path extensions of $F$ satisfy
    \[
        \lim_{\ell\to\infty} \la_1(F, A, \ell) < -\la,
    \]
    then there exists $m \in \N^+$ such that the path-clique extensions of $F$ satisfy
    \[
        \la_1(F, A, \ell, K_m) < -\la \text{ for every }\ell \in \N.
    \]
\end{lemma}

\begin{proof}
    Pick $\ell_0 \in \N$ such that $\la_1(F, A, \ell) < -\la$ for every $\ell \ge \ell_0$. Clearly $\la_1(F, A, \ell, K_m) < -\la$ for every $\ell \ge \ell_0$ and every $m \in \N^+$. It suffices to show that for every $\ell \in \sset{0, \dots, \ell_0 - 1}$ there exists $m \in \N^+$ such that $\la_1(F, A, \ell, K_m) < -\la$.
    
    Let $v_0 \dots v_{\ell_0}$ be the path added to $F$ to obtain $(F, A, \ell_0)$, where the vertex $v_0$ is connected to every vertex in $A$. Set $\la_0 := -\la_1(F, A, \ell_0)$, and let $\bx \colon V(F) \cup \sset{v_0, \dots, v_{\ell_0}} \to \R$ be an eigenvector of $(F, A, \ell_0)$ associated with $-\la_0$.

    Now fix $\ell \in \N$ with $\ell < \ell_0$, and let $m \in \N^+$ be determined later. Set $V_\ell := V(F) \cup \sset{v_0, \dots, v_\ell}$, and identify the vertex set of $(F, A, \ell, K_m)$ with $V_\ell \cup V(K_m)$. We abuse notation and write $x_i$ in place of $x_{v_i}$ for $i \in \sset{\ell, \dots, \ell_0}$. Define $\tbx\colon V_\ell \cup V(K_m) \to \R$ by
    \[
        \tx_v = \begin{cases}
            x_v & \text{if }v \in V_\ell; \\
            x_{\ell+1}/m & \text{if }v \in V(K_m).
        \end{cases}
    \]
    
    We claim that $\sum_{v \in V_\ell}x_v^2 > 0$. Indeed, assume for the sake of contradiction that $x_v = 0$ for $v \in V_\ell$. Using $-\la_0x_i = \sum_{u\sim v_i}x_u$ for $i \in \sset{\ell, \dots, \ell_0-1}$, where the sum is taken over all vertices $u$ that are adjacent to the vertex $v_i$ in $(F, A, \ell_0)$, one can prove inductively that $x_i = 0$ for $i \in \sset{\ell + 1, \dots, \ell_0}$, which contradicts the assumption that $\bx$ is a nonzero vector.
    
    Because $\sum_{v \in V_\ell} x_v^2 > 0$, clearly $\tbx$ is a nonzero vector. We compute
    \[
        \tbx^\intercal \tbx = \sum_{v \in V_\ell}x_v^2 + m(x_{\ell + 1}/m)^2.
    \]
    Moreover we can compute $\tbx^\intercal A_{(F, A, \ell, K_m)} \tbx$ as follows
    \[
        \tbx^\intercal A_{(F, A, \ell, K_m)} \tbx = \sum_{u, v \in V_\ell\colon u \sim v} x_ux_v + 2x_\ell x_{\ell+1} + m(m-1)(x_{\ell+1}/m)^2.
    \]
    Since $\bx$ is an eigenvector of $(F, A, \ell_0)$ associated with $-\la_0$, we obtain that
    \[
        \sum_{u, v \in V_\ell\colon u \sim v} x_ux_v + x_\ell x_{\ell + 1} = \sum_{v \in V_\ell}x_v\sum_{u \sim v}x_u = -\la_0 \sum_{v \in V_\ell}x_v^2.
    \]
    Thus $\tbx^\intercal A_{(F, A, \ell, K_m)} \tbx$ can be simplified to
    \[
        \tbx^\intercal A_{(F, A, \ell, K_m)} \tbx = -\la_0 \sum_{v \in V_\ell}x_v^2 + x_\ell x_{\ell+1} + m(m-1)(x_{\ell+1}/m)^2.
    \]
    The Rayleigh principle says that $\la_1(F, A, \ell, K_m)$ is at most
    \[
        \frac{-\la_0 \sum_{v \in V_\ell} x_v^2 + x_\ell x_{\ell+1} + m(m-1)(x_{\ell+1}/m)^2}{\sum_{v \in V_\ell}x_v^2 + m(x_{\ell + 1}/m)^2},
    \]
    which, as $m \to \infty$, approaches
    \[
        -\la_0 + \frac{(x_\ell + x_{\ell+1})x_{\ell+1}}{\sum_{v \in V_\ell}x_v^2}.
    \]
    Here we used the above claim that the denominator in the limit is positive.
    
    Recall that $\la_0 = -\la_1(F, A, \ell_0) > \la \ge 2$. It suffices to show that $(x_\ell + x_{\ell+1})x_{\ell+1} \le 0$. In fact, we prove inductively that $(x_i + x_{i+1})x_{i+1} \le 0$ for $i \in \sset{\ell, \dots, \ell_0 - 1}$. The base case where $i = \ell_0 - 1$ follows immediately from $-\la_0x_{\ell_0} = x_{\ell_0-1}$ and $\la_0 > 2$. For the inductive step, using $-\la_0x_{i+1} = x_i + x_{i+2}$ and $\la_0 > 2$, we obtain
    \begin{multline*}
        (x_i + x_{i+1})x_{i+1} = (-\la_0x_{i+1} - x_{i+2} + x_{i+1})x_{i+1} = -(\la_0 - 2)x_{i+1}^2 - (x_{i+1}+x_{i+2})x_{i+1} \\
        \le -(x_{i+1}+x_{i+2})x_{i+1} = -(x_{i+1}+x_{i+2})^2 + (x_{i+1}+x_{i+2})x_{i+2} \le (x_{i+1}+x_{i+2})x_{i+2},
    \end{multline*}
    which is nonpositive by the inductive hypothesis.
\end{proof}

We are ready to prove \cref{thm:main1-2}.

\begin{proof}[Proof of \cref{thm:main1-2}]
    Suppose that $\la \in [2, \la^*)$. Let $\F$ denote the set of minimal forbidden subgraphs for the family $\D_\infty$. Combining \cref{lem:computation-1,lem:path-clique-extension}, we choose $\ell, m \in \N^+$ such that for every $F \in \F$, the extension family $\X(F, \ell, m)$ is disjoint from $\Gl$. From \cref{lem:hoffman}(\ref{lem:s}), we know that $S_5 \not\in \Gl$. In particular, no graph in $\Gl$ contains any member in the following family
    \begin{equation} \label{eqn:main1-2}
        \sset{S_5} \cup \bigcup_{F \in \F} \X(F, \ell, m)
    \end{equation}
    as a subgraph. Using \cref{lem:forbid-extensions}, we obtain $N \in \N$ such that for every connected graph $G$ with more than $N$ vertices, if no member in \eqref{eqn:main1-2} is a subgraph of $G$, then neither is any $F \in \F$, and so $G$ is a generalized line graph and is in $\G(2)$ by \cref{thm:forb-char-generalized-line,thm:generalized-line-smallest-ev}. This implies that every connected graph in $\Gl \setminus \G(2)$ has at most $N$ vertices.
\end{proof}

We end this subsection with a remark on an alternative proof of \cref{thm:main1-2}. Instead of working with the family $\D_\infty$, one should be able to establish \cref{lem:computation-1} for every minimal forbidden subgraph for $\G(2)$, and then prove \cref{thm:main1-2} similarly. This alternative approach is more direct as it does not rely on generalized line graphs. However the drawback is obvious --- there are $1812$ minimal forbidden subgraphs for $\G(2)$. To the best of our knowledge, there is no easily accessible database for these $1812$ graphs; for example, a complete list of these graphs was printed to microfiche in \cite{BN92}, and a list of only $92$ of these graphs up to $8$ vertices was available in \cite[Fig.~2.4 and Table~A1.2]{CRS04}. One certainly can implement the reasonably fast algorithm in \cite{BN92} to enumerate the minimal forbidden subgraphs for $\G(2)$. However as we strive to keep the computer-assisted part of the proof to a minimum, we work with the $31$ minimal forbidden subgraphs for $\D_\infty$.

In fact, the family $\G(2)$ and its subfamily $\D_\infty$ are interchangeable when both families are restricted to sufficiently large connected graphs because of a characterization of $\G(2)$ due to Cameron, Goethals, Seidel, and Shult. To state their characterization, we introduce the following definitions.

\begin{definition}[Representation of graphs and root systems $D_n$ and $E_8$] \label{def:dn-e8}
    Given a graph $G$ and $V \subseteq \R^n$, we say that $G$ is \emph{represented} by $V$ if the Gram matrix of $V$ is equal to $A_G + 2I$, where $A_G$ is the adjacency matrix of $G$. The \emph{root systems} $D_n$ and $E_8$ are defined by the standard basis $e_1, \dots e_n$ of $\R^n$ as follows
    \[
        D_n := \dset{\eps_1 \bm{e}_i + \eps_2 \bm{e}_j}{\eps_1, \eps_2 = \pm 1, 1 \le i < j \le n}, \quad E_8 := D_8 \cup \dset{\frac12\sum_{i=1}^8 \eps_i \bm{e}_i}{\eps_i = \pm 1, \prod_{i=1}^8 \eps_i = 1}.
    \]
\end{definition}

\begin{theorem}[Theorems~4.2, 4.3 and 4.10 of Cameron et al.~\cite{CGSS76}] \label{thm:cgss}
    For every connected graph $G$, the smallest eigenvalue of $G$ is at least $-2$ if and only if $G$ is represented by a subset of $D_n$ or $E_8$. Moreover
    \begin{enumerate}[label=(\alph*)]
        \item a graph (not necessarily connected) is represented by a subset of $D_n$ if and only if it is a generalized line graph; and \label{thm:cgss-dn}
        \item a graph represented by a subset of $E_8$ has at most $36$ vertices, and its maximum degree is at most $28$. \qed
    \end{enumerate}
\end{theorem}

In particular, \cref{thm:cgss} implies that every connected graph in $\G(2)$ with more than $36$ vertices is a generalized line graph.

\subsection{Proof of \texorpdfstring{\cref{thm:main1}}{the first main theorem} for \texorpdfstring{$\la \ge \la^*$}{λ ≥ λ*}}

Suppose that $\sset{F_1, \dots, F_n}$ is a finite forbidden subgraph characterization of $\Gl$. Because every graph that is not in $\Gl$ contains $F_i$ as a subgraph for some $i \in \sset{1, \dots, n}$, no graph has its smallest eigenvalue in the open interval $(\max\dset{\la_1(F_i)}{i \in \sset{1, \dots, n}},-\la)$. Recall that $\La_1$ consists of $\la \in \R$ such that $-\la$ is the smallest eigenvalue of some graph. The contrapositive of the above observation says the following.

\begin{proposition} \label{lem:limit-points}
    Let $\lim_+ \La_1 := \dset{\la \in \R}{(\la, \la + \eps) \cap \La_1 \neq \varnothing \text{ for every }\eps > 0}$ be the set of right-sided limit points of $\La_1$. The family $\Gl$ does not have a finite forbidden subgraph characterization for any $\la \in \lim_+ \La_1$. \qed
\end{proposition}

In fact, we prove that $\La_1$ is dense in $(\la^*, \infty)$, from which the first main theorem and its corollary for $\la \ge \la^*$ follow.

\begin{theorem} \label{thm:main1-3}
    For every $\la > \la^*$, there exist graphs $G_1, G_2, \dots$ such that $\lim_{n\to\infty} \la_1(G_n) = -\la$.
\end{theorem}

\begin{proof}[Proof of \cref{thm:main1} for $\la \ge \la^*$]
    \cref{thm:main1-3} implies that $\lim_+\La_1 \supseteq [\la^*, \infty)$, which implies through \cref{lem:limit-points} that $\Gl$ has no finite forbidden subgraph characterization for any $\la \ge \la^*$.
\end{proof}

\begin{proof}[Proof of \cref{cor:main1} for $\la \ge \la^*$]
    It follows immediately from \cref{thm:main1-3} that $-\la$ is a limit point of the set of smallest eigenvalues of graphs for $\la \ge \la^*$.
\end{proof}

A large chunk of \cref{thm:main1-3} is essentially established by Shearer~\cite{S89}, who proved that the set of spectral radii of all graphs is dense in $(\la', \infty)$, where $\la' = \sqrt{2+\sqrt{5}}$. As was pointed out in \cite{D89}, Shearer actually proved that the set of spectral radii of all caterpillar trees\footnote{A caterpillar tree is a tree in which all the vertices are within distance 1 of a central path.} is dense in $(\la', \infty)$ already. Since a caterpillar tree is bipartite, we rephrase Shearer's result in terms of smallest eigenvalues.

\begin{theorem}[Shearer~\cite{S89}; cf.\ Theorem 3 of Doob~\cite{D89}] \label{thm:shearer}
    For every $\la \ge \la'$, there exist caterpillar trees $G_1, G_2, \dots$ such that $\lim_{n\to\infty} \la_1(G_n) = -\la$. \qed
\end{theorem}

To fill the gap between $\la^* \approx 2.01980$ and $\la' \approx 2.05817$, we use the following graphs.

\begin{definition}[Rowing graphs]
    Given a sequence $(a_1, \dots, a_n)$ of natural numbers, a \emph{rowing graph} $R(a_1, \dots, a_n)$ is obtained from the path $v_{-2}v_{-1}v_{0}v_{1}\dots v_{n}$ (called the \emph{central path}) by attaching a vertex (called the \emph{coxswain}) to $v_0$, and attaching a clique of order $a_i$ to both $v_{i-1}$ and $v_i$ for every $i \in \sset{1, \dots, n}$. See \cref{fig:rowing-graph} for an example of a rowing graph.
\end{definition}

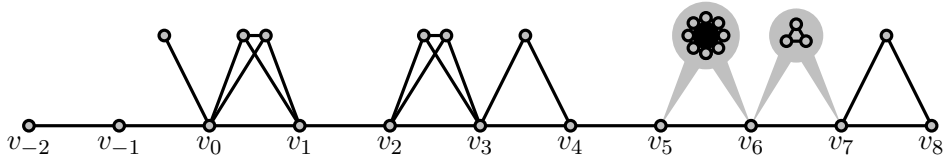
\begin{figure}[t]
    \centering
    \begin{tikzpicture}[very thick, scale=0.6]
        \draw (0,3);
        \draw (-1,2) node[vertex]{} -- (0,0);
        \draw (0.75,2) -- (1.25,2);
        \draw (0,0) -- (0.75,2) node[vertex]{} -- (2,0);
        \draw (0,0) -- (1.25,2) node[vertex]{} -- (2,0);
        
        \draw (4.75,2) -- (5.25,2);
        \draw (4,0) -- (4.75,2) node[vertex]{} -- (6,0);
        \draw (4,0) -- (5.25,2) node[vertex]{} -- (6,0);
        \draw (6,0) -- (7,2) node[vertex]{} -- (8,0);

        \draw[fill=litegray, draw=none] (10,0) -- (11,1.5) -- (12, 0) -- (11,2.7) -- cycle;

        \draw[fill=litegray, draw=none] (11, 2) circle (0.75);
        \draw[fill=black, draw=none] (11, 2) circle (0.3);

        \draw (11,1.608) node[vertex]{} -- ++(22.5:0.3) node[vertex]{} -- ++(67.5:0.3) node[vertex]{} -- ++(112.5:0.3) node[vertex]{} -- ++(157.5:0.3) node[vertex]{} -- ++(202.5:0.3) node[vertex]{} -- ++(247.5:0.3) node[vertex]{} -- ++(292.5:0.3) node[vertex]{} -- ++(337.5:0.3) node[vertex]{};

        \draw[fill=litegray, draw=none] (12,0) -- (13,1.6) -- (14,0) -- (13,2.6) -- cycle;

        \draw[fill=litegray, draw=none] (13, 2) circle (0.6);
        
        \draw (13,2.25) node[vertex]{} -- ++(240:0.433) node[vertex]{} -- ++(0:0.433) node[vertex]{} -- cycle;

        \draw (14,0) -- (15,2) node[vertex]{} -- (16,0);
        
        \draw (-4,0) node[vertex]{} -- (-2,0) node[vertex]{} -- (0,0) node[vertex]{} -- (2,0) node[vertex]{} -- (4,0) node[vertex]{} -- (6,0) node[vertex]{} -- (8,0) node[vertex]{} -- (10,0) node[vertex]{} -- (12,0) node[vertex]{} -- (14,0) node[vertex]{} -- (16,0) node[vertex]{};
        \draw (-1,2) node[above]{$v_c$};
        \draw (-4,0) node[below]{$v_{-2}$};
        \draw (-2,0) node[below]{$v_{-1}$};
        \draw (0,0) node[below]{$v_{0}$};
        \draw (2,0) node[below]{$v_1$};
        \draw (4,0) node[below]{$v_2$};
        \draw (6,0) node[below]{$v_3$};
        \draw (8,0) node[below]{$v_4$};
        \draw (10,0) node[below]{$v_5$};
        \draw (12,0) node[below]{$v_6$};
        \draw (14,0) node[below]{$v_7$};
        \draw (16,0) node[below]{$v_8$};
    \end{tikzpicture}
    \caption{A schematic drawing of the rowing graph $R(2,0,2,1,0,8,3,1)$.} \label{fig:rowing-graph}
\end{figure}

We consider rowing graphs for the following two heuristics. On the one hand, a rowing graph is almost a line graph, whose smallest eigenvalue is at least $-2$ --- when the coxwain is removed from a rowing graph, it becomes the line graph of a caterpillar tree. On the other hand, the rowing graph $R(a_1, \dots, a_n)$ contains $E_{2,n}$ (see \cref{fig:e2n}), whose smallest eigenvalue decreases to $-\la^*$ as $n \to \infty$.

We adopt the shorthand
\[
    (a_1, \dots, a_{k-1}, 0^{(\ell)}, a_k, \dots, a_n) := (a_1, \dots, a_{k-1}, \underbrace{0, \dots, 0}_{\ell}, a_k, \dots, a_n).
\]

\begin{lemma} \label{lem:rowing}
    Rowing graphs satisfy the following properties.
    \begin{enumerate}[label=(\alph*)]
        \item \label{lem:rowing-0} The smallest eigenvalue of a rowing graph is at least $-1-\sqrt{2}$.
        \item \label{lem:rowing-a} There exists $a_1 \in \N$ such that $\la_1(R(a_1)) < -\la' \approx -2.05817$.
        \item \label{lem:rowing-b} For every $\eps > 0$ there exists $\ell \in \N^+$ such that
            \begin{multline*}
                \la_1(R(a_1, \dots, a_{n_1}, 0^{(\ell)}, b_{1}, \dots, b_{n_2})) > \la_1(R(a_1, \dots, a_{n_1}, 0^{(\ell)})) - \eps\\
                \text{ for every }n_1, n_2 \in \N, (a_1, \dots, a_{n_1}) \in \N^{n_1} \text{ and } (b_1, \dots, b_{n_2}) \in \N^{n_2}.  
            \end{multline*}
        \item \label{lem:rowing-c} For every $n, \ell \in \N^+$ and $(a_1, \dots, a_n) \in \N^n$ with $a_n > 0$, if $\la_1(R(a_1, \dots, a_n, 0^{(\ell)})) \le -2$, then for every $\eps > 0$ there exists $a_{n+1} \in \N^+$ such that
            \[
                \la_1(R(a_1, \dots, a_{n-1}, a_n-1, a_{n+1}, 0^{(\ell - 1)})) < \la_1(R(a_1, \dots, a_n, 0^{(\ell)})) + \eps.
            \]
        \item \label{lem:rowing-d} For every $\eps > 0$, there exists $m \in \N^+$ such that for every $n \ge m$ and every $(a_1, \dots, a_n) \in \N^n$ there exists $k \in \sset{1, \dots, m}$ such that
            \[
                \la_1(R(a_1, \dots, a_{k-1}, 0, a_k, \dots, a_n)) < \la_1(R(a_1, \dots, a_n)) + \eps.
            \]
    \end{enumerate}
\end{lemma}

In the proof of \cref{lem:rowing}, we work with vectors $\bx$ on the vertex set of a rowing graph, whose coxswain and central path are denoted by $v_c$ and $v_{-2}v_{-1} \dots$, and we abuse notation and write $x_c$ and $x_i$ in place of $x_{v_c}$ and $x_{v_i}$ respectively.

\begin{proof}[Proof of \ref{lem:rowing-0}]
    Set $\ba = (a_1, \dots, a_n)$ and $\la = \la_1(R(\ba))$. Let $v_{-2} \dots v_n$ denote the central path of $R(\ba)$, and let $v_c$ denote the coxswain attached to $v_0$. Let $\bx\colon V(R(\ba)) \to \R$ be a unit eigenvector of $R(\ba)$ associated with $\la$. Clearly we have $\la x_c = x_0$, which implies that
    \begin{equation} \label{eqn:rowing-0}
        x_c^2 = \frac{x_c^2 + x_0^2}{1+\la^2} \le \frac{1}{1+\la^2}.
    \end{equation}
    Let $L$ be the rowing graph $R(\ba)$ with $v_c$ removed, and let $\bx_L$ be $\bx$ restricted to $V(L)$. Notice that $L$ is a line graph of a caterpillar tree, and so $\la_1(L) \ge -2$ (see \cref{footnote:line-graph} on \cpageref{footnote:line-graph}). Finally, we bound the smallest eigenvalue of $R(\ba)$ as follows:
    \begin{multline*}
        \la = \bx^\intercal A_{R(\ba)} \bx = \bx_L^\intercal A_L \bx_L + 2x_cx_0 \ge -2\bx_L^\intercal \bx_L + 2x_cx_0 \\
        = -2(1-x_c^2) + 2\la x_c^2 = -2+(2+2\la)x_c^2 \stackrel{\eqref{eqn:rowing-0}}{\ge} -2 + \frac{2+2\la}{1+\la^2},
    \end{multline*}
    which implies that $\la \ge -1-\sqrt{2}$. In the above inequality, we assumed that $\la \le -1$ which follows from the fact that $R(\ba)$ has an edge.
\end{proof}

\begin{proof}[Proof of \ref{lem:rowing-a}]
    Let $\bx$ be the vector that assigns $1$ to $v_{-2}$, $-2$ to $v_{-1}$, $4$ to $v_0$, $0$ to $v_1$, $-2$ to the coxswain, and $-4/a_1$ to every vertex in the clique of size $a_1$. The Rayleigh principle says that $\la_1(R(a_1))$ is at most
    \[
        \frac{\bx^\intercal A_{R(a_1)} \bx}{\bx^\intercal \bx} = \frac{-2(1 \cdot 2 + 2 \cdot 4 + 4 \cdot 2 + a_1 \cdot 4 \cdot (4/a_1)) + a_1(a_1-1) \cdot (-4/a_1)^2}{1^2 + (-2)^2 + 4^2 + (-2)^2 + a_1 \cdot (-4/a_1)^2},
    \]
    which approaches $-52/25 = -2.08$ as $a_1 \to \infty$.
\end{proof}

\begin{proof}[Proof of \ref{lem:rowing-b}]
    Take $\ell = \lceil 2/\eps \rceil + 5$. Let $v_{-2} \dots v_{n_1 + \ell + n_2}$ denote the central path of the rowing graph $R(\ba, 0^{(\ell)}, \bb)$, where $\ba = (a_1, \dots, a_{n_1})$ and $\bb = (b_1, \dots, b_{n_2})$, and let $\bx\colon V(R(\ba,0^{(\ell)},\bb)) \to \R$ be a unit eigenvector of $R(\ba,0^{(\ell)},\bb)$ associated with the smallest eigenvalue. Choose $k \in \sset{0, \dots, \ell - 1}$ such that $x_{n_1+k}x_{n_1+k+1}$ reaches the minimum in absolute value. In particular, using the inequality $\abs{x_{n_1+i}x_{n_1+i+1}} \le (x_{n_1+i}^2 + x_{n_1+i+1}^2)/2$, we obtain
    \begin{equation} \label{eqn:rowing-b}
        \abs{x_{n_1+k}x_{n_1+k+1}} \le \frac{1}{\ell}\sum_{i = 0}^{\ell - 1}\abs{x_{n_1+i}x_{n_1+i+1}} \le \frac{1}{\ell} \sum_{i=0}^{\ell}x_{n_1+i}^2 \le \frac{1}{\ell} < \frac{\eps}{2}.
    \end{equation}
    
    Notice that removing the edge $v_{n_1+k} v_{n_1+k+1}$ disconnects $R(\ba,0^{(\ell)},\bb)$ into two subgraphs, one of which is $R(\ba, 0^{(k)})$, while the other is a line graph, denoted $L$, of a caterpillar tree. Clearly
    \begin{equation} \label{eqn:rowing-b-1}
        \la_1(R(\ba, 0^{(k)})) \ge \la_1(R(\ba, 0^{(\ell)})).
    \end{equation}
    As $\ell \ge 6$, the graph $\widetilde{E}_8$ in \cref{fig:dynkin} is a proper subgraph of $R(\ba, 0^{(\ell)})$, and so $\la_1(R(\ba, 0^{(\ell)})) < -2$. Together with $\la_1(L) \ge -2$ (see \cref{footnote:line-graph} on \cpageref{footnote:line-graph}), we obtain
    \begin{equation} \label{eqn:rowing-b-2}
        \la_1(L) > \la_1(R(\ba, 0^{(\ell)})).
    \end{equation}
    Let $\bx_R$ and $\bx_L$ be the unit eigenvector $\bx$ restricted to $V(R(\ba,0^{(k)}))$ and $V(L)$. Finally, we bound the smallest eigenvalue of $R(\ba,0^{(\ell)},\bb)$ as follows:
    \begin{align*}
        \la_1(R(\ba,0^{(\ell)},\bb)) & \stackrel{\phantom{(\ref{eqn:rowing-b},\ref{eqn:rowing-b-1},\ref{eqn:rowing-b-2})}}{=} \bx^\intercal A_{R(\ba,0^{(\ell)},\bb)}\bx \\
        & \stackrel{\phantom{(\ref{eqn:rowing-b},\ref{eqn:rowing-b-1},\ref{eqn:rowing-b-2})}}{=} \bx_R^\intercal A_{R(\ba,0^{(k)})} \bx_R + 2x_{n_1+k}x_{n_1+k+1} + \bx_L^\intercal A_L\bx_L \\
        & \stackrel{\phantom{(\ref{eqn:rowing-b},\ref{eqn:rowing-b-1},\ref{eqn:rowing-b-2})}}{\ge} \la_1(R(\ba,0^{(k)})) \bx_R^\intercal \bx_R + 2x_{n_1+k}x_{n_1+k+1} + \la_1(L)\bx_L^\intercal \bx_L \\
        & \stackrel{(\ref{eqn:rowing-b},\ref{eqn:rowing-b-1},\ref{eqn:rowing-b-2})}{>} \la_1(R(\ba,0^{(\ell)}))(\bx_R^\intercal \bx_R + \bx_L^\intercal \bx_L) - \eps \\
        & \stackrel{\phantom{(\ref{eqn:rowing-b},\ref{eqn:rowing-b-1},\ref{eqn:rowing-b-2})}}{=} \la_1(R(\ba, 0^{(\ell)})) - \eps. \qedhere
    \end{align*}
\end{proof}

\begin{proof}[Proof of \ref{lem:rowing-c}]
    Set $\ba = (a_1, \dots, a_n, 0^{(\ell)})$ and $\la = \la_1(R(\ba))$. Let $\bx\colon V(R(\ba)) \to \R$ be an eigenvector of $R(\ba)$ associated with $\la$. Let $v_{-2}\dots v_{n + \ell}$ denote the central path of the rowing graph $R(\ba)$. In $R(\ba)$, let $K$ denote the clique of size $a_n$ attached to both $v_{n-1}$ and $v_n$, and pick an arbitrary vertex $u$ from $K$. We clearly have
    \begin{equation} \label{eqn:rowing-c-0}
        \begin{split}
            \la x_u & = x_{n-1} + \sigma + x_n, \\
            \la x_n & = x_{n-1} + \sigma + x_u + x_{n+1},
        \end{split}
    \end{equation}
    where $\sigma = \sum_{u' \in V(K) \setminus \sset{u}} x_{u'}$.The above identities imply that $\la(x_u - x_n) = x_n - x_u - x_{n+1}$, and so
    \begin{equation} \label{eqn:rowing-c-1}
        x_u = x_n - \frac{x_{n+1}}{1 + \la}.
    \end{equation}
    
    Let $R(\tba)$ be the rowing graph obtained from $R(\ba)$ by removing $u$ and attaching a clique, denoted $\tilde{K}$, to both $v_n$ and $v_{n+1}$, where $\tba = (a_1, \dots, a_n-1,a_{n+1},0^{(\ell-1)})$, and the order $a_{n+1}$ of $\tilde{K}$ is chosen later. In particular, $V(R(\tba)) = V(R(\ba)) \setminus \sset{u} \cup V(\tilde{K})$.
    
    We define a vector $\tbx\colon V(R(\tba))\to\R$ as follows:
    \[
        \tx_v = \begin{cases}
            x_n + x_u & \text{if } v = v_n; \\
            -(x_n + x_u + x_{n+1})/a_{n+1} & \text{if } v \in V(\tilde{K}); \\
            x_v & \text{otherwise}.
        \end{cases}
    \]
    Because $\bx$ is a nonzero vector, one can check that so is $\tbx$. We compute
    \[
        \tbx^\intercal \tbx = \bx^\intercal \bx - x_n^2 - x_u^2 + (x_n + x_u)^2 + \frac{(x_n + x_u + x_{n+1})^2}{a_{n+1}}.
    \]
    Moreover we can compare $\tbx^\intercal A_{R(\tba)} \tbx$ and $\bx^\intercal A_{R(\ba)} \bx$ as follows
    \begin{multline*}
        \tbx^\intercal A_{R(\tba)} \tbx = \bx^\intercal A_{R(\ba)} \bx - 2x_u\left(x_{n-1} + \sigma + x_n\right) + 2x_u\left(x_{n-1} + \sigma + x_{n+1} \right) \\
        - 2(x_n + x_u + x_{n+1})^2 + a_{n+1}(a_{n+1}-1)\left(\frac{x_n + x_u + x_{n+1}}{a_{n+1}}\right)^2,
    \end{multline*}
    which simplifies via \eqref{eqn:rowing-c-0} to
    \[
        \tbx^\intercal A_{R(\tba)} \tbx = \la \bx^\intercal \bx - 2\la x_u^2 + 2x_u (\la x_n - x_u) - \left(1 + \frac{1}{a_{n+1}}\right)(x_n + x_u + x_{n+1})^2.
    \]
    The Rayleigh principle says that $\la_1(R(\tba))$ is at most
    \[
        \frac{\tbx^\intercal A_{R(\tba)} \tbx}{\tbx^\intercal \tbx} = \frac{\la \bx^\intercal \bx - 2\la x_u^2 + 2x_u (\la x_n - x_u) - \left(1 + 1/a_{n+1}\right)(x_n + x_u + x_{n+1})^2}{\bx^\intercal \bx - x_n^2 - x_u^2 + (x_n + x_u)^2 + (x_n + x_u + x_{n+1})^2/a_{n+1}},
    \]
    which, as $a_{n+1} \to \infty$, approaches
    \[
        \frac{\la \bx^\intercal \bx - 2\la x_u^2 + 2x_u(\la x_n - x_u) - (x_n + x_u + x_{n+1})^2}{\bx^\intercal \bx - x_n^2 - x_u^2 + (x_n + x_u)^2}.
    \]
    Here we assumed that the denominator in the limit is nonzero. Indeed, suppose on the contrary that $\bx^\intercal \bx = x_n^2 + x_u^2 - (x_n + x_u)^2 = -2x_nx_u$. Because $-2x_nx_u \le x_n^2 + x_u^2 \le x_n^2 + x_u^2 + x_{n+1}^2 \le \bx^\intercal \bx$, it must be the case that $x_n + x_u = 0$ and $x_{n+1} = 0$. In view of \eqref{eqn:rowing-c-1}, we have $x_n = x_u = 0$ and hence $\bx = \bm{0}$, which is a contradiction.
    
    It suffices to prove that
    \[
    \la \bx^\intercal \bx - 2\la x_u^2 + 2x_u(\la x_n - x_u) - (x_n + x_u + x_{n+1})^2 \le \la\left( \bx^\intercal \bx - x_n^2 - x_u^2 + (x_n + x_u)^2 \right),
    \]
    which is equivalent to
    \[
        (x_n+x_u+x_{n+1})^2 + 2(1+\la)x_u^2 \ge 0.
    \]
    Using \eqref{eqn:rowing-c-1}, we know that the left hand side of the last inequality is equal to
    \begin{multline*}
        \left(2x_n + \frac{\la}{1+\la}x_{n+1}\right)^2 + 2(\la+1)\left(x_n-\frac{x_{n+1}}{1+\la}\right)^2 \\
        = (2\la + 6)x_n^2 - \frac{4}{1+\la}x_nx_{n+1}+\left(1+\frac{1}{(1+\la)^2}\right)x_{n+1}^2.
    \end{multline*}
    From \cref{lem:rowing}\ref{lem:rowing-0}, we know that $\la \ge -1-\sqrt{2}$ in addition to the assumption that $\la \le -2$. One can check
    \[
        1 + \frac{1}{(1+\la)^2} \ge 2\la + 6 \ge \frac{4}{1+\la}\cdot \frac{1}{\la} > 0, \quad \text{for }\la \in [-1-\sqrt2,-2].
    \]
    
    We are left to prove $x_n^2 - \la x_nx_{n+1} + x_{n+1}^2 \ge 0$. In fact, we prove inductively that $x_i^2 - \la x_ix_{i+1} + x_{i+1}^2 \ge 0$ for $i \in \sset{n, \dots, n + \ell - 1}$. The base case where $i = n + \ell - 1$ follows immediately from $\la x_{n+\ell} = x_{n+\ell-1}$. For the inductive step, using $\la x_{i+1} = x_i + x_{i+2}$, we obtain
    \[
        x_i^2 - \la x_i x_{i+1} + x_{i+1}^2 = (\la x_{i+1}-x_{i+2})^2 - \la (\la x_{i+1}-x_{i+2})x_{i+1} + x_{i+1}^2 = x_{i+1}^2 - \la x_{i+1}x_{i+2} + x_{i+2}^2,
    \]
    which is nonnegative by the inductive hypothesis.
\end{proof}

\begin{proof}[Proof of \ref{lem:rowing-d}]
    Take $m = \lceil 5/\eps \rceil + 1$. Let $v_{-2} \dots v_n$ denote the central path of the rowing graph $R(\ba)$, where $\ba = (a_1, \dots, a_n)$, and let $\bx\colon V(R(\ba)) \to \R$ be a unit eigenvector associated with the smallest eigenvalue of $R(\ba)$. Choose $k \in \sset{1, \dots, m}$ such that $x_{k-1}$ reaches the minimum in absolute value. In particular,
    \begin{equation} \label{eqn:rowing-d}
        x_{k-1}^2 \le \frac{1}{m} \sum_{i=0}^{m-1}x_i^2 \le \frac{1}{m} < \frac{\eps}{5}.
    \end{equation}
    
    Let $v_{-2} \dots v_{k-1}v_* v_{k} \dots v_n$ denote the central path of $R(\tba)$, where $\tba = (a_1, \dots, a_{k-1},0,a_{k}, \dots, a_n)$. We naturally view the vertex set of $R(\tba)$ as $V(R(\ba)) \cup \{ v_*\}$, and we extend the unit eigenvector $\bx\colon V(R(\ba)) \to \R$ to $\tbx\colon V(R(\tba)) \to \R$ by setting $\tilde{x}_* = x_{k-1}$. The Rayleigh principle says that $\la_1(R(\tba))$ is at most
    \begin{multline*}
        \frac{\tbx^\intercal A_{R(\tba)} \tbx}{\tbx^\intercal \tbx} = \frac{\bx^\intercal A_{R(\ba)} \bx + 2x_{k-1}^2}{\bx^\intercal \bx + x_{k-1}^2} = \frac{\la_1(R(\ba)) + 2x_{k-1}^2}{1 + x_{k-1}^2} \\
        = \la_1(R(\ba)) + \frac{(2-\la_1(R(\ba)))x_{k-1}^2}{1+x_{k-1}^2} \le \la_1(R(\ba)) + (2-\la_1(R(\ba)))x_{k-1}^2.
    \end{multline*}
    From \cref{lem:rowing}\ref{lem:rowing-0}, we obtain
    \begin{equation*}
        (2-\la_1(R(\ba)))x_{k-1}^2 \le (2 + (1 + \sqrt2))x_{k-1}^2 \stackrel{\eqref{eqn:rowing-d}}{<} \eps. \qedhere
    \end{equation*}
\end{proof}

We now have all of the ingredients needed to establish \cref{thm:main1-3}.

\begin{proof}[Proof of \cref{thm:main1-3}]
    In view of \cref{thm:shearer}, we may assume that $\la \in (\la^*, \la')$. Fix $\eps > 0$. We assume for the sake of contradiction that no rowing graph has its smallest eigenvalue in $(-\la - \eps, -\la)$ because otherwise we are done. Define
    \begin{gather*}
        S := \dset{(a_1, \dots, a_n) \in \N^n}{n \in \N^+ \text{ and } a_n > 0}, \\
        A := \dset{(a_1, \dots, a_n) \in S}{\la_1(R(a_1, \dots, a_n, 0^{(\ell)})) < -\la \text{ for some }\ell \in \N}.
    \end{gather*}
    
    \begin{claim*}
        For every $m \in \N^+$ there exists $(a_1, \dots, a_n) \in A$ with exactly $m$ nonzero entries such that
        \begin{equation} \label{eqn:main1-3-claim}
            (a_1, \dots, a_{k-1}, 0, a_k, \dots, a_n) \not\in A\text{ for every }k \in \sset{1, \dots, n}.
        \end{equation}
    \end{claim*}

    \begin{claimproof}[Proof of Claim]
        We order the elements in $S$ as follows:
        \begin{multline*}
            (a_1, \dots, a_{n_1}) \prec (b_1, \dots, b_{n_2}) \text{ if and only if}\\
            (a_1, \dots, a_{n_1}, 0^{(n_2)}) \text{ strictly precedes } (b_1, \dots, b_{n_2}, 0^{(n_1)}) \text{ lexicographically}.
        \end{multline*}
        Alternatively, $(a_1, \dots, a_{n_1}) \prec (b_1, \dots, b_{n_2})$ if and only if there exists $\delta > 0$ such that $\sum_{i=1}^{n_1} a_ix^i < \sum_{i=1}^{n_2} b_ix^i$ for every $x \in (0, \delta)$. We shall repeatedly use the fact that for every $\ell \in \N^+$ the set
        \[
            S_\ell := \dset{\ba \in S}{\text{the length of }\ba \text{ is at most }\ell}
        \]
        does not contain an infinite descending chain, hence $(S_\ell, \preceq)$ is well-founded, that is, every nonempty subset of $S_\ell$ has a minimal element.\footnote{Assuming the axiom of dependent choice, a weak form of the axiom of choice, $(S_\ell, \preceq)$ is well-founded if it contains no countable infinite descending chains.} This fact can be established by a simple induction on $\ell$.
        
        In particular, the ordering on $S$ implies that $(a_1, \dots a_{k-1}, 0, a_k, \dots, a_n) \prec (a_1, \dots, a_n)$ when both sequences are in $S$ and $k \in \sset{1, \dots, n}$. Thus it suffices to construct, for every $m \in \N^+$, a minimal element $\ba^{(m)}$ of $A_m$, where
        \[
            A_m := \dset{(a_1, \dots, a_n) \in A}{(a_1, \dots, a_n) \text{ has }m\text{ nonzero entries}}.
        \]
        Our inductive construction additionally requires that $\ba^{(1)}, \ba^{(2)}, \ba^{(3)},\dots$ form a descending chain.
        
        Apply \cref{lem:rowing}\ref{lem:rowing-b} to obtain $\ell \in \N^+$ such that
        \begin{multline} \label{eqn:main1-3-l}
            \la_1(R(a_1, \dots, a_{n_1}, 0^{(\ell)}, b_1, \dots, b_{n_2})) > \la_1(R(a_1, \dots, a_{n_1}, 0^{(\ell)})) - \eps \\
            \text{for every }n_1, n_2 \in \N, (a_1, \dots, a_{n_1}) \in \N^{n_1} \text{ and } (b_1, \dots, b_{n_2}) \in \N^{n_2}.
        \end{multline}
        In particular, when $n_1 = 0$, we have
        \[
            \la_1(R(0^{(\ell)}, b_1, \dots, b_{n_2})) > \la_1(R(0^{(\ell)})) - \eps \text{ for every }(b_1, \dots, b_{n_2}) \in \N^{n_2}.
        \]
        
        For the base case where $m = 1$, note that $R(0^{(\ell)})$ is just $E_{2,\ell}$ defined in \cref{lem:e2n}, which satisfies that
        \[
            \la_1(R(0^{(\ell)})) = \la_1(E_{2,\ell}) > -\la^* > -\la.
        \]
        Since no rowing graph has its smallest eigenvalue in $(-\la - \eps, -\la)$, the above two inequalities imply
        \[
            \la_1(R(0^{(\ell)}, b_1, \dots, b_{n_2})) \ge -\la \text{ for every }(b_1, \dots, b_{n_2}) \in \N^{n_2},
        \]
        which further implies that the length of every sequence in $A_1$ is at most $\ell$, that is, $A_1 \subseteq S_{\ell}$. Moreover, \cref{lem:rowing}\ref{lem:rowing-a} implies that $A_1$ is nonempty. Therefore there exists a minimal element of $A_1$.
    
        For the inductive step, suppose that $m \in \N^+$, and that $\ba^{(1)} \succ \dots \succ \ba^{(m)}$ are minimal elements of $A_1, \dots, A_m$ respectively. Say $\ba^{(m)} = (b_1, \dots, b_n)$. \cref{lem:rowing}\ref{lem:rowing-c} shows that there exists $b_{n+1} \in \N^+$ such that $\bb^{(m+1)} := (b_1, \dots, b_{n-1}, b_n-1, b_{n+1}) \in A$. Since $\bb^{(m+1)} \prec \ba^{(m)}$ and $\ba^{(m)}$ is minimal in $A_m$, it must be the case that $b_n > 1$ or equivalently $\bb^{(m+1)} \in A_{m+1}$. It suffices to show that
        \[
            \dset{\bb \in A_{m+1}}{\bb \preceq \bb^{(m+1)}} \subseteq S_{(m+1)\ell}.
        \]
        Assume for the sake of contradiction that there exists $\bb \in A_{m+1}$ such that $\bb \preceq \bb^{(m+1)}$ and the length of $\bb$ is more than $(m+1)\ell$. By the pigeonhole principle, there exists an initial segment, denoted $\bb'$, of $\bb$ such that $(\bb', 0^{(\ell)})$ is an initial segment of $\bb$. We may require in addition that either $\bb'$ is the empty sequence or the last entry of $\bb'$ is positive. If $\bb'$ is the empty sequence, then we can argue similarly to the base case that
        \[
            \la_1(R(\bb, c_1, \dots, c_{n'})) \ge -\la \text{ for every }(c_1, \dots, c_{n'}) \in \N^{n'},
        \]
        which contradicts $\bb \in A_{m+1}$.
        
        We may assume that the last entry of $\bb'$ is positive. Because $(\bb', 0^{(\ell)})$ is an initial segment of $\bb$, and the last entry of $\bb$ is positive, the number of nonzero entries in $\bb'$, denoted $i \in \N^+$, is at most $m$. Since $\bb' \prec \bb \preceq \bb^{(m+1)} \prec \ba^{(m)} \preceq \ba^{(i)}$, and $\ba^{(i)}$ is minimal in $A_i$, we know that $\bb' \not\in A$, and in particular
        \[
            \la_1(R(\bb', 0^{(\ell)})) \ge -\la.
        \]
        From \eqref{eqn:main1-3-l}, we know that
        \[
            \la_1(R(\bb, c_1, \dots, c_{n'})) > \la_1(R(\bb', 0^{(\ell)})) - \eps \text{ for every }(c_1, \dots, c_{n'}) \in \N^{n'}.
        \]
        Since no rowing graph has its smallest eigenvalue in $(-\la-\eps, -\la)$, the above two inequalities imply
        \[
            \la_1(R(\bb, c_1, \dots, c_{n'})) \ge -\la \text{ for every }(c_1, \dots, c_{n'}) \in \N^{n'},
        \]
        which contradicts $\bb \in A_{m+1}$.
    \end{claimproof}

    Finally, let $m$ be given by \cref{lem:rowing}\ref{lem:rowing-d}. The claim provides $(a_1, \dots, a_n) \in A$ with $m$ nonzero entries such that \eqref{eqn:main1-3-claim} holds. Let $\ell \in \N$ be such that
    \[
        \la_1(R(a_1, \dots, a_n, 0^{(\ell)})) < -\la.
    \]
    Since the length of $(a_1, \dots, a_n, 0^{(\ell)})$ is at least $m$, \cref{lem:rowing}\ref{lem:rowing-d} says that there exists $k \in \sset{1, \dots, m}$ such that
    \[
        \la_1(R(a_1, \dots, a_{k-1}, 0, a_k, \dots, a_n, 0^{(\ell)})) < \la_1(R(a_1, \dots, a_n, 0^{(\ell)})) + \eps,
    \]
    However \eqref{eqn:main1-3-claim} asserts that $(a_1, \dots, a_{k-1}, 0, a_k, \dots, a_n) \not\in A$, which implies that
    \[
        \la_1(R(a_1, \dots, a_{k-1}, 0, a_k, \dots, a_n, 0^{(\ell)})) \ge -\la.
    \]
    Combining the last three inequalities, we obtain that
    \begin{equation*}
        -\la - \eps < \la_1(R(a_1,\dots, a_n, 0^{(\ell)})) < -\la. \qedhere
    \end{equation*}
\end{proof}

\section{Forbidden subgraphs for \texorpdfstring{$\Gs$}{signed graphs with eigenvalues bounded from below}} \label{sec:forb-signed}

A useful tool in spectral graph theory for signed graphs is switching --- two signed graphs are \emph{switching equivalent} if one graph can be obtained from the other by reversing all the edges in a cut-set. An important feature of switching equivalence is that the switching equivalent signed graphs all have the same spectrum.

Hereinafter we adopt the following convention for an unsigned graph $G$. With a slight abuse of notation, we denote by $G$ the all-positive signed graph with underlying graph $G$. We also denote by $-G$ the all-negative signed graph with the same underlying graph.

To prove the second main theorem, we need to extend the concepts and results from the previous section. We recommend that the reader go through the rest of the section alongside \cref{sec:forb}.

\subsection{Proof of \texorpdfstring{\cref{thm:main2}}{the second main theorem} for \texorpdfstring{$\la < 2$}{λ < 2}}

We first generalize path extensions, path-clique extensions, clique extensions, and \cref{lem:hoffman}(\ref{lem:c1}, \ref{lem:c2}).

\begin{definition}
    Given a nonempty vertex subset $A$ of a signed graph $F^\pm$, a \emph{signed vertex subset} $A^\pm$ of $F^\pm$ is the vertex subset $A$ together with an assignment of signs (positive or negative). Given, in addition, $\ell \in \N$ and $m \in \N^+$,
    \begin{enumerate}[label=(\alph*)]
        \item the \emph{path extension} $(F^\pm, A^\pm, \ell)$ is obtained from $F^\pm$ by adding an all-positive path $v_0 \dots v_\ell$ of length $\ell$ and connecting $v_0$ to every vertex $v$ in $A$ by an edge signed according to the sign of $v$ in $A^\pm$;
        \item the \emph{path-clique extension} $(F^\pm, A^\pm, \ell, K_m)$ is further obtained from $(F^\pm, A^\pm, \ell)$ by adding an all-positive clique of size $m$ and connecting every vertex in the clique to $v_\ell$ by a positive edge;
        \item the \emph{clique extension} $(F^\pm, A^\pm, K_m)$ is obtained from $F^\pm$ by adding a clique of order $m$ and connecting every vertex in the clique to every vertex $v$ in $A$ by an edge signed according to the sign of $v$ in $A^\pm$.
    \end{enumerate}
\end{definition}

\begin{lemma} \label{lem:hoffman-signed}
    For every signed cycle $C^\pm_n$ of length $n$ and every signed set $A^\pm$ of two adjacent vertices of $C^\pm_n$, the path-clique extensions and the clique extensions of $C^\pm_n$ satisfy:
    \begin{itemize}
        \item[(c1)] $\lim_{m\to\infty} \la_1(C_n^\pm, A^\pm, \ell, K_m) \le -2$ for fixed $\ell \in \N$;\wlabel{lem:c1s}{c1}
        \item[(c2)] $\lim_{m\to\infty} \la_1(C_n^\pm, A^\pm, K_m) \le -2$.\wlabel{lem:c2s}{c2}
    \end{itemize}
\end{lemma}

\begin{proof}
    Let $v_0,v_1,\dots, v_{n-1}$ be the vertices of $C_n^\pm$, and let $\sigma\colon E(C^\pm_n) \to \sset{\pm 1}$ be the signing of $C^\pm_n$. Suppose that $A^\pm$ is a signed set of $\sset{v_0,v_{n-1}}$. By switching a suitable subset of $\sset{v_0, v_{n-1}}$, we may assume that $A^\pm = \{v_0^+, v_{n-1}^+\}$. If $C_n^\pm$ has an even number of positive edges, then $C^\pm_n$ is already switching equivalent to the all-negative cycle of length $n$, whose smallest eigenvalue is $-2$. If the edge $v_0v_{n-1}$ is negative, then both $(C_n^\pm, A^\pm, \ell, K_m)$ and $(C_n^\pm, A^\pm, K_m)$ contain a subgraph that is switching equivalent to the all-negative triangle $-K_3$, and so their smallest eigenvalues are at most $\la_1(-K_3) = -2$.

    Hereafter we assume that $C_n^\pm$ has an odd number of positive edges, and the edge $v_0v_{n-1}$ is positive. Furthermore, by switching a suitable subset of $\sset{v_1, \dots, v_{n-2}}$, we may assume that $v_0v_{n-1}$ is the only positive edge of the signed cycle $C_n^\pm$.

    For (\ref{lem:c1s}), we define a vector $\bx\colon V(C_n^\pm, A^\pm, \ell, K_m) \to \R$ as follows. Let $v_{n}v_{n+1}\dots v_{n+\ell}$ and $K_m$ be the path and the clique added to $C_n^\pm$ to obtain the path-clique extension $(C_n^\pm, A^\pm, \ell, K_m)$, where $v_{n}$ is adjacent to $v_0$ and $v_{n-1}$, and $v_{n+\ell}$ is adjacent to every vertex in $K_m$. Assign $x_{v_i} = -1$ for $i \in \sset{0, \dots, n-1}$, $x_{v_{n+i}} = 2(-1)^i$ for $i \in \sset{0, \dots, \ell}$, and $x_u = -x_{n+\ell}/m$ for every $u \in V(K_m)$. We abuse notation and write $x_i$ in place of $x_{v_i}$ for $i \in \sset{0, \dots, n+\ell}$. The Rayleigh principle says that the smallest eigenvalue of $(C_n^\pm, A^\pm, \ell, K_m)$ is at most
    \begin{multline*}
        \bigg(-2\sum_{i=0}^{n-2}x_ix_{i+1} + 2x_0x_{n-1} + 2(x_0 + x_{n-1})x_n + 2\sum_{i=0}^{\ell-1}x_{n+i}x_{n+i+1} - 2x_{n+\ell}^2 + \\
        + m(m-1)(x_{n+\ell}/m)^2\bigg) \bigg/ \bigg(\sum_{i=0}^{n-1} x_i^2 + \sum_{i=0}^\ell x_{n+i}^2 + m(x_{n+\ell}/m)^2\bigg),
    \end{multline*}
    which is equal to $(-2n - 8(\ell+1) - 4/m)/(n + 4(\ell+1) + 4/m)$, and approaches $-2$ as $m \to \infty$.

    For (\ref{lem:c2s}), we define a vector $\bx \colon V(C_n^\pm, A^\pm, K_m) \to \R$ as follows. Let $K_m$ be the clique added to $C_n^\pm$ to obtain the clique extension $(C_n^\pm, A^\pm, K_m)$. Assign $x_{v_i} = -1$ for $i \in \sset{0, \dots, n-1}$ and $x_u = -(x_0 + x_{n-1})/m$ for every $u \in V(K_m)$. We abuse notation and write $x_i$ in place of $x_{v_i}$ for $i \in \sset{0,\dots, n-1}$. The Rayleigh principle says that the smallest eigenvalue of $(C_n^\pm, A^\pm, K_m)$ is at most
    \[
        \frac{-2\sum_{i=0}^{n-2}x_ix_{i+1} + 2x_0x_{n-1} - 2(x_0 + x_{n-1})^2 + m(m-1)((x_0 + x_{n-1})/m)^2}{\sum_{i=0}^{n-1} x_i^2 + m((x_0 + x_{n-1})/m)^2},
    \]
    which is equal to $(-2n - 4/m)/(n + 4/m)$, and approaches $-2$ as $m \to \infty$.
\end{proof}

We naturally generalize extension families as well as \cref{lem:hoffman-extension}.

\begin{definition}[Signed extension family]
    Given a signed graph $F^\pm$ and $\ell, m \in \N^+$, the \emph{signed extension family} $\X^\pm(F^\pm, \ell, m)$ of $F^\pm$ consists of the path-extension $(F^\pm, A^\pm, \ell)$, the path-clique extensions $(F^\pm, A^\pm, \ell_0, K_m)$, and the clique extension $(F^\pm, A^\pm, K_m)$, where $A^\pm$ ranges over the nonempty signed vertex subsets of $F^\pm$, and $\ell_0$ ranges over $\sset{0, \dots, \ell-1}$.
\end{definition}

Note that for an unsigned graph $F$ and $\ell, m \in \N^+$, the extension family $\X(F, \ell, m)$ (defined in \cref{def:extension-family}) is a subfamily of the signed extension family $\X^\pm(F, \ell, m)$.

\begin{lemma} \label{lem:hoffman-extension-signed}
    For every $\la < 2$, there exist $\ell, m \in \N^+$ such that both the signed extension family $\X^\pm(C, \ell, m)$ of the claw graph $C$ and the signed extension family $\X^\pm(D, \ell, m)$ of the diamond graph $D$ are disjoint from $\Gs$.
\end{lemma}

\begin{proof}
    From \cref{lem:hoffman}(\ref{lem:n1}, \ref{lem:n2}, \ref{lem:p}) and \cref{lem:hoffman-extension}, we obtain $\ell, m \in \N^+$ such that $P_\ell \not\in \Gs$, none of the following graphs
    \begin{equation} \label{eqn:signed-k2-extensions}
        (\overline{K}_2, V(\overline{K}_2), 0, K_m), \dots, (\overline{K}_2, V(\overline{K}_2), \ell-1, K_m), \text{ and }(\overline{K}_2, V(\overline{K}_2), K_m)
    \end{equation}
    is in $\Gs$, and $\X(C, \ell, m) \cup \X(D, \ell, m)$ is disjoint from $\Gs$. Since $P_\ell \not\in \Gs$, no path extension in $\X^\pm(C, \ell, m) \cup \X^\pm(D, \ell, m)$ is in $\Gs$.
    
    We are left to deal with the path-clique extensions and the clique extension
    \begin{equation} \label{eqn:signed-extension}
        (F, A^\pm, 0, K_m), \dots, (F, A^\pm, \ell-1, K_m), \text{ and }(F, A^\pm, K_m),
    \end{equation}
    where $F$ is $C$ or $D$, and $A^\pm$ is the signed set of a nonempty vertex subset $A$ of $F$. We break the rest of the proof into three cases.
    
    \bigskip
    \noindent\textit{Case 1:} $A^\pm$ is all-positive or all-negative. This case follows since the signed graphs in \eqref{eqn:signed-extension}, after switching the cut-set between $V(F)$ and its complement in case $A^\pm$ is all-negative, become respectively $(F, A, 0, K_m), \dots, (F, A, \ell-1, K_m)$, and $(F, A, K_m)$ from the extension family $\X(F, \ell, m)$, which is disjoint from $\Gs$.
    
    \bigskip
    \noindent\textit{Case 2:} There exists an edge $uv$ of $F$ such that $\sset{u^-, v^+} \subseteq A^\pm$. Since the all-negative triangle $-K_3$, whose smallest eigenvalue is $-2$, is switching equivalent to a subgraph of $(F, A^\pm, 0)$, no signed graph in \eqref{eqn:signed-extension} is in $\Gs$.

    \bigskip
    \noindent\textit{Case 3:} $A^\pm$ is neither all-positive nor all-negative, and no edge $uv$ of $F$ satisfies that $\sset{u^-, v^+} \subseteq A^\pm$. Label the vertices of $C$ and $D$ as in \cref{fig:claw-and-diamond}. We may assume without loss of generality that $\sset{1^-, 3^+} \subseteq A^\pm$. One can check that, for $F \in \sset{C, D}$, the signed graphs in \eqref{eqn:signed-k2-extensions}, after switching at the vertex labeled by $1$, are respectively subgraphs of those in \eqref{eqn:signed-extension}, and hence no signed graph in \eqref{eqn:signed-extension} is in $\Gs$.
\end{proof}

Next, generalizing \cref{lem:forbid-extensions}, we show that forbidding a star, an all-negative complete graph, and a signed extension family of $F^\pm$ effectively forbids $F^\pm$ itself in every sufficiently large connected signed graph.

\begin{lemma} \label{lem:forbid-extensions-signed}
    For every signed graph $F^\pm$ and $k_1, k_2, \ell, m \in \N^+$, there exists $N \in \N$ such that for every connected signed graph $G^\pm$ with more than $N$ vertices, if $G^\pm$ does not contain any subgraph that is switching equivalent to any member in $\sset{S_{k_1}, -K_{k_2}} \cup \X^\pm(F^\pm, \ell, m)$, then $G^\pm$ does not contain any subgraph that is switching equivalent to $F^\pm$ either. \qed
\end{lemma}

We leave the proof to the reader as one can make the proof of \cref{lem:forbid-extensions} into that for \cref{lem:forbid-extensions-signed} by \emph{mutatis mutandis}. To get started, one might want to take $N = vd^\ell$, where $d$ is the Ramsey number $R(k_1, k_2, 3^{v+1} m + v)$.

The last ingredient is a sufficient condition for signed line graphs, which generalizes \cref{thm:claw-diamond}.

\begin{definition}[Bidirected graph and signed line graph] \label{def:signed-line-graph}
    A \emph{bidirected graph} is a graph in which each vertex-edge incidence has a positive or negative sign. We decorate variables for bidirected graphs with an arrow on the top. A signed graph $G^\pm$ is the \emph{signed line graph} of a bidirected graph $\vec{H}$ if the underlying graph of $G^\pm$ is the line graph of $\vec{H}$, and moreover for every two distinct edges $e_1$ and $e_2$ of $\vec{H}$ having a vertex $v$ in common, the sign of the edge $e_1 e_2$ in $G^\pm$ is the product of the signs of the two incidences $(v, e_1)$ and $(v, e_2)$ in $\vec{H}$.
\end{definition}

Pictorially we place a sign close to each incidence in a bidirected graph. See \cref{fig:bidirected-graph} for an example of a bidirected graph and its signed line graph.

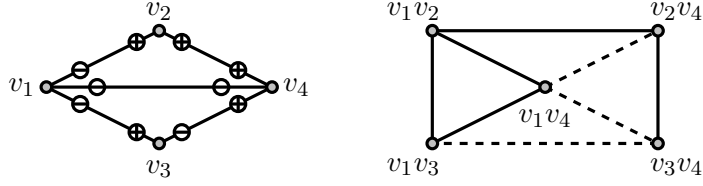
\begin{figure}
    \centering
    \begin{tikzpicture}[
        very thick, scale=0.9, baseline=(v.base),
        neg/.pic = {\draw[fill=white] (0,0) circle (0.1); \draw (-0.1,0)--(0.1,0);},
        pos/.pic = {\draw[fill=white] (0,0) circle (0.1); \draw (-0.1,0)--(0.1,0); \draw (0,0.1)--(0,-0.1);},
    ]
        \coordinate (v) at (0,0);
        \draw (-2,0) -- (0,1) -- (2,0) -- (0,-1) -- (-2,0) -- (2,0);
        \draw (-1.4,0.3) pic {neg};
        \draw (-0.4,0.8) pic {pos};
        \draw (0.4,0.8) pic {pos};
        \draw (1.4,0.3) pic {pos};
        \draw (-1.4,-0.3) pic {neg};
        \draw (-0.4,-0.8) pic {pos};
        \draw (1.4,-0.3) pic {pos};
        \draw (0.4,-0.8) pic {neg};
        \draw (-1.1,0) pic {neg};
        \draw (1.1,0) pic {neg};
        \node[vertex] at (-2,0) {};
        \node[vertex] at (2,0) {};
        \node[vertex] at (0,-1) {};
        \node[vertex] at (0,1) {};
        \draw (-2,0) node[left]{$v_1$};
        \draw (0,1) node[above]{$v_2$};
        \draw (0,-1.1) node[below]{$v_3$};
        \draw (2,0) node[right]{$v_4$};
    \end{tikzpicture}\qquad
    \begin{tikzpicture}[very thick, scale=0.6, baseline=(v.base)]
        \coordinate (v) at (0,0);
        \draw (-3,1.5) -- (3,1.5) -- (3,-1.5);
        \draw[dashed] (3,-1.5) -- (-3,-1.5);
        \draw (-3,-1.5) -- (-3,1.5) -- (0,0);
        \draw[dashed] (0,0) -- (3,-1.5);
        \draw (-3,-1.5) -- (0,0);
        \draw[dashed] (0,0) -- (3,1.5);
        \node[vertex] at (-3,1.5) {};
        \node[vertex] at (3,-1.5) {};
        \node[vertex] at (-3,-1.5) {};
        \node[vertex] at (3,1.5) {};
        \node[vertex] at (0,0) {};
        \node[below] at (0,-0.3) {$v_1v_4$};
        \node[above] at (-3.5,1.5) {$v_1v_2$};
        \node[above] at (3.5,1.5) {$v_2v_4$};
        \node[below] at (-3.5,-1.6) {$v_1v_3$};
        \node[below] at (3.5,-1.6) {$v_3v_4$};
    \end{tikzpicture}
    \caption{A bidirected graph and its signed line graph. In a signed graph, the positive edges are represented by solid segments and the negative edges are represented by dashed segments.} \label{fig:bidirected-graph}
\end{figure}

\begin{remark}
    Our definition of a signed line graph is different from the existing literature --- the conventional definition (e.g.\ Zaslavsky~\cite{Z10}) reverses all the edge signs of $G^\pm$ in \cref{def:signed-line-graph}. Our choice of the definition is based on aesthetic reasons: the signed line graph of an all-positive bidirected graph is still all-positive, and the smallest eigenvalue of a signed line graph is at least $-2$.
\end{remark}

\begin{lemma}[Proposition~2.2 of Vijayakumar~\cite{V87}] \label{lem:claw-diamond-signed}
    For every signed graph $G^\pm$, if $G^\pm$ does not contain any subgraph that is switching equivalent to the claw graph, the diamond graph, or the all-negative triangle, then $G^\pm$ is a signed line graph.
\end{lemma}

Although the conclusion in \cite[Proposition~2.2]{V87} is weaker than that in \cref{lem:claw-diamond-signed}, Vijayakumar's proof already gives \cref{lem:claw-diamond-signed}. Below we reproduce his proof in terms of signed line graphs.

\begin{proof}
    Let $G$ be the underlying graph of $G^\pm$. One can check that $G$ contains neither the claw graph nor the diamond graph as a subgraph. Thus \cref{thm:claw-diamond} implies that $G$ is a line graph of another graph, denoted $H$, without isolated vertices. We may identify the vertex set of $G^\pm$ with the edge set of $H$.
    
    For every $v \in V(H)$, let $E_v(H)$ be the set of edges that are incident to $v$ in $H$. Define a signing $\sigma \colon \dset{(v, e)}{v \in V(H), e \in E_v(H)} \to \sset{\pm1}$ such that for every $v \in V(H)$,
    \begin{equation} \label{eqn:bidirected-graph}
        \sigma(v, e_1)\sigma(v, e_2) = A_{G^\pm}(e_1,e_2) \text{ for distinct } e_1, e_2 \in E_v(H).
    \end{equation}
    Such a signing $\sigma$ can be chosen as follows: for every $v\in V(H)$, define $\sigma(v, e_0) = 1$ for some arbitrary $e_0 \in E_v(H)$; and define $\sigma(v, e) = A_{G^\pm}(e_0, e)$ for any other $e \in E_v(H)$. Since $G^\pm$ does not contain any subgraph that is switching equivalent to the all-negative triangle, it can be seen that \eqref{eqn:bidirected-graph} holds. Assigning $\sigma(v, e)$ to every incidence $(v, e)$ of $H$, we obtain a bidirected graph $\vec{H}$ whose signed line graph is $G^\pm$.
\end{proof}

Similar to \cref{lem:line-graph-of-tn}, we obtain a finite set of forbidden subgraphs for $\Gs$ that forces every sufficiently large connected signed graph to be the singed line graph of a bidirected tree whose complexity is uniformly bounded.

\begin{lemma} \label{lem:main2-claim}
    For every $\la < 2$, there exist $N \in \N$ and a finite family $\F_1^\pm$ that is disjoint from $\Gs$ such that for every connected signed graph $G^\pm$ with more than $N$ vertices, if $G^\pm$ contains no member in $\F_1^\pm$ as a subgraph, then there exists a bidirected rooted tree $\vec{H} \in \vec{\mathcal{T}}_N$ such that $G^\pm$ is the signed line graph of $\vec{H}$, where $\vec{\mathcal{T}}_N$ is the family of bidirected rooted tree such that every connected component obtained from removing the root has at most $N$ vertices.
\end{lemma}

\begin{proof}
    Denote $\mathcal{C}^\pm_n$ the set of signed cycles of length $n$. We obtain $\ell, m \in \N^+$ from \cref{lem:hoffman,lem:hoffman-signed,lem:hoffman-extension-signed} such that the following family
    \begin{align*}
        \widetilde{\F}_1^\pm := & \sset{S_4, P_\ell, -K_3} \cup \X^\pm(C, \ell, m) \cup \X^\pm(D, \ell, m) \\
        & \cup \dset{(C_n^\pm, V_2(C_n), \ell_0, K_m)}{n \in \sset{3, \dots, \ell + 1}, C_n^\pm \in \mathcal{C}_n^\pm, \ell_0 \in \sset{0, \dots, \ell - 1}} \\
        & \cup \dset{(C_n^\pm, V_2(C_n), K_m)}{n \in \sset{3, \dots, \ell + 1}, C_n^\pm \in \mathcal{C}_n^\pm} \\
        & \cup \dset{(K_m, V(K_m), \ell_0, K_m)}{\ell_0 \in \sset{0, \dots, \ell-1}}
    \end{align*}
    is disjoint from $\Gs$. Let $\F_1^\pm$ be the smallest family, that is closed under switching, containing $\widetilde{\F}_1^\pm$. Clearly $\F_1^\pm$ is also disjoint from $\Gs$.
    
    We omit the rest of the proof as it, \emph{mutatis mutandis}, is very much like the one in the proof of \cref{thm:main1} for $\la < 2$ on \cpageref{proof:main1-claim-start}. We point out that one should make use of \cref{lem:forbid-extensions-signed,lem:claw-diamond-signed} in place of \cref{lem:forbid-extensions,thm:claw-diamond} respectively.
\end{proof}

We are ready to present the proof of the second main theorem for $\la < 2$.

\begin{proof}[Proof of \cref{thm:main2} for $\la < 2$]
    Let $N \in \N$ and $\F_1^\pm$ be given by \cref{lem:main2-claim}, and set
    \begin{align*}
        \F_0^\pm & := \dset{G^\pm \not\in \Gs}{G^\pm \text{ has at most }N\text{ vertices}}, \\
        \widetilde{\F}_2^\pm & := \dset{G^\pm \not\in \Gs}{\exists \vec{H} \in \vec{\mathcal{T}}_N \text{ s.t. }G^\pm \text{ is the signed line graph of }\vec{H}},
    \end{align*}
    where $\vec{\mathcal{T}}_N$ is the family of bidirected rooted tree such that every connected component obtained from removing the root has at most $N$ vertices. Setting $\F_2^\pm$ to be the family of signed graphs that are minimal in $\widetilde{\F}_2^\pm$ under taking subgraphs, one can check that $\F_0^\pm \cup \F_1^\pm \cup \F_2^\pm$ is a forbidden subgraph characterization of $\Gs$.

    It suffices to prove that $\F_2^\pm$ is finite. Let $\vec{T}_1, \dots, \vec{T}_n$ be an enumeration of bidirected rooted trees $\vec{T}$ on at most $N+1$ vertices. We encode $G^\pm \in \F_2^\pm$ by $t_{G^\pm} \in \N^n$ as follows. Let $\vec{H}$ be the bidirected rooted tree in $\vec{\mathcal{T}}_N$ such that $G^\pm$ is the signed line graph of $\vec{H}$. After removing the root $u$ from $\vec{H}$, let $U_1, U_2, \dots, U_m$ be the vertex sets of the connected components. For $i \in \sset{1, \dots, m}$, we view the subgraph $\vec{H}_i := \vec{H}[\sset{u} \cup U_i]$ as a bidirected tree rooted at $u$. Set $t_{G^\pm} = (t_1, \dots, t_n)$, where $t_i$ is the number of occurrences of $\vec{T}_i$ in $\vec{H}_1, \dots, \vec{H}_m$. Because no member of $\F_2^\pm$ is a subgraph of any other, one can deduce that $\dset{t_{G^\pm}}{G^\pm \in \F_2^\pm}$ is an antichain in $(\N^n, \le)$, and so $\F_2^\pm$ is finite by \cref{lem:dickson}.
\end{proof}

\subsection{Proof of \texorpdfstring{\cref{thm:main2}}{the second main theorem} for \texorpdfstring{$\la \in [2, \la^*)$}{λ in [2, λ*)}}

We prove a generalization of \cref{thm:main1-2}, from which the second main theorem for $\la \in [2, \la^*)$ follows.

\begin{theorem} \label{thm:main2-2}
    For every $\la \in [2,\la^*)$, the number of connected signed graphs in $\Gs \setminus \G^\pm(2)$ is finite.
\end{theorem}

\begin{proof}[Proof of \cref{thm:main2} for $\la \in [2, \la^*)$]
    Using the fact that every minimal forbidden subgraph for $\G^\pm(2)$ has at most $10$ vertices (see \cite{V87}), and \cref{thm:main2-2}, one can prove this case by following the argument in the proof of \cref{thm:main1} for $\la \in [2, \la^*)$ on \cpageref{proof:main1-0}.
\end{proof}

Recall from \cref{thm:cgss}\ref{thm:cgss-dn} that a graph is a generalized line graph if and only if it is represented by a subset of the root system $D_n$ defined in \cref{def:dn-e8}. For signed graphs, we work directly with those that are represented by a subset of $D_n$.

\begin{definition}[Representation of signed graphs]
    Given a signed graph $G^\pm$ and a subset $V$ of $\R^n$, we say that $G^\pm$ is \emph{represented} by $V$ if the Gram matrix of $V$ is equal to $A_{G^\pm} + 2I$, where $A_{G^\pm}$ is the signed adjacency matrix of $G^\pm$.
\end{definition}

These signed graphs that are represented by a subset of $D_n$, like generalized line graphs, also have a finite forbidden subgraph characterization.

\begin{theorem}[Chawathe and Vijayakumar~\cite{CV90}] \label{thm:cv}
    The minimal forbidden subgraphs for the family $\D_\infty^\pm$ of signed graphs that are represented by a subset of $D_n$ are listed in \cref{fig:minimal-forb-1,fig:minimal-forb-2} up to switching equivalence. \qed
\end{theorem}

\begin{figure}
    \centering
    \begin{tikzpicture}[very thick, baseline=(v.base)]
    \coordinate (v) at (0,.09549150281);
    \coordinate (v0) at (0,-.20601132958);
    \coordinate (v1) at (0,1);
    \coordinate (v2) at (-.95105651629,-.80901699437);
    \coordinate (v3) at (.95105651629,-.80901699437);
    \draw (v0) -- (v1);
    \draw (v1) -- (v2);
    \draw (v1) -- (v3);
    \draw[dashed] (v2) -- (v3);
    \node[circ] at (v0) {0};
    \node[circ] at (v1) {1};
    \node[circ] at (v2) {2};
    \node[circ] at (v3) {3};
    \node at (0,-1.30450849719) {$S_{32}$};
    \end{tikzpicture}
    \begin{tikzpicture}[very thick, baseline=(v.base)]
    \coordinate (v) at (0,.09549150281);
    \coordinate (v0) at (0,-.20601132958);
    \coordinate (v1) at (0,1);
    \coordinate (v2) at (-.95105651629,-.80901699437);
    \coordinate (v3) at (.95105651629,-.80901699437);
    \draw (v0) -- (v2);
    \draw (v0) -- (v3);
    \draw (v1) -- (v2);
    \draw (v1) -- (v3);
    \draw[dashed] (v2) -- (v3);
    \node[circ] at (v0) {0};
    \node[circ] at (v1) {1};
    \node[circ] at (v2) {2};
    \node[circ] at (v3) {3};
    \node at (0,-1.30450849719) {$S_{33}$};
    \end{tikzpicture}
    \begin{tikzpicture}[very thick, baseline=(v.base)]
    \coordinate (v) at (0,.09549150281);
    \coordinate (v0) at (0,-.20601132958);
    \coordinate (v1) at (0,1);
    \coordinate (v2) at (-.95105651629,-.80901699437);
    \coordinate (v3) at (.95105651629,-.80901699437);
    \draw (v0) -- (v2);
    \draw (v0) -- (v3);
    \draw (v1) -- (v2);
    \draw (v1) -- (v3);
    \draw[dashed] (v0) -- (v1);
    \draw[dashed] (v2) -- (v3);
    \node[circ] at (v0) {0};
    \node[circ] at (v1) {1};
    \node[circ] at (v2) {2};
    \node[circ] at (v3) {3};
    \node at (0,-1.30450849719) {$S_{34}$};
    \end{tikzpicture}
    \begin{tikzpicture}[very thick, baseline=(v.base)]
    \coordinate (v) at (0,.09549150281);
    \coordinate (v0) at (0,-.20601132958);
    \coordinate (v1) at (0,1);
    \coordinate (v2) at (-.95105651629,-.80901699437);
    \coordinate (v3) at (.95105651629,-.80901699437);
    \draw (v0) -- (v1);
    \draw (v0) -- (v2);
    \draw (v0) -- (v3);
    \draw (v1) -- (v2);
    \draw (v1) -- (v3);
    \draw[dashed] (v2) -- (v3);
    \node[circ] at (v0) {0};
    \node[circ] at (v1) {1};
    \node[circ] at (v2) {2};
    \node[circ] at (v3) {3};
    \node at (0,-1.30450849719) {$S_{35}$};
    \end{tikzpicture}
    \begin{tikzpicture}[very thick, baseline=(v.base)]
    \coordinate (v) at (0,.09549150281);
    \coordinate (v0) at (0,0);
    \path (v0) + (90:1) coordinate (v1);
    \path (v0) + (162:1) coordinate (v2);
    \path (v0) + (234:1) coordinate (v3);
    \path (v0) + (306:1) coordinate (v4);
    \path (v0) + (378:1) coordinate (v5);
    \draw (v0) -- (v1);
    \draw (v1) -- (v2);
    \draw (v1) -- (v5);
    \draw (v2) -- (v3);
    \draw (v4) -- (v5);
    \draw[dashed] (v3) -- (v4);
    \node[circ] at (v0) {0};
    \node[circ] at (v1) {1};
    \node[circ] at (v2) {2};
    \node[circ] at (v3) {3};
    \node[circ] at (v4) {4};
    \node[circ] at (v5) {5};
    \node at (0,-1.30450849719) {$S_{36}$};
    \end{tikzpicture}
    \begin{tikzpicture}[very thick, baseline=(v.base)]
    \coordinate (v) at (0,.09549150281);
    \coordinate (v0) at (0,0);
    \path (v0) + (90:1) coordinate (v1);
    \path (v0) + (162:1) coordinate (v2);
    \path (v0) + (234:1) coordinate (v3);
    \path (v0) + (306:1) coordinate (v4);
    \path (v0) + (378:1) coordinate (v5);
    \draw (v0) -- (v3);
    \draw (v0) -- (v4);
    \draw (v1) -- (v2);
    \draw (v1) -- (v5);
    \draw (v2) -- (v3);
    \draw[dashed] (v0) -- (v1);
    \node[circ] at (v0) {0};
    \node[circ] at (v1) {1};
    \node[circ] at (v2) {2};
    \node[circ] at (v3) {3};
    \node[circ] at (v4) {4};
    \node[circ] at (v5) {5};
    \node at (0,-1.30450849719) {$S_{37}$};
    \end{tikzpicture}
    \begin{tikzpicture}[very thick, baseline=(v.base)]
    \coordinate (v) at (0,.09549150281);
    \coordinate (v0) at (0,0);
    \path (v0) + (90:1) coordinate (v1);
    \path (v0) + (162:1) coordinate (v2);
    \path (v0) + (234:1) coordinate (v3);
    \path (v0) + (306:1) coordinate (v4);
    \path (v0) + (378:1) coordinate (v5);
    \draw (v0) -- (v3);
    \draw (v0) -- (v4);
    \draw (v1) -- (v2);
    \draw (v1) -- (v5);
    \draw (v2) -- (v3);
    \draw (v4) -- (v5);
    \draw[dashed] (v1) -- (v0);
    \node[circ] at (v0) {0};
    \node[circ] at (v1) {1};
    \node[circ] at (v2) {2};
    \node[circ] at (v3) {3};
    \node[circ] at (v4) {4};
    \node[circ] at (v5) {5};
    \node at (0,-1.30450849719) {$S_{38}$};
    \end{tikzpicture}
    \begin{tikzpicture}[very thick, baseline=(v.base)]
    \coordinate (v) at (0,.09549150281);
    \coordinate (v0) at (0,0);
    \path (v0) + (90:1) coordinate (v1);
    \path (v0) + (162:1) coordinate (v2);
    \path (v0) + (234:1) coordinate (v3);
    \path (v0) + (306:1) coordinate (v4);
    \path (v0) + (378:1) coordinate (v5);
    \draw (v0) -- (v3);
    \draw (v0) -- (v4);
    \draw (v1) -- (v2);
    \draw (v1) -- (v5);
    \draw (v2) -- (v3);
    \draw (v3) -- (v4);
    \draw[dashed] (v0) -- (v1);
    \node[circ] at (v0) {0};
    \node[circ] at (v1) {1};
    \node[circ] at (v2) {2};
    \node[circ] at (v3) {3};
    \node[circ] at (v4) {4};
    \node[circ] at (v5) {5};
    \node at (0,-1.30450849719) {$S_{39}$};
    \end{tikzpicture}
    \begin{tikzpicture}[very thick, baseline=(v.base)]
    \coordinate (v) at (0,.09549150281);
    \coordinate (v0) at (0,0);
    \path (v0) + (90:1) coordinate (v1);
    \path (v0) + (162:1) coordinate (v2);
    \path (v0) + (234:1) coordinate (v3);
    \path (v0) + (306:1) coordinate (v4);
    \path (v0) + (378:1) coordinate (v5);
    \draw (v0) -- (v1);
    \draw (v0) -- (v2);
    \draw (v0) -- (v3);
    \draw (v0) -- (v5);
    \draw (v2) -- (v3);
    \draw (v4) -- (v5);
    \draw[dashed] (v3) -- (v4);
    \node[circ] at (v0) {0};
    \node[circ] at (v1) {1};
    \node[circ] at (v2) {2};
    \node[circ] at (v3) {3};
    \node[circ] at (v4) {4};
    \node[circ] at (v5) {5};
    \node at (0,-1.30450849719) {$S_{40}$};
    \end{tikzpicture}
    \begin{tikzpicture}[very thick, baseline=(v.base)]
    \coordinate (v) at (0,.09549150281);
    \coordinate (v0) at (0,0);
    \path (v0) + (90:1) coordinate (v1);
    \path (v0) + (162:1) coordinate (v2);
    \path (v0) + (234:1) coordinate (v3);
    \path (v0) + (306:1) coordinate (v4);
    \path (v0) + (378:1) coordinate (v5);
    \draw (v0) -- (v3);
    \draw (v0) -- (v4);
    \draw (v1) -- (v2);
    \draw (v1) -- (v5);
    \draw (v2) -- (v3);
    \draw (v3) -- (v4);
    \draw (v4) -- (v5);
    \draw[dashed] (v0) -- (v1);
    \node[circ] at (v0) {0};
    \node[circ] at (v1) {1};
    \node[circ] at (v2) {2};
    \node[circ] at (v3) {3};
    \node[circ] at (v4) {4};
    \node[circ] at (v5) {5};
    \node at (0,-1.30450849719) {$S_{41}$};
    \end{tikzpicture}
    \begin{tikzpicture}[very thick, baseline=(v.base)]
    \coordinate (v) at (0,.09549150281);
    \coordinate (v0) at (0,0);
    \path (v0) + (90:1) coordinate (v1);
    \path (v0) + (162:1) coordinate (v2);
    \path (v0) + (234:1) coordinate (v3);
    \path (v0) + (306:1) coordinate (v4);
    \path (v0) + (378:1) coordinate (v5);
    \draw (v0) -- (v1);
    \draw (v0) -- (v2);
    \draw (v0) -- (v5);
    \draw (v1) -- (v2);
    \draw (v1) -- (v5);
    \draw (v2) -- (v3);
    \draw (v4) -- (v5);
    \draw[dashed] (v3) -- (v4);
    \node[circ] at (v0) {0};
    \node[circ] at (v1) {1};
    \node[circ] at (v2) {2};
    \node[circ] at (v3) {3};
    \node[circ] at (v4) {4};
    \node[circ] at (v5) {5};
    \node at (0,-1.30450849719) {$S_{42}$};
    \end{tikzpicture}
    \begin{tikzpicture}[very thick, baseline=(v.base)]
    \coordinate (v) at (0,.09549150281);
    \coordinate (v0) at (0,0);
    \path (v0) + (90:1) coordinate (v1);
    \path (v0) + (162:1) coordinate (v2);
    \path (v0) + (234:1) coordinate (v3);
    \path (v0) + (306:1) coordinate (v4);
    \path (v0) + (378:1) coordinate (v5);
    \draw (v0) -- (v1);
    \draw (v0) -- (v2);
    \draw (v0) -- (v4);
    \draw (v0) -- (v5);
    \draw (v1) -- (v2);
    \draw (v1) -- (v5);
    \draw (v2) -- (v3);
    \draw[dashed] (v3) -- (v4);
    \node[circ] at (v0) {0};
    \node[circ] at (v1) {1};
    \node[circ] at (v2) {2};
    \node[circ] at (v3) {3};
    \node[circ] at (v4) {4};
    \node[circ] at (v5) {5};
    \node at (0,-1.30450849719) {$S_{43}$};
    \end{tikzpicture}
    \begin{tikzpicture}[very thick, baseline=(v.base)]
    \coordinate (v) at (0,.09549150281);
    \coordinate (v0) at (0,0);
    \path (v0) + (90:1) coordinate (v1);
    \path (v0) + (162:1) coordinate (v2);
    \path (v0) + (234:1) coordinate (v3);
    \path (v0) + (306:1) coordinate (v4);
    \path (v0) + (378:1) coordinate (v5);
    \draw (v0) -- (v2);
    \draw (v0) -- (v3);
    \draw (v0) -- (v5);
    \draw (v1) -- (v5);
    \draw (v2) -- (v3);
    \draw (v2) -- (v5);
    \draw (v4) -- (v5);
    \draw[dashed] (v3) -- (v4);
    \node[circ] at (v0) {0};
    \node[circ] at (v1) {1};
    \node[circ] at (v2) {2};
    \node[circ] at (v3) {3};
    \node[circ] at (v4) {4};
    \node[circ] at (v5) {5};
    \node at (0,-1.30450849719) {$S_{44}$};
    \end{tikzpicture}
    \begin{tikzpicture}[very thick, baseline=(v.base)]
    \coordinate (v) at (0,.09549150281);
    \coordinate (v0) at (0,0);
    \path (v0) + (90:1) coordinate (v1);
    \path (v0) + (162:1) coordinate (v2);
    \path (v0) + (234:1) coordinate (v3);
    \path (v0) + (306:1) coordinate (v4);
    \path (v0) + (378:1) coordinate (v5);
    \draw (v0) -- (v2);
    \draw (v0) -- (v3);
    \draw (v1) -- (v2);
    \draw (v1) -- (v4);
    \draw (v1) -- (v5);
    \draw (v2) -- (v3);
    \draw (v2) -- (v5);
    \draw (v4) -- (v5);
    \draw[dashed] (v3) -- (v4);
    \node[circ] at (v0) {0};
    \node[circ] at (v1) {1};
    \node[circ] at (v2) {2};
    \node[circ] at (v3) {3};
    \node[circ] at (v4) {4};
    \node[circ] at (v5) {5};
    \node at (0,-1.30450849719) {$S_{45}$};
    \end{tikzpicture}
    \begin{tikzpicture}[very thick, baseline=(v.base)]
    \coordinate (v) at (0,.09549150281);
    \coordinate (v0) at (0,0);
    \path (v0) + (90:1) coordinate (v1);
    \path (v0) + (162:1) coordinate (v2);
    \path (v0) + (234:1) coordinate (v3);
    \path (v0) + (306:1) coordinate (v4);
    \path (v0) + (378:1) coordinate (v5);
    \draw (v0) -- (v2);
    \draw (v0) -- (v3);
    \draw (v0) -- (v5);
    \draw (v1) -- (v2);
    \draw (v1) -- (v5);
    \draw (v2) -- (v3);
    \draw (v2) -- (v5);
    \draw (v4) -- (v5);
    \draw[dashed] (v3) -- (v4);
    \node[circ] at (v0) {0};
    \node[circ] at (v1) {1};
    \node[circ] at (v2) {2};
    \node[circ] at (v3) {3};
    \node[circ] at (v4) {4};
    \node[circ] at (v5) {5};
    \node at (0,-1.30450849719) {$S_{46}$};
    \end{tikzpicture}
    \begin{tikzpicture}[very thick, baseline=(v.base)]
    \coordinate (v) at (0,.09549150281);
    \coordinate (v0) at (0,0);
    \path (v0) + (90:1) coordinate (v1);
    \path (v0) + (162:1) coordinate (v2);
    \path (v0) + (234:1) coordinate (v3);
    \path (v0) + (306:1) coordinate (v4);
    \path (v0) + (378:1) coordinate (v5);
    \draw (v0) -- (v1);
    \draw (v0) -- (v2);
    \draw (v0) -- (v3);
    \draw (v0) -- (v5);
    \draw (v1) -- (v2);
    \draw (v1) -- (v4);
    \draw (v1) -- (v5);
    \draw (v2) -- (v3);
    \draw (v4) -- (v5);
    \draw[dashed] (v3) -- (v4);
    \node[circ] at (v0) {0};
    \node[circ] at (v1) {1};
    \node[circ] at (v2) {2};
    \node[circ] at (v3) {3};
    \node[circ] at (v4) {4};
    \node[circ] at (v5) {5};
    \node at (0,-1.30450849719) {$S_{47}$};
    \end{tikzpicture}
    \begin{tikzpicture}[very thick, baseline=(v.base)]
    \coordinate (v) at (0,.09549150281);
    \coordinate (v0) at (0,0);
    \path (v0) + (90:1) coordinate (v1);
    \path (v0) + (162:1) coordinate (v2);
    \path (v0) + (234:1) coordinate (v3);
    \path (v0) + (306:1) coordinate (v4);
    \path (v0) + (378:1) coordinate (v5);
    \draw (v0) -- (v1);
    \draw (v0) -- (v2);
    \draw (v0) -- (v3);
    \draw (v0) -- (v5);
    \draw (v1) -- (v2);
    \draw (v1) -- (v5);
    \draw (v2) -- (v3);
    \draw (v2) -- (v5);
    \draw (v4) -- (v5);
    \draw[dashed] (v3) -- (v4);
    \node[circ] at (v0) {0};
    \node[circ] at (v1) {1};
    \node[circ] at (v2) {2};
    \node[circ] at (v3) {3};
    \node[circ] at (v4) {4};
    \node[circ] at (v5) {5};
    \node at (0,-1.30450849719) {$S_{48}$};
    \end{tikzpicture}
    \begin{tikzpicture}[very thick, baseline=(v.base)]
    \coordinate (v) at (0,.09549150281);
    \coordinate (v0) at (0,0);
    \path (v0) + (90:1) coordinate (v1);
    \path (v0) + (162:1) coordinate (v2);
    \path (v0) + (234:1) coordinate (v3);
    \path (v0) + (306:1) coordinate (v4);
    \path (v0) + (378:1) coordinate (v5);
    \draw (v0) -- (v1);
    \draw (v0) -- (v2);
    \draw (v0) -- (v3);
    \draw (v0) -- (v5);
    \draw (v1) -- (v2);
    \draw (v1) -- (v4);
    \draw (v1) -- (v5);
    \draw (v2) -- (v3);
    \draw (v2) -- (v5);
    \draw (v4) -- (v5);
    \draw[dashed] (v3) -- (v4);
    \node[circ] at (v0) {0};
    \node[circ] at (v1) {1};
    \node[circ] at (v2) {2};
    \node[circ] at (v3) {3};
    \node[circ] at (v4) {4};
    \node[circ] at (v5) {5};
    \node at (0,-1.30450849719) {$S_{49}$};
    \end{tikzpicture}        
    \caption{Additional minimal forbidden subgraphs for $\D_\infty^\pm$ (up to switching equivalence).} \label{fig:minimal-forb-2}
\end{figure}
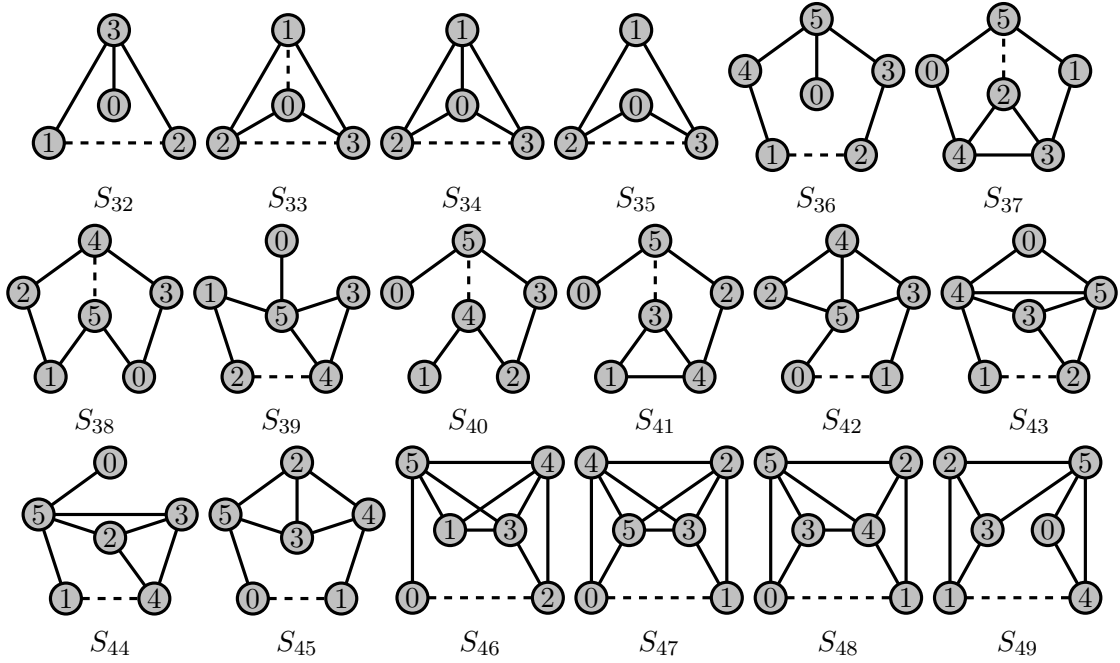

Next, generalizing \cref{lem:computation-1}, we carry out the following computation.

\begin{lemma} \label{lem:computation-2}
    For every minimal forbidden subgraph $F^\pm$ for the family $\D^\pm_\infty$, and every nonempty signed vertex subset $A^\pm$ of $F^\pm$, the path extensions and the clique extensions of $F^\pm$ satisfy
    \[
        \lim_{\ell \to \infty} \la_1(F^\pm, A^\pm, \ell) \le -\la^* \quad \text{and} \quad \lim_{m \to \infty} \la_1(F^\pm, A^\pm, K_m) < -\la^*.
    \]
    Moreover, the equality holds in the first inequality if and only if $F^\pm = G_4$ and $A^\pm \in \sset{\sset{3^\pm}, \sset{4^\pm}}$.
\end{lemma}

We prove \cref{lem:computation-2} under computer assistance in \cref{sec:numerics}. Lastly, the proof of \cref{lem:path-clique-extension} can be taken almost verbatim to show the following generalization.

\begin{lemma} \label{lem:path-clique-extension-signed}
    Suppose that $A^\pm$ is a nonempty vertex subset of a signed graph $F^\pm$ and $\la \ge 2$. If the path extensions of $F^\pm$ satisfy
    \[
        \lim_{\ell \to \infty} \la_1(F^\pm, A^\pm, \ell) < -\la,
    \]
    then there exists $m \in \N^+$ such that the path-clique extensions of $F^\pm$ satisfy
    \[
        \pushQED{\qed}
        \la_1(F^\pm, A^\pm, \ell, K_m) < -\la \text{ for every }\ell \in \N.\qedhere
        \popQED
    \]
\end{lemma}

We are ready to prove \cref{thm:main2-2}.

\begin{proof}[Proof of \cref{thm:main2-2}]
    Suppose that $\la \in [2, \la^*)$. Let $\F^\pm$ denote the set of minimal forbidden subgraphs for the family $\D^\pm_\infty$. According to \cref{thm:cv}, the family $\F^\pm$ is finite. Combining \cref{lem:computation-2,lem:path-clique-extension-signed}, we choose $\ell, m \in \N^+$ such that for every $F^\pm \in \F^\pm$, the signed extension family $\X^\pm(F^\pm, \ell, m)$ is disjoint from $\Gs$. We know that $S_5 \not\in \Gs$ and $-K_4 \not\in \Gs$. In particular, no graph in $\Gs$ contains a subgraph that is switching equivalent to any member in the following family
    \begin{equation} \label{eqn:main2-2}
        \sset{S_5, -K_4} \cup \bigcup_{F^\pm \in \F^\pm} \X^\pm(F^\pm, \ell, m).
    \end{equation}
    Using \cref{lem:forbid-extensions-signed}, we obtain $N \in \N$ such that for every connected signed graph $G^\pm$ with more than $N$ vertices, if no member in \eqref{eqn:main2-2} is switching equivalent to a subgraph of $G^\pm$, then neither is any $F^\pm \in \F^\pm$, and so $G^\pm$ is represented by a subset of $D_n$ and is clearly in $\G^\pm(2)$. This implies that every connected graph in $\Gs \setminus \G^\pm(2)$ has at most $N$ vertices.
\end{proof}

\subsection{Proof of \texorpdfstring{\cref{thm:main2}}{the second main theorem} for \texorpdfstring{$\la \ge \la^*$}{λ ≥ λ*}}

\begin{proof}[Proof of \cref{thm:main2} for $\la \ge \la^*$]
    If $\F^\pm$ is a finite forbidden subgraph characterization of $\Gs$, then the all-positive signed graphs in $\F^\pm$ would form a finite forbidden subgraph characterization of $\Gl$, which contradicts \cref{thm:main1} for $\la \ge \la^*$.
\end{proof}

\section{Spherical two-distance sets} \label{sec:app}

We introduce the definition of chromatic number for signed graphs.

\begin{definition}[Chromatic number] \label{def:chromatic}
    A \emph{valid $p$-coloring} of a signed graph $G^\pm$ is a coloring of the vertices using $p$ colors such that the endpoints of every negative edge receive different colors, and the endpoints of every positive edge receive identical colors. (See \cref{fig:coloring} for an example.) The \emph{chromatic number} $\chi(G^\pm)$ of a signed graph $G^\pm$ is the smallest $p$ for which $G^\pm$ has a valid $p$-coloring. If $G^\pm$ does not have a valid $p$-coloring for any $p$, we write $\chi(G^\pm) = \infty$.
\end{definition}

\begin{figure}[t]
    \centering
    \begin{tikzpicture}[very thick,scale=.8]
        \coordinate (a) at (0.922649731,3);
        \coordinate (b) at (2.077350269,3);
        \coordinate (c) at (-0.577350269,1);
        \coordinate (d) at (0.577350269,1);
        \coordinate (e) at (0,0);
        \coordinate (f) at (2.422649731,1);
        \coordinate (g) at (3.577350269,1);
        \coordinate (h) at (3,0);
        \draw [fill=red!30,draw=none] (0,0.65) circle[radius=1];
        \draw [fill=blue!30,draw=none] (3,0.65) circle[radius=1];
        \draw [fill=green!30,draw=none] (1.5,3) circle[radius=1];
        \draw[dashed] (c) -- (a) -- (d);
        \draw[dashed] (g) -- (b) -- (f);
        \draw[dashed] (d) -- (h);
        \draw[dashed] (e) -- (f);
        \draw (a) node[vertex]{} -- (b) node[vertex]{};
        \draw (c) node[vertex]{} -- (e) node[vertex]{} -- (d) node[vertex]{};
        \draw (f) node[vertex]{} -- (g) node[vertex]{} -- (h) node[vertex]{} -- cycle;
    \end{tikzpicture}
    \caption{A valid $3$-coloring of a signed graph.} \label{fig:coloring}
\end{figure}
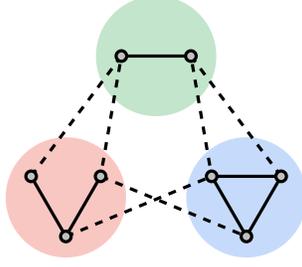

Recall from \cref{sec:intro} the following spectral graph theoretic quantity,
\begin{equation} \label{def:kpl}
    \kpl = \inf\dset{\frac{\abs{G^\pm}}{\mult(\la, G^\pm)}}{\chi(G^\pm) \le p \text{ and }\la^1(G^\pm) = \la}.
\end{equation}
We say that $k_p(\la)$ is \emph{achievable} if it is finite and the infimum can be attained. The spectral graph theoretic quantity $k_p(\la)$ can be seen as a generalization of the \emph{spectral radius order} $k(\la)$ defined by
\[
    k(\la) = \min\dset{\abs{G}}{\la^1(G) = \la}.
\]
Indeed, $k_1(\la) = k_2(\la) = k(\la)$. However, the behavior of $k_p(\la)$ is far more mysterious when $p \ge 3$ (see the remark after \cite[Definition~1.10]{JTYZZ23} for more details). The technique developed in the current paper enables us to certify values of $\kpl$ whenever $\la < \la^*$.

\begin{theorem} \label{thm:kpl}
    For every $p \ge 3$ and $\la < \la^*$, there exists $N \in \N$ such that
    \[
        \kpl = \min\dset{\frac{\abs{G^\pm}}{\mult(\la, G^\pm)}}{\abs{G^\pm} \le N, \chi(G^\pm) \le p \text{ and }\la^1(G^\pm) = \la},
    \]
    and in particular, $\kpl$ is achievable whenever it is finite. Moreover, for $\la = 2$, there exists $p_0 \ge 3$ such that $k_p(2) = p^2/(p-1)^2$ for all $p \ge p_0$.
\end{theorem}

We will return to the proof of \cref{thm:kpl} towards the end of the subsection. The next result constructs large spherical codes with two fixed angles.

\begin{proposition}[Proposition~2.2 of Jiang et al.~\cite{JTYZZ23}] \label{thm:lower-bound}
    Fix $-1 \le \be < 0 \le \al < 1$. Then $\Nabd \ge d$ for every $d \in \N^+$. Moreover if $\kpl < \infty$, where $\la =(1-\al)/(\al - \be)$ and $p = \lfloor -\al/\be \rfloor + 1$, then
    \begin{equation*}
        \pushQED{\qed}
        N_{\al,\be}(d) \ge \begin{dcases}
            \frac{\kpl d}{\kpl-1} - O_{\al,\be}(1) & \text{if }\kpl \text{ is achievable}, \\
            \frac{\kpl d}{\kpl-1} - o(d) & \text{otherwise}.
        \end{dcases} \qedhere
        \popQED
    \end{equation*}
\end{proposition}

\cref{conj:main} asserts that the constructions in \cref{thm:lower-bound} are optimal up to an error sublinear in $d$. A framework to bound $\Nabd$ from above was formulated in \cite{JTYZZ23}.

\begin{definition}[Definition~5.2 of Jiang et al.~\cite{JTYZZ23}] \label{def:mph}
    Given $p \in \N^+$ and a family $\HH$ of signed graphs, let $M_{p, \HH}(\la, N)$ be the maximum possible value of $\mult(\la, G^\pm)$ over all signed graphs $G^\pm$ on at most $N$ vertices that do not contain any member of $\HH$ as a subgraph and satisfy $\chi(G^\pm) \le p$ and $\la^{p+1}(G^\pm) \le \la$.
\end{definition}

\begin{theorem}[Theorem~5.3 of Jiang et al.~\cite{JTYZZ23}] \label{thm:upper-bound}
    Fix $-1 \le \be < 0 \le \al < 1$. Set $\la = (1-\al)/(\al - \be)$ and $p = \lfloor -\al/\be \rfloor + 1$. Let $\HH$ be a finite family of signed graphs with $\la^1(H^\pm) > \la$ for each $H^\pm \in \HH$. Then
    \begin{equation*}
        \pushQED{\qed}
        N_{\al,\be}(d) \le d + M_{p, \HH}(\la, \Nabd) + O_{\al, \be, 
        \HH}(1).  \qedhere
        \popQED
    \end{equation*}
\end{theorem}

A proper choice of the finite family $\HH$ in the last theorem establishes \cref{conj:main} for $\la < \la^*$.

\begin{corollary} \label{cor:spherical}
    Fix $-1 \le \be < 0 \le \al < 1$. Set $\la = (1-\al)/(\al-\be)$ and $p = \lfloor -\al / \be \rfloor + 1$. If $\la < \la^*$, then
    \[
        \Nabd = \begin{dcases}
            \frac{\kpl d}{\kpl-1} + O_{\al,\be}(1) & \text{if }\kpl < \infty, \\
            d + O_{\al,\be}(1) & \text{otherwise}.
        \end{dcases}
    \]
\end{corollary}

\begin{proof}
    According to \cref{cor:main2}, we can take a finite forbidden subgraph characterization, denoted $\HH$, of the family $\Gsp$ of signed graphs with largest eigenvalue at most $\la$. Thus $M_{p, \HH}(\la, N)$ is the maximum possible value of $\mult(\la, G^\pm)$ over all signed graphs $G^\pm$ on at most $N$ vertices that satisfy $\chi(G^\pm) \le p$ and $\la^1(G^\pm) \le \la$. Note that $M_{p, \HH}(\la, N)$ is a nondecreasing function of $N$. We break the rest of the proof into two cases.

    \bigskip
    \noindent\textit{Case 1:} $M_{p, \HH}(\la, N) = 0$ for every $N \in \N^+$. In other words, there is no signed graph $G^\pm$ that satisfy $\chi(G^\pm) \le p$ and $\la^1(G^\pm) = \la$. Thus this case corresponds to the case where $\kpl = \infty$. According to \cref{thm:lower-bound,thm:upper-bound}, we have $\Nabd = d + O_{\al,\be}(1)$.

    \bigskip
    \noindent\textit{Case 2:} $M_{p, \HH}(\la, N) > 0$ for every $N \ge d_0$. This case corresponds to the case where $\kpl < \infty$. Suppose that $d \ge d_0$. \cref{thm:lower-bound} says that $\Nabd \ge d$, and so $M_{p, \HH}(\la, \Nabd) > 0$. Let $G^\pm$ be a signed graph with at most $N_{\al,\be}(d)$ vertices that satisfy $\chi(G^\pm) \le p$, $\la^1(G^\pm) \le \la$ and $\mult(\la, G^\pm) = M_{p, \HH}(\la, \Nabd)$. Thus
    \[
        \kpl \le \frac{\abs{G^\pm}}{\mult(\la, G^\pm)} \le \frac{\Nabd}{M_{p, \HH}(\la, \Nabd)}.
    \]
    \cref{thm:upper-bound} then implies that
    \[
        \Nabd \le \frac{\kpl d}{\kpl-1} + O_{\al,\be}(1),
    \]
    which, in view of \cref{thm:kpl,thm:lower-bound}, gives $\Nabd = \kpl d / (\kpl - 1) + O_{\al,\be}(1)$.
\end{proof}

Lastly, we come back to the achievability of the spectral graph theoretic quantity $\kpl$. Actually we may assume in the definition \eqref{def:kpl} of $\kpl$ that $G^\pm$ is connected by passing to a suitable connected component of $G^\pm$ in case $G^\pm$ is not connected. Furthermore, as we shall mainly work with the signed graph that reverses all the edge signs of $G^\pm$ in \eqref{def:kpl}, we redefine $\kpl$ as follows:
\[
    \kpl = \inf\dset{\frac{\abs{G^\pm}}{\mult(-\la, G^\pm)}}{G^\pm \text{ is connected}, \chi(-G^\pm) \le p, \text{ and }\la_1(G^\pm) = -\la},
\]
where $-G^\pm$ reverses all the edge signs of $G^\pm$.

We break the proof of \cref{thm:kpl} into three cases $\la < 2$, $\la \in (2, \la^*)$ and $\la = 2$.

\begin{proof}[Proof of \cref{thm:kpl} for $\la < 2$]
    We obtain $\ell \in \N^+$ from \cref{lem:hoffman}(\ref{lem:p}) such that $P_{\ell+1} \not\in \Gs$. Applying \cref{lem:forbid-extensions-signed} to $F^\pm = K_1$ (the $1$-vertex graph), $k_1 = 4$, $k_2 = 3$ and $m = 2p$, we obtain $N \in \N$ such that for every connected signed graph $G^\pm$, if $G^\pm$ does not contain any subgraph that is switching equivalent to any member in $\sset{S_4, -K_3} \cup \X^\pm(K_1, \ell, 2p)$, then $G^\pm$ contains at most $N$ vertices.
    
    It suffices to show that for every connected signed graph $G^\pm$ with more than $N$ vertices, either $\chi(-G^\pm) > p$ or $G^\pm \not\in \Gs$. By our choice of $N$, the signed graph $G^\pm$ contains a subgraph that is switching equivalent to a member in $\sset{S_4, -K_3} \cup \X^\pm(K_1, \ell, 2p)$. We break the rest of the proof into two cases.
    
    \bigskip
    \noindent\textit{Case 1:} $G^\pm$ contains a subgraph that is switching equivalent to $S_4$, $-K_3$, or a path extension in $\X^\pm(K_1, \ell, 2p)$. Since $\la_1(S_4) = \la_1(-K_3) = -2$ and every path extension in $\X^\pm(K_1, \ell, 2p)$ is switching equivalent to $P_{\ell+1}$, the signed graph $G^\pm$ is not in $\Gs$.

    \bigskip
    \noindent\textit{Case 2:} $G^\pm$ contains a subgraph that is switching equivalent to a path-clique extension or a clique extension in $\X^\pm(K_1, \ell, 2p)$. Observe that every path-clique extension and every clique extension in $\X^\pm(K_1, \ell, 2p)$ contains a subgraph that is switching equivalent to $K_{2p+1}$, and moreover every signed graph that is switching equivalent to $K_{2p+1}$ always contains $K_{p+1}$ as a subgraph. Thus the signed graph $G^\pm$ contains $K_{p+1}$ as a subgraph. Therefore $-G^\pm$ contains $-K_{p+1}$ as a subgraph, and so $\chi(-G^\pm) \ge p+1$.
\end{proof}

The proof of \cref{thm:kpl} for $\la \in (2, \la^*)$ follows immediately from \cref{thm:main2-2}.

\begin{proof}[Proof of \cref{thm:kpl} for $\la \in (2, \la^*)$]
    \cref{thm:main2-2} implies that there exists $N\in\N$ such that every connected signed graph $G^\pm$ with $\la_1(G^\pm) = -\la$ has at most $N$ vertices.
\end{proof}

The proof of \cref{thm:kpl} for $\la = 2$ is trickier. We need the following generalization of \cref{thm:cgss}. This generalization was explicitly recognized by Chawathe and Vijayakumar~\cite[Theorem~1.1]{CV90}, who attributed the result to Witt~\cite{W41}, and referred the reader to \cite{CGSS76} for a combinatorial proof.

\begin{theorem}[Witt~\cite{W41} and Cameron et al.~\cite{CGSS76}] \label{thm:cgss-signed}
    For every connected signed graph $G^\pm$, the smallest eigenvalue of $G^\pm$ is at least $-2$ if and only if $G^\pm$ is represented by a subset of $D_n$ or $E_8$.
\end{theorem}

Recall that $\D_\infty^\pm$ denotes the family of signed graphs that are represented by a subset of $D_n$ for some $n \in \N^+$, where
\[
    D_n := \dset{\eps_1 \bm{e}_i + \eps_2 \bm{e}_j}{\eps_1, \eps_2 = \pm 1, 1 \le i < j \le n}.
\]
We shall focus on the signed graphs in $\D_\infty^\pm$. To that end, we generalize bidirected graphs to multigraphs as follows.

\begin{definition}[Bidirected multigraph]
    A \emph{bidirected multigraph} is a multigraph (allowing parallel edges, but no loops) in which each vertex-edge incidence $(v, e)$ has a positive or negative sign, denoted $\sigma(v, e) \in \sset{\pm 1}$, such that every pair of parallel edges $e_1$ and $e_2$ satisfies the following condition: for one of the shared endpoints, say $v_1$, the incidences $(v_1, e_1)$ and $(v_1, e_2)$ have the same sign, while for the other shared endpoint $v_2$, the incidences $(v_2, e_1)$ and $(v_2, e_2)$ have opposite signs. In particular, there are at most two edges between any two vertices in a bidirected multigraph.
\end{definition}

As we transfer relevant properties of a signed graph $G^\pm \in \D_\infty^\pm$ to a bidirected multigraph $\vec{H}$, a vertex coloring of $G^\pm$ will correspond to an edge coloring of $\vec{H}$.

\begin{definition}[Intersecting edges and proper edge colorings] \label{def:proper-edge-coloring}
    As opposed to parallel edges, we say that two edges of a bidirected multigraph \emph{intersect} at a vertex $v$ if $v$ is the only vertex they have in common. A \emph{proper $p$-edge-coloring} of a bidirected multigraph $\vec{H}$ is a coloring of the edges using $p$ colors such that for each pair of edges $e_1, e_2$ that intersect at one vertex, say $v$, the edges $e_1$ and $e_2$ receive different colors when the incidences $(v, e_1)$ and $(v, e_2)$ have the same sign, whereas $e_1$ and $e_2$ receive identical colors when $(v, e_1)$ and $(v, e_2)$ have opposite signs. See \cref{fig:bidirected-multigraph} for an example.
\end{definition}

\begin{definition}[Associated vectors and uniform vertices]
    Suppose that $v_1, \dots, v_n$ are the vertices of a bidirected multigraph $\vec{H}$. For every edge $e \in E(\vec{H})$ with endpoints $v_i$ and $v_j$, define the \emph{associated vector} of the edge $e$ by
    \[
        \vec{e} := \sigma(v_i, e)\bm{e}_i + \sigma(v_j, e)\bm{e}_j \in D_n.
    \]
    We say that a vertex $v$ of $\vec{H}$ is \emph{uniform} if the sign $\sigma(v, e)$ is the same for every edge $e$ incident to $v$.
\end{definition}

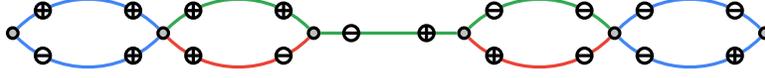
\begin{figure}[t]
    \centering
    \begin{tikzpicture}[
        very thick, baseline=(v.base),
        neg/.pic = {\draw[fill=white] (0,0) circle (0.1); \draw (-0.1,0)--(0.1,0);},
        pos/.pic = {\draw[fill=white] (0,0) circle (0.1); \draw (-0.1,0)--(0.1,0); \draw (0,0.1)--(0,-0.1);},
    ]
        \node[vertex] (a) at (-2,0) {};
        \node[vertex] (b) at (0,0) {};
        \node[vertex] (c) at (2,0) {};
        \node[vertex] (d) at (4,0) {};
        \node[vertex] (e) at (6,0) {};
        \node[vertex] (f) at (8,0) {};
        \path[color=blue] (a) edge[bend left=45] (b);
        \path[color=blue] (a) edge[bend left=-45] (b);
        \path[color=green] (b) edge[bend left=45] (c);
        \path[color=red] (b) edge[bend left=-45] (c);
        \draw[color=green] (c) -- (d);
        \path[color=green] (d) edge[bend left=45] (e);
        \path[color=red] (d) edge[bend left=-45] (e);
        \path[color=blue] (e) edge[bend left=45] (f);
        \path[color=blue] (e) edge[bend left=-45] (f);
        \draw (-1.6,0.3) pic {pos};
        \draw (-1.6,-0.3) pic {neg};
        \draw (-0.4,0.3) pic {pos};
        \draw (0.4,0.3) pic {pos};
        \draw (-0.4,-0.3) pic {pos};
        \draw (0.4,-0.3) pic {pos};
        \draw (1.6,0.3) pic {pos};
        \draw (1.6,-0.3) pic {neg};
        \draw (2.5,0) pic {neg};
        \draw (3.5,0) pic {pos};
        \draw (4.4,0.3) pic {neg};
        \draw (4.4,-0.3) pic {pos};
        \draw (5.6,0.3) pic {neg};
        \draw (5.6,-0.3) pic {neg};
        \draw (6.4,0.3) pic {neg};
        \draw (6.4,-0.3) pic {neg};
        \draw (7.6,0.3) pic {neg};
        \draw (7.6,-0.3) pic {pos};
    \end{tikzpicture}
    \caption{A proper $3$-edge-coloring of a bidirected multigraph.} \label{fig:bidirected-multigraph}
\end{figure}

\begin{proposition} \label{lem:transfer}
    There is a one-to-one correspondence between signed graphs $G^\pm$ in $\D_\infty^\pm$ and bidirected multigraphs $\vec{H}$ without isolated vertices such that the following properties hold.
    \begin{enumerate}[label=(\arabic*)]
        \item The underlying graph of $G^\pm$ is the line graph of $\vec{H}$, in particular, the number of vertices of $G^\pm$ is equal to the number of edges of $\vec{H}$.
        \item The signed graph $G^\pm$ is connected if and only if $\vec{H}$ is connected except when $\vec{H}$ consists of a single pair of parallel edges.
        \item The signed graph $-G^\pm$ has a valid $p$-coloring if and only if $\vec{H}$ has a proper $p$-edge-coloring.
        \item The multiplicity $\mult(-2, G^\pm)$ is equal to $m - \rank V$, where $m$ is the number of edges of $\vec{H}$ and $V = \{\vec{e} \in D_n \colon e \in E(\vec{H})\}$.
        \item The signed graph $G^\pm$ is all-positive if and only if every vertex of $\vec{H}$ is uniform.
    \end{enumerate}    
\end{proposition}

\begin{proof}
    Suppose that $G^\pm$ is represented by $V \subseteq D_n$. Construct a bidirected multigraph $\vec{H}$ with vertices $v_1, \dots, v_n$ and incidence signing $\sigma$ as follows. For each vector in $V$ of the form $\eps_1 \bm{e}_i + \eps_2 \bm{e}_j$ with $i < j$, we put an edge $e$ connecting $v_i$ and $v_j$, and we assign $\sigma(v_i, e) = \eps_1$ and $\sigma(v_j, e) = \eps_2$. For every pair of parallel edges $e_1$ and $e_2$, which connect $v_i$ and $v_j$, since $\sigma(v_i, e_1)\bm{e}_i + \sigma(v_j, e_1)\bm{e}_j$ and $\sigma(v_i, e_2)\bm{e}_i + \sigma(v_j, e_2)\bm{e}_j$ are two distinct vectors from $V$, their inner product satisfies
    \[
        \sset{0, \pm2} \ni \sigma(v_i, e_1)\sigma(v_i, e_2) + \sigma(v_j, e_1)\sigma(v_j, e_2) = A_{G^\pm}(e_1,e_2) \in \sset{0,\pm 1},
    \]
    and so their inner product must be $0$. Therefore $\vec{H}$ is a bona fide bidirected multigraph. We may remove all the isolated vertices (if any) from $\vec{H}$.
    
    Conversely one can read off the vectors in $V$ from $\vec{H}$ --- for every edge $e$, include the associated vector $\vec{e} \in D_n$ in $V$. Furthermore one can reconstruct $G^\pm$ from $\vec{H}$ as follows. We identify the vertex set of $G^\pm$ with the edge set of $\vec{H}$, and for each pair of distinct vertices $e_1$ and $e_2$ of $G^\pm$, they are adjacent in $G^\pm$ if and only if $e_1$ and $e_2$ intersect at one vertex, say $v$, in $\vec{H}$, and the sign of the edge $e_1e_2$ in $G^\pm$ is $\sigma(v, e_1)\sigma(v, e_2)$. One can then check that the reconstructed $G^\pm$ is indeed represented by $V$.
    
    Finally, it is routine to transfer properties back and forth between $G^\pm$ and $\vec{H}$.
\end{proof}

Now we are ready to strengthen the last case of \cref{thm:kpl}.

\begin{lemma} \label{lem:d-infinity}
    For every $p \ge 3$, every connected signed graph $G^\pm$ in $\D_\infty^\pm$ with $\chi(-G^\pm) \le p$ and $\la_1(G^\pm) = -2$ satisfies
    \[
        \frac{\abs{G^\pm}}{\mult(-2, G^\pm)} \ge \left(\frac{p}{p-1}\right)^2,
    \]
    and equality holds if and only if $G^\pm$ is the all-positive line graph of the complete bipartite graph $K_{p,p}$.
\end{lemma}

\begin{proof}
    According to \cref{lem:transfer}, it suffices to show that for every connected bidirected multigraph $\vec{H}$ with $n$ vertices and $m$ edges, if $\vec{H}$ has a proper $p$-edge-coloring, then
    \[
        \frac{m}{m - \rank V} \ge \left(\frac{p}{p-1}\right)^2 \text{ or equivalently }\frac{m}{\rank V} \le \frac{p^2}{2p-1}, \quad\text{where } V := \dset{\vec{e} \in D_n}{e \in E(\vec{H})},
    \]
    and equality holds if and only if $\vec{H}$ is the complete bipartite graph $K_{p,p}$ with uniform vertices.

    Suppose that $c\colon E(\vec{H}) \to \sset{c_1, \dots, c_p}$ is a proper $p$-edge-coloring of $\vec{H}$. One can deduce from \cref{def:proper-edge-coloring} the following properties.
    \begin{enumerate}[label=(\alph*)]
        \item For two edges $e_1$ and $e_2$ that intersect at a vertex $v$, if $\sigma(v, e_1) = \sigma(v, e_2)$, then $c(e_1) \neq c(e_2)$.\label{item:p-a}
        \item For two edges $e_1$ and $e_2$ that intersect at a vertex $v$, if $\sigma(v, e_1) \neq \sigma(v, e_2)$, then $c(e_1) = c(e_2)$, and any other edge incident to $v$ has to be parallel to either $e_1$ or $e_2$.\label{item:p-b}
    \end{enumerate}
    We break the rest of the proof on the upper bound of $m / \rank V$ into two cases.

    \bigskip
    \noindent\textit{Case 1:} $\vec{H}$ has no parallel edges. We claim that the maximum degree of $\vec{H}$ is at most $p$. Suppose on the contrary that $k$ edges $e_1, \dots, e_k$ pairwise intersect at a vertex $v$, where $k > p$. Because $k > p \ge 3$, the contrapositive of \ref{item:p-b} implies that $\sigma(v, e_i)$ is the same for $i \in \sset{1, \dots, k}$. Moreover \ref{item:p-a} says that $c(e_1), \dots, c(e_k)$ are distinct, which contradicts the assumption that $c$ is a proper $p$-edge-coloring. In particular, the claim implies that \begin{equation} \label{eqn:kpl-m}
        m/n \le p/2.
    \end{equation}
    
    Since $\vec{H}$ is connected, one can check that the set of vectors $\dset{\vec{e} \in D_n}{e \in E(T)}$, where $T$ is a spanning tree of $\vec{H}$, is linearly independent. Therefore $\rank V \in \sset{n - 1, n}$. We may assume that \[
        4 \le n \le 2p \quad\text{and}\quad \rank V = n-1.
    \]
    Indeed, in the case where $n \le 3$, we have
    \[
        \frac{m}{\rank V} \le \frac{m}{n-1} \le \frac{\binom{n}{2}}{n-1} = \frac{n}{2} \le \frac{3}{2} < \frac{9}{5} \le \frac{p^2}{2p-1},
    \]
    in the case where $n > 2p$, we have
    \[
        \frac{m}{\rank V} \le \frac{m}{n-1} \stackrel{\eqref{eqn:kpl-m}}{\le} \frac{p/2}{1-1/n} < \frac{p/2}{1-1/(2p)} = \frac{p^2}{2p-1}.
    \]
    and in the case where $\rank V = n$, we also have
    \[
        \frac{m}{\rank V} = \frac{m}{n} \stackrel{\eqref{eqn:kpl-m}}{\le} \frac{p}{2} < \frac{p^2}{2p-1},
    \]
    
    Let $U$ be the vertex set of $\vec{H}$, and define the vertex subset
    \[
        U_0 := \dset{u \in U}{u \text{ is \emph{not} uniform}}.
    \]
    We claim that $\vec{H}[U \setminus U_0]$ is bipartite. To specify its bipartition, let $\bx\colon U \to \R$ be a nonzero vector in the orthogonal complement of $V$. For every edge $e\in E(\vec{H})$ with endpoints $u_1$ and $u_2$, since $\bx$ is orthogonal to the vector $\vec{e}$ associated to $e$, we know that
    \begin{equation} \label{eqn:sigma-c}
        \sigma(u_1, e)\bx(u_1) + \sigma(u_2, e)\bx(u_2) = 0,
    \end{equation}
    hence $\abs{\bx(u_1)} = \abs{\bx(u_2)}$. Because $\vec{H}$ is connected, we deduce that $\abs{\bx(u)}$ is constant for every $u \in U$, and in particular, $\bx(u) \neq 0$ for every $u \in U$. Since every vertex in $U \setminus U_0$ is uniform, we partition $U \setminus U_0$ as follows:
    \begin{align*}
        U_1 & := \dset{u \in U \setminus U_0}{\sigma(u, e)\bx(u) > 0 \text{ for every edge }e \text{ incident to }u}, \\
        U_2 & := \dset{u \in U \setminus U_0}{\sigma(u, e)\bx(u) < 0 \text{ for every edge }e \text{ incident to }u}.
    \end{align*}
    In view of \eqref{eqn:sigma-c}, we deduce that $\vec{H}[U \setminus U_0]$ is a bipartite graph with parts $U_1$ and $U_2$, and so the number of edges in $\vec{H}[U \setminus U_0]$ is at most $\abs{U_1}\abs{U_2}$. According to \ref{item:p-b}, every vertex in $U_0$ has degree at most $2$. Thus, we can estimate $m$ using $n_0 := \abs{U_0}$ as follows:
    \[
        m \le \abs{U_1}\abs{U_2} + 2\abs{U_0} \le \left(\frac{n - n_0}{2}\right)^2 + 2n_0 \le \max\left(\frac{n^2}{4}, 2n-3\right),
    \]
    where the last inequality assumes that $0 \le n_0 \le n-2$. In the corner case where $n_0 \ge n-1$, we can estimate $m$ by
    \[
        2m \le (n-1)(\abs{U_1} + \abs{U_2}) + 2\abs{U_0} = (n-1)(n - n_0) + 2n_0 \le 3(n-1),
    \]
    which implies $m \le 3(n-1)/2 < 2n-3$. Therefore, we always have
    \[
        \frac{m}{\rank V} = \frac{m}{n-1} \le \max\left(\frac{n^2}{4(n-1)}, \frac{2n-3}{n-1}\right) = \begin{cases}
            \frac{2n-3}{n-1} & \text{if }n < 6, \\
            \frac{n^2}{4(n-1)} & \text{if }n \ge 6,
        \end{cases}
    \]
    which, as a function of $n$, increases on $(1,\infty)$. As $n \le 2p$ and $p \ge 3$, the function reaches its maximum $p^2/(2p-1)$ at $n = 2p$, and so $m/\rank V \le p^2/(2p-1)$. It is easy to see that equality holds if and only if $\vec{H}$ is the complete bipartite graph $K_{p,p}$ with uniform vertices.

    \bigskip
    \noindent\textit{Case 2:} $\vec{H}$ has parallel edges. Let $e_1$ and $e_2$ be a pair of parallel edges. Since $\vec{H}$ is connected, there is a spanning tree $T$ that contains $e_1$. One can check that the set of vectors $\dset{\vec{e} \in D_n}{e \in E(T) \cup \sset{e_2}}$ is linearly independent. Therefore
    \[
        \rank V = n.
    \]

    We next use the discharging method to bound the average degree $2m / n$ of $\vec{H}$. Each vertex $v$ starts with the charge $d(v)$, that is, the degree of $v$. We claim that for every vertex $v$ with $d(v) > p$, exactly one of the following two situations happens.
    \begin{enumerate}[label=(\Roman*)]
        \item The vertex $v$ is uniform, that is, the sign $\sigma(v, e)$ is the same for every edge $e$ incident to $v$. \label{item:p-1}
        \item The edges incident to $v$ are precisely two pairs $(e_1, e_1')$ and $(e_2, e_2')$ of parallel edges satisfying $\sigma(v, e_1) \neq \sigma(v, e_2)$, and in particular $d(v) = 4$ and $p = 3$. \label{item:p-2}
    \end{enumerate}
    Indeed, suppose that $d(v) > p$, and suppose that situation \ref{item:p-1} does not happen, that is, there exist two edges $e_1$ and $e_2$ incident to $v$ such that $\sigma(v, e_1) \neq \sigma(v, e_2)$. Because $d(v) > p \ge 3$, we may further assume that $e_1$ and $e_2$ intersect at $v$ by replacing $e_1$ or $e_2$ with another edge incident to $v$ in case $e_1$ and $e_2$ are parallel. Because $d(v) \ge 4$, according to \ref{item:p-b}, situation \ref{item:p-2} must happen.

    To describe the discharging rules, we say that two edges $e_1$ and $e_2$ are \emph{twin} if they are parallel and $c(e_1) = c(e_2)$, and for every vertex $v$ with $d(v) > p$, we call $v$ a \emph{type I} vertex when \ref{item:p-1} happens, and a \emph{type II} vertex when \ref{item:p-2} happens.
    
    \begin{itemize}
        \item The first discharging rule sends $1$ from every type I vertex $v$ to each of its neighbors $v'$ that are connected to $v$ through twin edges.
        \item The second discharging rule sends $1/2$ from every type II vertex $v$ to each of its two neighbors.
    \end{itemize}

    One can further deduce from \cref{def:proper-edge-coloring} the following property of twin edges.
    \begin{enumerate}[label=(\alph*), start=3]
        \item For twin edges $e_1$ and $e_2$ incident to a vertex $v$, if $\sigma(v, e_1) \neq \sigma(v, e_2)$, then $e_1$ and $e_2$ are the only two edges incident to $v$.\label{item:p-c}
    \end{enumerate}
    
    We first consider the charge of each vertex after the first discharging rule is applied. Suppose that $v$ is a type I vertex, and $v'$ is an arbitrary vertex that is connected to $v$ through a pair of parallel edges $e_1$ and $e_2$. Since $\sigma(v, e_1) = \sigma(v, e_2)$, it must be the case that $\sigma(v', e_1) \neq \sigma(v', e_2)$, and so $v'$ is not a type I vertex. Thus $v$ does not receive any charge from its neighbors. According to \ref{item:p-a}, among all the edges incident to $v$, only the parallel ones can receive identical colors. Thus after $v$ sends out charges to its neighbors, its charge decreases exactly to the number of distinct colors assigned to the edges incident to $v$, and so the charge of $v$ is at most $p$ now. Suppose in addition that $e_1$ and $e_2$ are twin edges, or in other words, $v'$ receives $1$ from $v$. According to \ref{item:p-c}, the only edges incident to $v'$ are $e_1$ and $e_2$, and so the charge of $v'$ is at most $3$, which is at most $p$, now.

    We then consider the charge of each vertex after the second discharging rule is applied. Notice that the second discharging rule only kicks in when $p = 3$. Suppose that $v$ is a type II vertex, and two pairs $(e_1, e_1')$ and $(e_2, e_2')$ of parallel edges connect $v$ respectively to $v_1$ and $v_2$ such that $\sigma(v, e_1) \neq \sigma(v, e_2)$. Without loss of generality, we may assume that $\sigma(v, e_1) = +1$ and $\sigma(v, e_2) = -1$. We have the four possibilities for $\sigma(v, e_1')$ and $\sigma(v, e_2')$. Once $\sigma(v, e_1')$ and $\sigma(v, e_2')$ are chosen, we can determine, according to \ref{item:p-a} and \ref{item:p-b}, which edges incident to $v$ receive identical or distinct colors. We summarize the four possibilities in the following figure. For each possibility, it is easy to check that $v$ did not receive any charge from $v_1$ or $v_2$ when the first discharging rule was applied. Therefore the final charge of $v$ is $3$, which is equal to $p$.

    \begin{figure}[ht]
        \centering
        \begin{tikzpicture}[
            very thick, baseline=(v.base),
            neg/.pic = {\draw[fill=white] (0,0) circle (0.1); \draw (-0.1,0)--(0.1,0);},
            pos/.pic = {\draw[fill=white] (0,0) circle (0.1); \draw (-0.1,0)--(0.1,0); \draw (0,0.1)--(0,-0.1);},
        ]
            \node[vertex] (a) at (-2,0) {};
            \node[vertex] (b) at (0,0) {};
            \node[vertex] (c) at (2,0) {};
            \draw (a) node[left] {$v_1$};
            \draw (0,-.1) node[below] {$v$};
            \draw (c) node[right] {$v_2$};
            \path[color=blue] (a) edge[bend left=45] (b);
            \path[color=blue] (b) edge[bend left=45] (c);
            \path[color=blue] (a) edge[bend left=-45] (b);
            \path[color=red] (b) edge[bend left=-45] (c);
            \draw (-0.4,0.3) pic {pos};
            \draw (0.4,0.3) pic {neg};
            \draw (-0.4,-0.3) pic {pos};
            \draw (0.4,-0.3) pic {pos};
        \end{tikzpicture}\quad%
        \begin{tikzpicture}[
            very thick, baseline=(v.base),
            neg/.pic = {\draw[fill=white] (0,0) circle (0.1); \draw (-0.1,0)--(0.1,0);},
            pos/.pic = {\draw[fill=white] (0,0) circle (0.1); \draw (-0.1,0)--(0.1,0); \draw (0,0.1)--(0,-0.1);},
        ]
            \node[vertex] (a) at (-2,0) {};
            \node[vertex] (b) at (0,0) {};
            \node[vertex] (c) at (2,0) {};
            \draw (a) node[left] {$v_1$};
            \draw (0,-.1) node[below] {$v$};
            \draw (c) node[right] {$v_2$};
            \path[color=blue] (a) edge[bend left=45] (b);
            \path[color=blue] (b) edge[bend left=45] (c);
            \path[color=red] (a) edge[bend left=-45] (b);
            \path[color=blue] (b) edge[bend left=-45] (c);
            \draw (-0.4,0.3) pic {pos};
            \draw (0.4,0.3) pic {neg};
            \draw (-0.4,-0.3) pic {neg};
            \draw (0.4,-0.3) pic {neg};
        \end{tikzpicture}\quad%
        \begin{tikzpicture}[
            very thick, baseline=(v.base),
            neg/.pic = {\draw[fill=white] (0,0) circle (0.1); \draw (-0.1,0)--(0.1,0);},
            pos/.pic = {\draw[fill=white] (0,0) circle (0.1); \draw (-0.1,0)--(0.1,0); \draw (0,0.1)--(0,-0.1);},
        ]
            \node[vertex] (a) at (-2,0) {};
            \node[vertex] (b) at (0,0) {};
            \node[vertex] (c) at (2,0) {};
            \draw (a) node[left] {$v_1$};
            \draw (0,-.1) node[below] {$v$};
            \draw (c) node[right] {$v_2$};
            \path[color=blue] (a) edge[bend left=45] (b);
            \path[color=blue] (b) edge[bend left=45] (c);
            \path[color=blue] (a) edge[bend left=-45] (b);
            \path[color=blue] (b) edge[bend left=-45] (c);
            \draw (-0.4,0.3) pic {pos};
            \draw (0.4,0.3) pic {neg};
            \draw (-0.4,-0.3) pic {pos};
            \draw (0.4,-0.3) pic {neg};
        \end{tikzpicture}\quad%
        \begin{tikzpicture}[
            very thick, baseline=(v.base),
            neg/.pic = {\draw[fill=white] (0,0) circle (0.1); \draw (-0.1,0)--(0.1,0);},
            pos/.pic = {\draw[fill=white] (0,0) circle (0.1); \draw (-0.1,0)--(0.1,0); \draw (0,0.1)--(0,-0.1);},
        ]
            \node[vertex] (a) at (-2,0) {};
            \node[vertex] (b) at (0,0) {};
            \node[vertex] (c) at (2,0) {};
            \draw (a) node[left] {$v_1$};
            \draw (0,-.1) node[below] {$v$};
            \draw (c) node[right] {$v_2$};
            \path[color=blue] (a) edge[bend left=45] (b);
            \path[color=blue] (b) edge[bend left=45] (c);
            \path[color=red] (a) edge[bend left=-45] (b);
            \path[color=red] (b) edge[bend left=-45] (c);
            \draw (-0.4,0.3) pic {pos};
            \draw (0.4,0.3) pic {neg};
            \draw (-0.4,-0.3) pic {neg};
            \draw (0.4,-0.3) pic {pos};
        \end{tikzpicture}
    \end{figure}
    
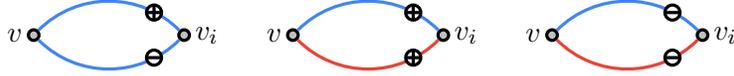
    We are left to decide the final charge of $v_1$ and $v_2$. Observe that for $i \in \sset{1,2}$ there are, \emph{au fond}, only three possible ways, as shown below, to color $e_i$ and $e_i'$ and to assign $\sigma(v_i, e_i)$ and $\sigma(v_i, e_i')$. For the first possibility, according to \ref{item:p-c}, the only edges incident to $v_i$ are $e_i$ and $e_i'$, and so the final charge of $v_i$ is $5/2$, which is less than $p$. For the other two possibilities, the contrapositive of \ref{item:p-b} implies that $v_i$ must be uniform, and so its charge was at most $p$ after the first discharging rule was applied. Assume for the sake of contradiction that $v_i$ is connected to another type II vertex $w$ through edges $f_i$ and $f_i'$. Since $v_i$ is uniform, in view of the three possibilities above, we know that $c(e_i) \neq c(e_i')$ and $c(f_i) \neq c(f_i')$. In view of \ref{item:p-a}, we know that $\sset{c(e_i), c(e_i')} \cap \sset{c(f_i), c(f_i')} = \varnothing$. Thus the four edges $e_i$, $e_i'$, $f_i$ and $f_i'$ receive distinct colors under $c$, which contradicts the assumption that $p = 3$. Therefore $v$ is the only type II neighbor of $v_i$, and the final charge of $v_i$ is at most $p + 1/2$.

    \begin{figure}[ht]
        \centering
        \begin{tikzpicture}[
            very thick, baseline=(v.base),
            neg/.pic = {\draw[fill=white] (0,0) circle (0.1); \draw (-0.1,0)--(0.1,0);},
            pos/.pic = {\draw[fill=white] (0,0) circle (0.1); \draw (-0.1,0)--(0.1,0); \draw (0,0.1)--(0,-0.1);},
        ]
            \node[vertex] (a) at (-2,0) {};
            \node[vertex] (b) at (0,0) {};
            \draw (a) node[left] {$v$};
            \draw (b) node[right] {$v_i$};
            \path[color=blue] (a) edge[bend left=45] (b);
            \path[color=blue] (a) edge[bend left=-45] (b);
            \draw (-0.4,0.3) pic {pos};
            \draw (-0.4,-0.3) pic {neg};
        \end{tikzpicture}\quad%
        \begin{tikzpicture}[
            very thick, baseline=(v.base),
            neg/.pic = {\draw[fill=white] (0,0) circle (0.1); \draw (-0.1,0)--(0.1,0);},
            pos/.pic = {\draw[fill=white] (0,0) circle (0.1); \draw (-0.1,0)--(0.1,0); \draw (0,0.1)--(0,-0.1);},
        ]
            \node[vertex] (a) at (-2,0) {};
            \node[vertex] (b) at (0,0) {};
            \draw (a) node[left] {$v$};
            \draw (b) node[right] {$v_i$};
            \path[color=blue] (a) edge[bend left=45] (b);
            \path[color=red] (a) edge[bend left=-45] (b);
            \draw (-0.4,0.3) pic {pos};
            \draw (-0.4,-0.3) pic {pos};
        \end{tikzpicture}\quad%
        \begin{tikzpicture}[
            very thick, baseline=(v.base),
            neg/.pic = {\draw[fill=white] (0,0) circle (0.1); \draw (-0.1,0)--(0.1,0);},
            pos/.pic = {\draw[fill=white] (0,0) circle (0.1); \draw (-0.1,0)--(0.1,0); \draw (0,0.1)--(0,-0.1);},
        ]
            \node[vertex] (a) at (-2,0) {};
            \node[vertex] (b) at (0,0) {};
            \draw (a) node[left] {$v$};
            \draw (b) node[right] {$v_i$};
            \path[color=blue] (a) edge[bend left=45] (b);
            \path[color=red] (a) edge[bend left=-45] (b);
            \draw (-0.4,0.3) pic {neg};
            \draw (-0.4,-0.3) pic {neg};
        \end{tikzpicture}
    \end{figure}

    Since the discharging rules preserve the total charge, the average degree $2m/n$ is at most the maximum final charge, which is at most $p + 1/2$. Finally, we obtain that
    \begin{equation*}
        \frac{m}{\rank V} = \frac{m}{n} \le \frac{p+1/2}{2} < \frac{p^2}{2p-1}. \qedhere
    \end{equation*}
\end{proof}

\begin{proof}[Proof of \cref{thm:kpl} for $\la = 2$] \label{proof:la-2}
    From \cref{thm:cgss-signed} and \cref{lem:d-infinity}, we know that $k_p(2)$ is the smaller of the two quantities $p^2/(p-1)^2$ and
    \[
        k_p^* := \inf\dset{\frac{\abs{G^\pm}}{\mult(-2, G^\pm)}}{G^\pm \text{ is represented by a subset of }E_8 \text{ and } \chi(-G^\pm) \le p}.
    \]
    Because there are only finitely many signed graphs that are represented by a subset of $E_8$, the infimum in the definition of $k_p^*$ is in fact a minimum, hence $k_p(2)$ is achievable. Moreover $k_p^* \ge k^*$, where
    \begin{equation} \label{eqn:k-star}
        k^* := \min\dset{\frac{\abs{G^\pm}}{\mult(-2, G^\pm)}}{G^\pm \text{ is represented by a subset of }E_8}.
    \end{equation}
    Since $k^* > 1$, there exists $p_0 \ge 3$ such that $p^2 / (p-1)^2 \le k^* \le k_p^*$ for every $p \ge p_0$.
\end{proof}

\section{Concluding remarks} \label{sec:discussion}

Extending Smith's classification \cite{S70} of connected graphs with spectral radius at most $2$, Cvetkovi\'c, Doob and Gutman~\cite[Theorem~3.8]{CDG82} characterized the connected graphs with spectral radius in $(2, \la')$, which were later explicitly classified by Brouwer and Neumaier~\cite{BN89}. Recall that $\la' = \sqrt{2 + \sqrt5} \approx 2.05817$. Retrospectively, perhaps both \cref{thm:forb1} and the fact that $\la'$ is the smallest $\la \in \R$ such that the set of graph spectral radii is dense in $(\la, \infty)$ (cf. \cite{H72,S89}) indicated that it is possible to classify connected graphs with spectral radius at most $\la'$.

A similar line of research was carried out for signed graphs. McKee and Smyth~\cite{MS07} determined all the connected signed graphs with spectral radius at most $2$ --- they are switching equivalent to subgraphs of $S_{14}, S_{16}$ and $T_{2n}$ for $n \ge 3$ in \cref{fig:ms}. When investigating Lehmer's Mahler measure problem, McKee and Smyth~\cite[Theorem~3]{MS12} further determined the $17$ connected signed graphs with spectral radius in $(2, 2.019)$. Belardo, Cioab\u{a}, Koolen and Wang~\cite[Problem~3.11]{BCKW18} raised the question on the classification of connected signed graphs with spectral radius at most $\la'$. This question was very recently resolved by 
Wang, Dong, Hou and Li~\cite{WDHL23}.

\begin{figure}
    \centering
    \begin{tikzpicture}[very thick, scale=1.6, baseline=(v.base)]
        \coordinate (v) at (0,0);
        \foreach \r in {38.57,90,141.43,192.86,244.29,295.71,347.14} {
            \draw[rotate=\r, black] (0:1) -- (154.29:1);
            \draw[rotate=\r, black] (0:1) -- (43.5:0.628);
            \draw[rotate=\r, black, dashed] (102.86:1) -- (43.5:0.628);
        }
        \foreach \r in {38.57,90,141.43,192.86,244.29,295.71,347.14} {
            \draw[rotate=\r] (0:1) node[vertex]{};
            \draw[rotate=\r] (43.5:0.628) node[vertex]{};
        }
        \node at (0,-1.3) {$S_{14}$};
    \end{tikzpicture}\qquad%
    \begin{tikzpicture}[very thick, scale=1.2, baseline=(u.base)]
        \coordinate (u) at (0,1.25);
        \coordinate (v) at (2,1);
  
        \draw[dashed] (0,0) -- ++(v);
        \draw (0,1) -- ++(v);
        \draw (1,0) -- ++(v);
        \draw (1,1) -- ++(v);
        \draw (0.5,0.5) -- ++(v);
        \draw (1.5,0.5) -- ++(v);
        \draw (0.5,1.5) -- ++(v);
        \draw[dashed] (1.5,1.5) -- ++(v);
  
        \draw (0,0) -- (1,0) -- (1,1);
        \draw (0,1) -- (0.5,1.5);
        \draw (1.5,0.5) -- (1.5,1.5);
        \draw (0,0) -- (0,1);
        \draw[dashed] (0,1) -- (1,1);
        \draw[dashed] (1,0) -- (1.5,0.5);
        \draw (1.5,0.5) -- (0.5,0.5) -- (0,0);
        \draw[dashed] (0.5,0.5) -- (0.5,1.5);
        \draw (0.5,1.5) -- (1.5,1.5) -- (1,1);
        \draw (0,0) node[vertex]{};
        \draw (1,0) node[vertex]{};
        \draw (0,1) node[vertex]{};
        \draw (1,1) node[vertex]{};
        \draw (0.5,0.5) node[vertex]{};
        \draw (1.5,0.5) node[vertex]{};
        \draw (0.5,1.5) node[vertex]{};
        \draw (1.5,1.5) node[vertex]{};
  
        \begin{scope}[shift={(v)}]
            \draw[dashed] (1,1) -- (1,0);
            \draw (1,0) -- (0,0) -- (0.5,0.5) -- (0.5,1.5) -- (1.5,1.5);
            \draw[dashed] (0.5,1.5) -- (0,1);
            \draw (0,0) -- (0,1) -- (1,1) -- (1.5,1.5) -- (1.5,0.5) -- (1,0);
            \draw[dashed] (1.5,0.5) -- (0.5,0.5);
            \draw (0,0) node[vertex]{};
            \draw (1,0) node[vertex]{};
            \draw (0,1) node[vertex]{};
            \draw (1,1) node[vertex]{};
            \draw (0.5,0.5) node[vertex]{};
            \draw (1.5,0.5) node[vertex]{};
            \draw (0.5,1.5) node[vertex]{};
            \draw (1.5,1.5) node[vertex]{};
        \end{scope}
        \node at (1.75,-.5) {$S_{16}$};
    \end{tikzpicture}\qquad%
    \begin{tikzpicture}[very thick, scale=1.6, baseline=(v.base)]
        \coordinate (v) at (0,0);
        \foreach \r in {30,60,90,120,150,180,210,240,270,300,330} {
            \draw[rotate=\r, black] (-15:1) -- (15:1);
            \draw[rotate=\r, black, dashed] (-15:0.75) -- (15:0.75);
            \draw[rotate=\r, black] (15:1) -- (-15:0.75);
            \draw[rotate=\r, black, dashed] (-15:1) -- (15:0.75);
            \draw[rotate=\r] (-15:1) node[vertex]{};
            \draw[rotate=\r] (-15:0.75) node[vertex]{};
        }
        \draw (-15:1) node[vertex]{};
        \draw (-15:0.75) node[vertex]{};
        \node at (.875,.06) {$\vdots$};
        \node at (0,-1.3) {$T_{2n}$ ($n \ge 3$)};
    \end{tikzpicture}
    \caption{Maximal connected signed graphs with spectral radius at most $2$ up to switching equivalence. The number of vertices in $T_{2n}$ is $2n$.} \label{fig:ms}
\end{figure}

With regard to smallest eigenvalues, we would like to extend \cref{thm:cgss} of Cameron et al.~\cite{CGSS76} beyond $\G(2)$.

\begin{problem} \label{prob:first}
    Classify all the connected graphs with smallest eigenvalue in $(-\la^*, -2)$. In particular, classify such graphs that have sufficiently many vertices.
\end{problem}

It is worth mentioning that Bussemaker and Neumaier~\cite[Theorem~2.5]{BN92} showed that $E_{2,6}$ defined in \cref{lem:e2n} is the only connected graph with smallest eigenvalue in $[-\la^1(E_{2,6}), -2)$, where $\la^1(E_{2,6}) \approx 2.00659$. Very recently, Acharya and Jiang \cite{AJ24} provided a complete solution to \cref{prob:first} --- there are 794 infinite families of graphs and 4752 exceptional graphs.

We can ask the same question for signed graphs, extending \cref{thm:cgss-signed} beyond $\G^\pm(2)$.

\begin{problem}
    Classify all the connected signed graphs with smallest eigenvalue in $(-\la^*, -2)$. In particular, classify such signed graphs that have sufficiently many vertices.
\end{problem}

Turning to spherical two-distance sets with two fixed angles, it is plausible to establish more instances of \cref{conj:main} by taking advantage of $\chi(G^\pm)$ in \cref{def:mph}. Denote by $\G^\mp_p(\la)$ the family of signed graphs $G^\pm$ with $\chi(G^\pm) \le p$ and $\la^1(G^\pm) \le \la$. Observe that $\G^\mp_p(\la)$, just like $\Gsp$, is still closed under taking subgraphs. We raise the following question, whose solution could possibly establish more instances of \cref{conj:main}.

\begin{problem} \label{prob:d}
    For every $p \in \N^+$, determine the set of $\la \in \R$ for which $\G^\mp_p(\la)$ has a finite forbidden subgraph characterization.
\end{problem}

Notice that $\G^\mp_1(\la)$ consists of unsigned graphs only, and $\G^\mp_2(\la)$ consists of signed graphs that are switching equivalent to unsigned graphs. Thus \cref{thm:forb1} essentially answers \cref{prob:d} when $p \in \sset{1,2}$.

Finally, we point out that the proof of \cref{thm:kpl} for $\la = 2$ actually shows that $k_p(2) = p^2/(p-1)^2$ for every $p \ge 1/(1-1/\sqrt{k^*})$, where $k^*$ is defined as in \eqref{eqn:k-star}. We compute $k^* = 15/14$ hence $1/(1-1/\sqrt{k^*}) \approx 29.49$. Indeed, Stani\'{c} noted in \cite{S20} that, up to switching equivalence, there is a unique maximal signed graph, denoted $M_8^\pm$, that is represented by a subset of $E_8$, and moreover, the spectrum of $M_8^\pm$ is $\sset{28^8, (-2)^{112}}$. By the Cauchy interlacing theorem, every signed graph $G^\pm$ that is represented by a subset of $E_8$ satisfies $\la^1(G^\pm) \le 28$, and thus
\[
    0 = \sum_i \la_i(G^\pm) \le (-2)\mult(-2,G^\pm) + 28(\abs{G^\pm} - \mult(-2,G^\pm)),
\]
which implies that $\abs{G^\pm}/\mult(-2,G^\pm) \ge 15/14$. Since equality holds when $G^\pm = M_8^\pm$, we obtain $k^* = 15/14$. We leave the determination of $k_p(2)$ for small $p$ as an open problem.

\begin{problem}
    For every $p \in \sset{3,4,\dots,29}$, determine the value of $k_p(2)$.
\end{problem}

\appendix

\section{Computer-assisted proofs} \label{sec:numerics}

Notice that the path extensions $(G_4, \sset{3^\pm}, \ell)$ and $(G_4, \sset{4^\pm}, \ell)$ are switching equivalent to $E_{2, \ell+3}$ (see \cref{fig:e2n}), and so \cref{lem:e2n} implies that equality holds for $F = G_4$ and $A \in \sset{\sset{3},\sset{4}}$ in \cref{lem:computation-1}, and for $F^\pm = G_4$ and $A \in \sset{\sset{3^\pm},\sset{4^\pm}}$ in \cref{lem:computation-2}.

The termination of a program that solves the following computational problem is a proof of the strict inequalities in \cref{lem:computation-1,lem:computation-2}. In fact, we strengthen these strict inequalities by replacing $\la^* \approx 2.0198$ with $101/50 = 2.02$. Note that $-101/50$ is not an algebraic integer, and hence it cannot be an eigenvalue of any signed graph.

\bigskip
\noindent\textit{Input.} The first line of the input gives the number \texttt{N} of minimal forbidden subgraphs for $\D_\infty^\pm$ (up to switching equivalence). Each of the \texttt{N} lines that follow represents a signed graph in \cref{fig:minimal-forb-1,fig:minimal-forb-2} by two strings. The first string is the label of the signed graph. The second string is of the form \texttt{u[1]u[2]\ldots u[2e-1]u[2e]} possibly followed by \texttt{-v[1]v[2]\ldots v[2f-1]v[2f]}, which lists the positive edges \texttt{u[1]u[2],\ldots ,u[2e-1]u[2e]} and the negative edges \texttt{v[1]v[2],\ldots ,v[2f-1]v[2f]}.

\begin{table}[t]
    {\ttfamily
    \begin{tabularx}{\textwidth}{XX}
        49 & G1 0312142334 \\
        G2 020412142334 & G3 02030412142334 \\
        G4 0102052345 & G5 0102030534 \\
        G6 0102030405 & G7 011215233445 \\
        G8 020304051234 & G9 010203040534 \\
        G10 01020512152345 & G11 01020305121545 \\
        G12 01020304052345 & G13 0102051215233445 \\
        G14 0312131415233445 & G15 0205121523253435 \\
        G16 0102030405233445 & G17 0205121523253545 \\
        G18 020304051215232534 & G19 010203051215232534 \\
        G20 010203040512152345 & G21 010203040512152325 \\
        G22 020512152324253545 & G23 02121314152324343545 \\
        G24 01020304051215233445 & G25 01030412131415233445 \\
        G26 01020304051213141534 & G27 0102030412131415233445 \\
        G28 0102030405121314152345 & G29 010203040512131415233445 \\
        G30 010203040512152324253545 & G31 01020304051215232425343545 \\
        S32 011213-23 & S33 02031213-23 \\
        S34 02031213-0123 & S35 0102031213-23 \\
        S36 0112152345-34 & S37 0304121523-01 \\
        S38 030412152345-10 & S39 030412152334-01 \\
        S40 010203052345-34 & S41 03041215233445-01 \\
        S42 01020512152345-34 & S43 01020405121523-34 \\
        S44 02030515232545-34 & S45 0203121415232545-34 \\
        S46 0203051215232545-34 & S47 010203051214152345-34 \\
        S48 010203051215232545-34 & S49 01020305121415232545-34 \\
    \end{tabularx}}
    \caption{Input.} \label{table:input}
\end{table}

\bigskip
\noindent\textit{Output.} For each of the \texttt{N} signed graphs $F^\pm$, output one line containing \texttt{x y z}, where \texttt{x} is the label of $F^\pm$, \texttt{y} and \texttt{z} are both \texttt{-1} if $\la_1(F^\pm) < -101/50$ already, otherwise \texttt{y} is, in all but one case, the minimum $\ell$ such that $\la_1(F^\pm, A^\pm, \ell) < -101/50$ for every nonempty signed vertex subset $A^\pm$ of $F^\pm$, and \texttt{z} is the minimum $m$ such that $\la_1(F^\pm, A^\pm, K_m) < -101/50$ for every nonempty signed vertex subset $A^\pm$ of $F^\pm$. In the exceptional case where the label of $F^\pm$ is \texttt{G4}, \texttt{y} is defined similarly but $A^\pm$ cannot be $\sset{3^\pm}$ or $\sset{4^\pm}$, and an extra $\texttt{*}$ is appended to \texttt{y}.

\begin{table}[t]
    {\ttfamily
    \begin{tabularx}{\textwidth}{XXXXXX}
        G1 -1 -1& 
        G2 -1 -1&
        G3 -1 -1&
        G4 8* 7&
        G5 -1 -1&
        G6 -1 -1\\
        G7 6 5&
        G8 11 7&
        G9 -1 -1&
        G10 7 6&
        G11 7 6&
        G12 6 6\\
        G13 4 4&
        G14 5 5&
        G15 -1 -1&
        G16 5 5&
        G17 -1 -1&
        G18 4 5\\
        G19 6 5&
        G20 4 5&
        G21 5 5&
        G22 -1 -1&
        G23 6 5&
        G24 4 4\\
        G25 4 4&
        G26 -1 -1&
        G27 4 4&
        G28 4 5&
        G29 4 4&
        G30 -1 -1\\
        G31 4 4&
        S32 -1 -1&
        S33 -1 -1&
        S34 -1 -1&
        S35 -1 -1&
        S36 -1 -1\\
        S37 6 5&
        S38 4 4&
        S39 6 5&
        S40 5 5&
        S41 4 4&
        S42 -1 -1\\
        S43 4 5&
        S44 5 5&
        S45 5 5&
        S46 4 4&
        S47 4 4&
        S48 4 4\\
        S49 4 4
    \end{tabularx}}
    \caption{Output.} \label{table:output}
\end{table}

\bigskip

Our implementation is straightforward. We iterate through the minimal forbidden subgraphs $F^\pm$ for $\D_\infty^\pm$. In each iteration, we compute \texttt{y} and \texttt{z} as follows. We first check whether $A_{F^\pm} + (101/50)I$ is positive definite --- if not we output \texttt{-1 -1} for \texttt{y z} and continue to the next iteration. We test whether a matrix is positive definite by checking whether its leading principal minors are all positive. If the smallest eigenvalue of $F^\pm$ turns out to be more than $-101/50$, we then go through all possible nonempty signed vertex subsets $A^\pm$ of $F^\pm$. For each $A^\pm$, we increase $\ell$ from $0$ until the determinant of $A_{(F^\pm, A^\pm, \ell)} + (101/50)I$ is negative, and we increase $m$ from $1$ until the determinant of $A_{(F^\pm, A^\pm, K_m)} + (101/50)I$ is negative. We record the largest $\ell$ and $m$ when we go through $A^\pm$, and output them as \texttt{y} and \texttt{z}. In the exceptional case where the label of $F^\pm$ is \texttt{G4}, we skip the computation of $\ell$ when $A^\pm$ is $\sset{3^\pm}$ or $\sset{4^\pm}$. We avoid floating-point errors by representing every number as a rational number.

The actual code, written in Ruby, is available as the ancillary file \texttt{maximum\_extensions.rb} in the arXiv version of this paper. As our code halts, we obtain a proof of \cref{lem:computation-1,lem:computation-2}. We provide the input in \cref{table:input} for the convenience of anyone who wants to program independently, and we provide in addition the output in \cref{table:output} for cross-check.

\section*{Acknowledgements}

We thank Sebastian Cioab\u{a} for sending us a copy of \cite{H77a}, and Jing Zhang for pointing us to \cite{D13}. We are very grateful to Zhuang Xiong and two anonymous referees for their valuable feedback on the earlier versions of the paper. One referee noted a gap in the proof of a byproduct result of \cref{thm:main2} on symmetric hollow integer matrices in an earlier version; that result was removed from the present manuscript and later proved in \cite{J25}.

\bibliographystyle{plain}
\bibliography{forbidden}

\end{document}